\documentclass{amsart}

\usepackage{a4}
\usepackage{graphicx}
\usepackage{verbatim}
\usepackage{amssymb} 
\usepackage{comment}
\usepackage{multirow}
\usepackage{dirtytalk}
\usepackage{imakeidx}
\usepackage{xcolor}
\usepackage{enumitem}

\usepackage[utf8]{inputenc}

\newcommand{\matrices}[2]{\mathrm{M}(#1\times #2,\mathbb{R})}

\newcommand{\altidentity}[1]{I_{n_{#1}}^{\pm}}

\newcommand{\boxstar}{%
  {\setlength{\fboxsep}{0.5pt}\fbox{\scriptsize \(*\)}}%
}

\newcommand{\boxstartwo}{%
  {\setlength{\fboxsep}{0.5pt}\fbox{\scriptsize \(2*\)}}%
}

\usepackage{multicol}
\usepackage{mathdots}

\newtheorem{theorem}{Theorem}[section]
\newtheorem*{theorem*}{Theorem}
\newtheorem{lemma}[theorem]{Lemma}

\newtheorem{example}[theorem]{Example}

\newtheorem{prop}[theorem]{Proposition}
\newtheorem{coro}[theorem]{Corollary}

\theoremstyle{definition}
\newtheorem{definition}[theorem]{Definition}

\theoremstyle{remark}
\newtheorem{remark}[theorem]{Remark}

\numberwithin{equation}{section}
\makeatletter
\newcommand{\leqnomode}{\tagsleft@true}
\makeatother

\newcommand{\abs}[1]{\lvert#1\rvert}

\makeindex
\begin{document}

\title[Moduli spaces of presymplectic forms]
{The moduli spaces of presymplectic forms on almost abelian Lie algebras}

\author{Luis Pedro Castellanos Moscoso}
\address{Osaka Central Advanced Mathematical Institute (OCAMI), Osaka Metropolitan University}
\email{caste3.1416@gmail.com}





\begin{abstract}
    We obtain necessary and sufficient conditions to determine the existence of presymplectic forms of a given rank on all almost abelian Lie algebras. We also study the moduli space of presymplectic forms (this is the set of all closed 2-forms of a given rank under a certain natural equivalence relation) on almost abelian Lie algebras. Most importantly, we show that for any almost abelian Lie algebra its moduli space of symplectic forms is finite. Moreover, we show that up to such natural equivalence all symplectic forms are permutations of a canonical 2-form. The important step in the proof is obtaining canonical representatives for a certain congruence of matrices, which is of some interest for matrix theory on its own. 
\end{abstract}

\maketitle



\section{Introduction}
A classical problem in differential geometry is to determine which geometric structures a given manifold admits and to describe their moduli up to some natural equivalence relation. In particular, this paper discusses left-invariant presymplectic forms in almost abelian Lie groups. This paper is a continuation of previous papers \cite{firstpaper} and  \cite{diagonal}. The results of this paper generalize the results in those papers. There is a small difference with the results in \cite{firstpaper}. In Proposition 5.6. of that paper it is shown that, up to the same natural equivalence relation considered later in this paper, the set of nondegenerate 2-forms for the couple of almost abelian Lie algebras considered there, is finite.  This is slightly stronger than Theorem \ref{main2} of this paper, when we restrict it to those two Lie algebras, because even though it also proves finiteness under the same equivalence relation, it requires the 2-forms to be closed. There are no caveats with respect to the results in the latter paper. 

In the period between the publication of those two previous papers and this one the study of geometric structures on almost abelian Lie algebras has been quite active, with several new results in different geometries. See for example \cite{metricsclassification1}, \cite{metricsclassification2}, \cite{symplecticlassification}, \cite{complexclassification}, \cite{khalerclassfication} and the references within these and the previous two articles mentioned earlier for other related recent and classic results. In particular, in \cite{symplecticlassification} the authors also study symplectic structures on almost abelian Lie algebras, they obtain a classification of the Lie algebras that admit symplectic structures up to isomorphism, this can be compared to our first main result (see below), but we generalize to presymplectic forms (see next section for the definition).  
The main results of this paper can be summarized in the following two informal statements. The first main theorem allows us to easily determine the existence of presymplectic forms of any rank on all almost abelian Lie algebras.
\begin{theorem*} [Main result 1]
    Let $\mathfrak{g}$ be an almost abelian Lie algebra.  There exists a presymplectic $2$-form of rank $R$ if \say{enough} eigenvalues of the matrix associated to $\mathfrak{g}$ come in positive-negative pairs. 
\end{theorem*}
The second main result is obtained when we study the moduli space of presymplectic forms (see next section for the precise definition) on almost abelian Lie algebras. For the symplectic case we obtain a particularly nice description: the moduli space of symplectic forms on any almost abelian Lie algebra is finite. This is our second main result. 
\begin{theorem*}  [Main result 2]
   Let $\mathfrak{g}$ be an almost abelian Lie algebra. Then the moduli space of symplectic forms on $\mathfrak{g}$  is finite.  Moreover, any symplectic form can be identified with a permutation of a canonical $2$-form. 
\end{theorem*}

The precise statements correspond to Theorem \ref{Main1} and Theorem \ref{main2}, respectively. The crucial step to obtain these results is to obtain canonical representatives for a matrix congruence relation in Proposition $\ref{canonicalformsreal}$ and  Proposition $\ref{canonicalformscomplex}$. The proofs draw on ideas from Lemma 4.1 in \cite{Horn}.

\section{Preliminaries and Setting}\label{preliminaries}
\subsection{Left-invariant presymplectic $2$-forms}Let $G$ be a simply connected Lie group with dimension $D$ and $\mathfrak{g}$ its corresponding Lie algebra. For a given $R\leq D$, we have the corresponding set of all left-invariant $2$-forms of rank $R$ on $G$, that we will denote by
\[\Omega^2_{R, \; \text{LI}}  \left(G\right):= \left\{ \omega(\cdot{,}\cdot) \in \ \Gamma \left( \bigwedge\nolimits^{\!2}  T^*G \right) \mid \mathrm{rank}\;\omega=R,\; \text{left-invariant}  \right\}.
\]
If for a given $\omega \in \Omega^2_{R, \; \text{LI}}(G)$, we also have that it is closed, that is $d\omega=0$, then the pair $(G,\omega)$ is called \textit{presymplectic Lie group of rank R} and $\omega$ is called a \textit{left-invariant presymplectic form of rank $R$}.
We have the following natural equivalence relation.
\begin{definition}\label{symplectomorphism}
Let $\omega_1,\omega_2\in\Omega^2_{R, \; \text{LI}}(G)$. $(G,\omega_1)$ and $(G,\omega_2)$ are said to be \textit{equivalent up to automorphism} (resp. \textit{equivalent up to automorphism and scale}) if there exists $\phi \in \mathrm{Aut}(G)$ such that $\phi^{*}\omega_1=\omega_2$ (resp. if there exist $\phi \in \mathrm{Aut}(G)$ and a constant $c\neq0$ such that $c\cdot(\phi)^{*}\omega_1=\omega_2$). 
\end{definition}
It is well known that the space $\Omega^2_{R, \; \text{LI}}$ can be identified with the space of $2$-forms of rank $R$ on $\mathfrak{g}$, denoted by  

\[
\Omega^2_R(\mathfrak{g})= \bigwedge\nolimits^{\!2}_R \mathfrak{g}^* :=\left\{ \omega\left(\cdot{,}\cdot\right) \in \bigwedge\nolimits^{\!2} \mathfrak{g}^* \mid \mathrm{rank}\;\omega = R  \right\}.  
\]
For this set, we have the following natural equivalence relation. 
\begin{definition}\label{symplectomorphismlie}
Let $\omega_1, \omega_2  \in \Omega^2_R(\mathfrak{g})$. $(\mathfrak{g}, \omega_1)$ and $(\mathfrak{g},\omega_2)$ are said to be \textit{equivalent up to automorphism} (resp. \textit{equivalent up to automorphism and scale}) if there exists $\phi \in \mathrm{Aut}(\mathfrak{g})$ such that $\phi^{*}\omega_1=\omega_2$ (resp. if there exist $\phi \in \mathrm{Aut}(\mathfrak{g})$ and a constant $c\neq0$ such that $c\cdot(\phi)^{*}\omega_1=\omega_2$). 
\end{definition}
As both notions in Definitions~\ref{symplectomorphism} and \ref{symplectomorphismlie} of equivalence coincide for simply connected Lie groups,  we focus from now on only on the Lie algebra. The following is well known. 

\begin{prop}[cf. \cite{SymplecticLieGroups}, Chapter 0]\label{closedcondition} 
Let $\omega_{\mathfrak{g}} \in \Omega^2_R(\mathfrak{g})$, and $\omega_G\in \Omega^2_{R, \; \text{LI}}  \left(G\right)$ be the corresponding $2$-form on the Lie group. Then $\omega_G$ is closed if and only if $\omega_{\mathfrak{g}}$ satisfies, for all $x,y,z \in \mathfrak{g}$
\[
d\omega_{\mathfrak{g}}(x,y,z):=\omega_{\mathfrak{g}} (x,[y,z])+ \omega_{\mathfrak{g}} (z,[x,y])+\omega_{\mathfrak{g}} (y,[z,x])=0. 
\]
\end{prop}
A $2$-form $\omega_{\mathfrak{g}} \in \Omega^2_R(\mathfrak{g})$ that satisfies the previous property is called a \textit{closed $2$-form of rank $R$}, or a \textit{presymplectic form of rank R} on the Lie algebra $\mathfrak{g}$. Then the pair $(\mathfrak{g},\omega_{\mathfrak{g}})$ is called \textit{presymplectic Lie algebra of rank R}. 
\begin{remark}
    A presymplectic vector space of rank $R$ is a pair $(V,\omega)$, where $V$ is a vector space and $\omega \in \bigwedge\nolimits^{\!2} V^*$ (that is $\omega$ is a $2$-form). For every $\omega_{\mathfrak{g}} \in \Omega^2_R(\mathfrak{g})$, the pair $(\mathfrak{g},\omega_{\mathfrak{g}})$ is a presymplectic vector space. So the notion for a Lie algebra is stronger, as it requires the $2$-form to be closed. 
\end{remark}

We have the set of presymplectic forms on a Lie algebra $\mathfrak{g}$ 
\[
\Omega^2_R(\mathfrak{g})\supset \Omega^2_{R,\,closed}(\mathfrak{g}):= \left\{ \omega \in \Omega^2_R(\mathfrak{g}) \mid d\omega=0 \right\}.
\]

We identify $\mathfrak{g}\cong \mathbb{R}^{D}$, we have then the canonical basis $\lbrace e_1,\ldots,e_{D}\rbrace$ and the corresponding dual basis $\lbrace \varepsilon^1,\ldots,\varepsilon^{D}\rbrace$. Then the canonical $2$-form of rank $R=2r$ is given by 

\[
\omega_0^R:=\sum_{i=1}^r \varepsilon^{i}\wedge  \varepsilon^{r+i}.
\]
If $R=D$ we omit the superscript and write $\omega_0$. The corresponding  matrix $J_{(D,R)}$ for the canonical $2$-form of rank $R=2r$ is given by
\begin{equation}\label{matricesoftwoform}
\begin{aligned}
J_{(D,R)}&:= \begin{pmatrix}
J_R &  0\\
0 & 0
\end{pmatrix}, \; \text{where} \;  J_R:= \begin{pmatrix}
0 & I_r \\
-I_r & 0
\end{pmatrix}. 
    \end{aligned}
\end{equation}
$I_r$ is the identity matrix. Denote by $\mathrm{Sp}(D,R)$  the group of linear maps which preserve the canonical $2$-form of rank $R$. For $R=D$,  this is the usual symplectic group and, with $D=2d$, is usually denoted by $\mathrm{Sp}(2d)$. The symplectic group can be described by 

\begin{equation}\label{symplecticmatrices}
 \mathrm{Sp}(2d):=\left\{ A \in \mathrm{GL}(D,\mathbb{R}) \mid (A^t)J_DA=J_D \right\}.
\end{equation}
If $R<D$, then the group $\mathrm{Sp}(D,R)$ is less known but can be derived easily. It can be described by 
\begin{equation} 
    \begin{aligned}
  \mathrm{Sp}(D,R)&:= \left\{ \begin{pmatrix}
S & 0 \\
M & G
\end{pmatrix} \in \mathrm{GL}(D,\mathbb{R})\;\middle|\; \begin{aligned}
 &S\in \mathrm{Sp}(2r) \\ 
&M \in \matrices{D-R}{R}\\
&G \in \mathrm{GL}(D-R,\mathbb{R})
\end{aligned}
               \right\}.  \\
                  \end{aligned} 
\end{equation} 
With the basis fixed we also have the usual identification 
\begin{equation}\label{identificationwithskewsymmetric}
\left\{ B \in \mathrm{Skew}(D\times D,\mathbb{R}) \;\middle|\; \begin{aligned} \mathrm{Rank}(B)=R. 
\end{aligned}
 \right\} \xrightarrow{\text{bij}}  \Omega^2_{R}(\mathfrak{g}), \quad B \mapsto \omega 
\end{equation}
where $\omega(x,y):=x^tBy$. 

The general linear group $\mathrm{GL}(D,\mathbb{R})$ acts transitively on $\Omega^2_R(\mathfrak{g})$ by 
\[
g.\omega(\cdot{,}\cdot)=\omega(g^{-1}(\cdot) {,} g^{-1}(\cdot)) \; \; \; \forall g\in \mathrm{GL}(D,\mathbb{R}).
\] 
From the theory of homogeneous spaces we have the identification
\[
\Omega^2_R(\mathfrak{g}) \cong \mathrm{GL}(D,\mathbb{R}) / \mathrm{Sp}(D,R).
\]

Define $\mathbb{R}^{\times}:=\mathbb{R} \setminus \{0\}$. Then we can consider the set 
\[
\mathbb{R}^{\times}\mathrm{Aut}(\mathfrak{g}):=\left\{ c\phi \in \mathrm{GL}(D,\mathbb{R}) \mid \phi \in \mathrm{Aut}(\mathfrak{g}), c \in \mathbb{R}^{\times} \right\},
\]
which is a subgroup of $\mathrm{GL}(D,\mathbb{R})$. Hence it naturally acts on $\Omega^2_R(\mathfrak{g})$. We can then consider the orbit space of this action. The orbit space of the action of $\mathbb{R}^{\times}\mathrm{Aut}(\mathfrak{g})$ on $\Omega^2_R(\mathfrak{g}) $ we will call the \textit{moduli space of $2$-forms of rank R} and is denoted by 

\begin{equation}\label{definitionofmodulispace} 
\begin{aligned} 
\mathfrak{P}\Omega^2_R(\mathfrak{g}):&=\mathbb{R}^{\times}\mathrm{Aut}(\mathfrak{g})\setminus \Omega^2_R(\mathfrak{g}):=\left\{ \mathbb{R}^{\times}\mathrm{Aut}(\mathfrak{g}).\omega \mid \omega \in \Omega^2_R(\mathfrak{g}) \right\} \\
&\cong  \left\{ B \in \mathrm{Skew}(D\times D,\mathbb{R}) \;\middle|\; \begin{aligned} \mathrm{Rank}(B)=R 
\end{aligned}
 \right\} /_{\mathrm{cong}}  \mathbb{R}^{\times}\mathrm{Aut}(\mathfrak{g}).
\end{aligned}
\end{equation}
This last set denotes the congruence classes over the set $\mathbb{R}^{\times}\mathrm{Aut}(\mathfrak{g})$.  For $\omega \in \Omega^2_R(\mathfrak{g})$, $[\omega]$ will denote the corresponding orbit of $\mathfrak{P}\Omega^2_R(\mathfrak{g})$. One can easily see that, if $\omega_1, \omega_2 \in \Omega^2_R(\mathfrak{g})$ are in the same $\mathbb{R}^{\times}\mathrm{Aut}(\mathfrak{g})$-orbit, then they are equivalent up to automorphism and scale. Therefore there is a surjection from the moduli space $\mathfrak{P}\Omega^2_R(\mathfrak{g})$ onto the quotient space
\[
 \Omega^2_R(\mathfrak{g}) / \text{``up to automorphism and scale''}.
\]
This correspondence is not necessarily bijective.  The action of $\mathbb{R}^{\times}\mathrm{Aut}(\mathfrak{g})$ also preserves the closedness  of 2-forms.  We then also have the subset

\begin{align}
    \mathfrak{P}\Omega^2_R(\mathfrak{g})\supset  \mathfrak{P}\Omega^2_{R,\,closed}(\mathfrak{g})&:=\mathbb{R}^{\times}\mathrm{Aut}(\mathfrak{g})\setminus \Omega^2_{R,\,closed} \\
    &=\left\{ \mathbb{R}^{\times}\mathrm{Aut}(\mathfrak{g}).\omega \mid \omega \in \Omega^2_{R,\,closed} \right\}.
\end{align}

This set will be called \textit{moduli space of presymplectic forms of rank R}. In the following sections we study $\mathfrak{P}\Omega^2_{R,\,closed}(\mathfrak{g})$ for a particular family of Lie algebras. In the next section we will define this family and state some of its properties. 
\section{Almost abelian Lie algebras}
\begin{definition}
A non-abelian Lie algebra $\mathfrak{g}$ is called almost abelian if it contains a codimension 1 abelian subalgebra.
\end{definition}
The following facts and a lot more about the structure of almost abelian Lie algebras can be mostly found in \cite{almostabelian}. 
\begin{prop}
An almost abelian Lie algebra $\mathfrak{g}$ has a codimension 1 abelian ideal $L$, 
and is therefore isomorphic to the semidirect product
\[
\mathfrak{g}=\mathbb{R}e \ltimes L,
\]
for some $e\in \mathfrak{g}$.
\end{prop}

Next we need the description of the automorphism group $\mathrm{Aut}(\mathfrak{g})$. In Section 4 of \cite{almostabelian} the author describes $\mathrm{Aut}(\mathfrak{g})$ for almost abelian Lie algebras in terms of a decomposition of $\mathfrak{g}$ as a direct sum indecomposable Lie algebra and an abelian Lie algebra (see Proposition 7). For our purposes it will be more convenient to describe $\mathrm{Aut}(\mathfrak{g})$ adapted to the decomposition $\mathfrak{g}=\mathbb{R}e \ltimes L$, so we present the next proof. 

\begin{prop}\label{automorphismgroupalmostabelian} Let $\mathfrak{g}=\mathbb{R}e \ltimes L$ be an almost abelian Lie algebra. Write $\mathrm{ad}_{e}$ as
\begin{equation}\label{Adeinblock}
\mathrm{ad}_{e}= \begin{pmatrix}
0 & 0 \\
0 & A
\end{pmatrix}.
\end{equation}
If $\mathrm{rank}(A)\geq 2$, or  $\mathrm{rank}(A)=1$ and $A$ not nilpotent $(A^2\neq 0)$, then

\[
\mathrm{Aut}(\mathfrak{g})=
    \left\{ \phi=\begin{pmatrix}
\alpha & 0 \\
\gamma & \Delta
\end{pmatrix} \in \mathrm{GL}(D,\mathbb{R})\;\middle|\; \begin{aligned}
 &\gamma \in \mathbb{R}^{D-1}\\ 
&\alpha \in \mathbb{R}^{\times}\\
&\Delta A-\alpha A \Delta=0 
\end{aligned}
               \right\}.
\]
If $\mathrm{rank}(A) =1$ and $A$ is nilpotent $(A^2=0)$ then  

\[
\mathrm{Aut}(\mathfrak{g})=A_1\cup A_2,
\]
where 
\begin{align*}
    A_1&:=  \left\{ \begin{pmatrix}
\alpha & 0 \\
\gamma & \Delta
\end{pmatrix} \in \mathrm{GL}(D,\mathbb{R})\;\middle|\; \begin{aligned}
 &\gamma \in \mathbb{R}^{D-1}\\ 
&\alpha \in \mathbb{R}^{\times}\\
&\Delta A-\alpha A \Delta=0 
\end{aligned}
               \right\} \\
 A_2&:= \left\{ \begin{pmatrix}
\alpha & \beta^t \\
\gamma & \Delta
\end{pmatrix} \in \mathrm{GL}(D,\mathbb{R})\;\middle|\; \begin{aligned}
&0\neq\beta^t \in \mathbb{R}^{D-1} \\
 &\gamma \in \mathbb{R}^{D-1}\\
 &\mathrm{Im} (A)\subset \ker (\beta^t) \\
& A\Delta=u\otimes\beta^t \;
\text{ for some }u\in\mathrm{Im}A \\
&\Delta A=\alpha A\Delta-(A\gamma)\otimes\beta^t 
\end{aligned}
               \right\}              
\end{align*}

\end{prop}
\begin{proof}
    Consider a matrix 
    \[
  \phi= \begin{pmatrix}
\alpha & \beta^t \\
\gamma & \Delta 
\end{pmatrix} \in \mathrm{GL}(D,\mathbb{R}).
    \]
    Direct calculation from the automorphism condition 
   $\phi[\cdot,\cdot]=[\phi(\cdot),\phi(\cdot)]$  gives the following 3 conditions

    \begin{align}
        \beta^tA x&=0,  \label{1ofauto}\\
       \Delta A x-\alpha A \Delta x+(\beta^tx)A\gamma&=0, \label{2ofauto} \\
       (\beta^tx)A\Delta y-(\beta^ty)A \Delta x &=0 ,  \label{3ofauto}
    \end{align}
    for all $x,y \in L $.  First Let us prove that $\beta^t\neq 0$ implies that $\mathrm{rank} (A)=1$.  $\det (\phi)\neq 0$ gives us that $\ker \Delta\cap \ker \beta^t = \lbrace 0\rbrace$ (as maps acting on $L$), so that $\Delta \mid_{\ker \beta^t}$ is injective. As $\dim \ker \beta^t= \dim (L)-1$, we then get 
    \begin{equation}\label{equation1auto}
        \dim \Delta \ker \beta^t= \dim (L)-1 
    \end{equation}
   Now from $(\ref{3ofauto})$ we get, by choosing  $y \in \ker \beta^t$, that 
   
    \begin{equation}\label{equation2auto}
      \Delta \ker \beta^t \subset \ker A. 
    \end{equation}
    Combining $\eqref{equation1auto}$ and $\eqref{equation2auto}$ we get $\dim \ker A \geq \dim L-1 $, and this implies $\mathrm{rank}(A)=1$ (remember we exclude the case $\mathrm{rank} (A)  = 0$ from the definition of almost abelian Lie algebra). 
    
    Now consider the case $\mathrm{rank} (A)=1$. There exists $\rho \in L^{*}$ and $0\neq z \in L$ such that $Ax=\rho(x)z$ for all $x \in L$. In particular we have $A^2=\rho (z)A$. We now show that if $A$ is not nilpotent ($A^2\neq0$, that is $\rho(z)\neq0$), then again we must have, as before, that $\beta^t=0$.  Suppose instead that $\beta^t\neq 0$. Evaluate the previous 3 conditions in $z$. From \eqref{1ofauto} we get $\beta^t z=0$. Then from \eqref{3ofauto}, we get $\Delta z \in \ker A$ . This and \eqref{2ofauto} gives us $\Delta z=0$. At last we also have $\phi(z)=(\beta^t z, \Delta z)=0$, but this is not possible as $\phi$ is nonsingular, so indeed we must have $\beta^t =0$.  If $A$ is nilpotent ($A^2=0$ or equivalently $\rho(z)=0$),  we cannot conclude as before that $\Delta z = 0$, so we can indeed have $\beta^t \neq 0$ . For this case we can rewrite the first two conditions to make them a little more clear. We have 

\begin{align*}
\eqref{1ofauto} \iff & \mathrm{Im} (A)\subset \ker (\beta^t), \\
\eqref{3ofauto} \iff&  A\Delta  = u \otimes \beta^t , \text{ for some} \; u \in \mathrm{Im}(A).    
\end{align*}
\end{proof}

\begin{remark}
   The second case in Proposition \ref{automorphismgroupalmostabelian}: $\mathrm{rank}(A)=1$ and $A$ nilpotent is precisely the case $\mathfrak{g}=\mathfrak{H_3}\oplus \mathbb{R}^{D-3}$. That is $\mathfrak{g}$ is the direct sum of the 3-dimensional Heisenberg Lie algebra and an abelian Lie algebra. The fact that in this particular case the automorphism group is different from the rest of cases is also shown in \cite{almostabelian} (Proposition 8 and Proposition 9) and its explicit form is well known.
\end{remark}
In the following chapters we will only use the following.  
  \begin{coro} 
      Let $\mathfrak{g}=\mathbb{R}e \ltimes L$ be an almost abelian Lie algebra. For $\mathrm{ad_e}$ written as in \eqref{Adeinblock} we have 
      \begin{equation} 
        \mathbb{R}^{\times}\mathrm{Aut}(\mathfrak{g}) \supset  \left\{ \begin{pmatrix}
\alpha & 0 \\
\gamma & \Delta
\end{pmatrix} \in \mathrm{GL}(D,\mathbb{R})\;\middle|\; \begin{aligned}
 &\gamma \in \mathbb{R}^{D-1}\\ 
&\alpha \in \mathbb{R}^{\times}\\
&\Delta A-A\Delta=0 
\end{aligned}
               \right\}.
    \end{equation}
  \end{coro}
\begin{proof}
    The proof is quite straightforward, so we omit it.
\end{proof}
The Lie algebra structure of an almost abelian Lie algebra $\mathfrak{g}=\mathbb{R}e \ltimes L$ is completely determined by $\mathrm{ad}_{e}$:
\[
\left[e,v\right] = \mathrm{ad}_{e}v, \; v\in L.
\]

Therefore, any pair $(V,T)$, where $V$ is a vector space and $T$ a nonzero linear map, determines an almost abelian Lie algebra $\mathfrak{g}$ and vice versa, but different maps can yield isomorphic Lie algebras. Two pairs $(V_1,T_1)$ and $(V_2,T_2)$ are said to be \emph{similar} if there exists an invertible map $\phi$ such that $\phi T_1 \phi^{-1}=T_2$. In this case we write $(V_1,T_1)\sim (V_2,T_2)$. Isomorphism classes of  almost abelian Lie algebras correspond to the similarity classes of linear operators on vector
spaces up to scaling.
\begin{theorem}
Two almost abelian Lie algebras $\mathfrak{g}=\mathbb{R}e \ltimes L$ and $\mathfrak{g}^{\prime}=\mathbb{R}e^{\prime} \ltimes L^{\prime}$ are isomorphic if and only if $(L,\mathrm{ad_{e}}|_L)\sim (L^{\prime},\lambda \mathrm{ad_{e^\prime}}|_{L^{\prime}})$ for some constant $\lambda \neq 0$.
\end{theorem}
This theorem allows us to consider, without loss of generality,  almost abelian Lie algebras  $\mathfrak{g}$ such that the map $ad_{e}$ is given in real Jordan normal form (see \cite{matrixanalysis}, Theorem 3.4.1.5.). So from now on we suppose that $\mathrm{ad}_e|_L=\mathcal{J}_N$, $\mathcal{J}_N$ denotes a matrix in real Jordan normal form.

\subsection{The sets $\Omega^2_{R,\,closed}$ and $\mathfrak{P}\Omega^2_{R,\,closed}$ for almost abelian Lie algebras.}
In this section we first rewrite the closedness condition of Proposition \ref{closedcondition} for almost abelian Lie algebras. Then we show that the study of the moduli space $\mathfrak{P}\Omega^2_{R,\, closed}(\mathfrak{g})$ can be reduced to a tractable matrix problem.  As before we consider an almost abelian Lie algebra
\[
\mathfrak{g}=\mathbb{R}e_1 \ltimes \mathrm{Span}(e_2,\ldots,e_D).
\]
$\dim(\mathfrak{g})=D=N+1$. Without loss of generality we suppose that $\mathrm{ad}_{e_1}|_L=\mathcal{J}_N$, where $\mathcal{J}_N$ is a matrix in real Jordan normal form. Recall also that $R=2r$. For the following calculations it will be useful to rewrite the matrices in $(\ref{matricesoftwoform})$ as follows 

\begin{equation}
    \begin{aligned}
       J_{(D,R)}=& \left( \begin{array}{c|c} 
0 & u^t_r \\ \hline
-u_r & (J_{(D,R)})_{1,1} \end{array} \right)
    \end{aligned}
  \end{equation}  

where $u_r\in \mathbb{R}^{D-1}$ is a vector whose $r$-th component is 1 and the rest 0. \newline $(J_{(D,R)})_{1,1} \in \matrices{D-1}{D-1}$ is the principal minor of $(J_{(D,R)})$ obtained after removing the first row and column. Notice that $\mathrm{Rank}((J_{(D,R)})_{1,1})=R-2$. In particular, we will need the following subset of the group of permutation matrices of dimension $D$: 

\[
\mathrm{Per}(D,\mathbb{R})_{e_1}:=\left\{ \left( \begin{array}{c|c} 
1 & 0 \\ \hline
0 & P 
\end{array} \right) \mid P\in \mathrm{Per}(D-1,\mathbb{R})  \right\},
\]

Here $\mathrm{Per}(D-1,\mathbb{R})$ is the group of permutation matrices of dimension $D-1$. We can think of the first set as the stabilizer of $e_1$. We use the following notation for a member of $\mathrm{Per}(D,\mathbb{R})_{e_1}$
\begin{align*}
  \bar{P}&=\left( \begin{array}{c|c} 
1 & 0 \\ \hline
0 & P 
\end{array} \right) \in \mathrm{Per}(D,\mathbb{R})_{e_1}.
\end{align*}
Using a similar notation, we define
\[
\bar{\mathcal{J}}_{N}:=\left( \begin{array}{c|c} 
0 & 0 \\ \hline
0 & \mathcal{J}_{N} \end{array} \right).
\]
Once again, using a similar notation, suppose that we have $\bar{B} \in \mathrm{Skew}(D\times D,\mathbb{R})$, $\mathrm{Rank}(\bar{B})=R$. Write this matrix as 
\begin{equation}\label{skewsymmetricmatrix}
    \bar{B}=\left( \begin{array}{c|c} 
0 & b^t \\ \hline
-b & B
\end{array} \right), \quad B\in \mathrm{Skew}(D-1\times D-1,\mathbb{R}).
\end{equation}
Notice that  $R-2\leq \mathrm{Rank} (B) \leq R$. Therefore, if $R=D$, then $\mathrm{Rank} (B)=D-2$. If $R\neq D$, then  we can have either $\mathrm{Rank} (B)=R-2$ or $\mathrm{Rank} (B)=R$. 
Next we consider the closedness condition in Proposition \ref{closedcondition} for almost abelian Lie algebras (see also Corollary 5.5 in \cite{symplecticlassification}).
\begin{prop}
    Let $\omega \in \Omega^2_R(\mathfrak{g})$ correspond to a skew-symmetric matrix $\bar{B}$ as in $(\ref{skewsymmetricmatrix})$. Then $d\omega=0$ if and only if 
    \begin{align}\label{conditionclosedmatrixsymplecticcase}
B\mathcal{J}_N+(\mathcal{J}_N)^tB=0.
 \end{align}
 \end{prop}
 \begin{proof}
     From Proposition \ref{closedcondition}, $d\omega=0$ if and only if for all $i,j\neq 1$ 
\begin{align*}
    0&=d\omega(e_1,e_i,e_j)\\
    &=\omega(e_j,[ e_1,e_i ])+\omega(e_i,[ e_j,e_1 ] )\\
    &=\omega(e_j,[ e_1,e_i])+\omega([e_1,e_j],e_i)\\
   &=\omega(e_j,\bar{\mathcal{J}}_Ne_i)+\omega(\bar{\mathcal{J}}_Ne_j,e_i).
    \end{align*}
    We can rewrite this last equation as 
\[
e_j^t\bar{B}\bar{\mathcal{J}}_Ne_i+e_j^t\bar{\mathcal{J}}_N^t\bar{B}e_i=0, \;\;  \text{for all}\; i,j\neq 1.
\]
Therefore the condition becomes 
\begin{align}\label{closedconditionmatricessymplecticcase}
 B\mathcal{J}_N+({\mathcal{J}}_N)^tB=0.   
 \end{align}
    \end{proof}
Now we can state our first necessary and sufficient condition for the existence of $2$-forms of a given rank. First define the following set 
\[
\mathcal{ H}_{\mathcal{J}_N}^{K}:=\left\{ B \in \mathrm{Skew}(N\times N,\mathbb{R}) \;\middle|\; \begin{aligned}
B\mathcal{J}_N+\mathcal{J}_N^tB=0,  \\ \mathrm{Rank}(B)=K 
\end{aligned}
 \right\}.
\]

We state a convention: throughout the paper, whenever a rank-indexed set (like $\mathcal{ H}_{\mathcal{J}_N}^{K}$ and many that will appear later) is assigned a
negative rank, it is understood to be empty.
\begin{prop}\label{conditionsforexistenceclosedform}
We have the following:
\begin{enumerate}[label=\emph{(\alph*)}, leftmargin=*, itemsep=0.3em] 
    \item If $R=D$, then $\Omega^2_{D,\,\text{closed}}(\mathfrak{g})\neq \emptyset$ if and only if $\mathcal{ H}_{\mathcal{J}_N}^{D-2}\neq \emptyset$.
    \item  If $R\neq D$, then $\Omega^2_{R,\,\text{closed}}(\mathfrak{g})\neq \emptyset$ if and only if $\mathcal{H}_{\mathcal{J}_N}^{R} \cup \mathcal{ H}_{\mathcal{J}_N}^{R-2} \neq \emptyset$.
\end{enumerate}
    
\end{prop}
\begin{proof}
    Let $\omega \in \Omega^2_{R,\,\text{closed}}(\mathfrak{g})$. This defines a skew-symmetric matrix $\bar{B}$ as in $(\ref{skewsymmetricmatrix})$. The block $B$ must satisfy equation $(\ref{conditionclosedmatrixsymplecticcase})$ and the possible ranks of $B$ are exactly as in the statement of this proposition. That is if $R=D$, $B \in \mathcal{H}_{\mathcal{J}_N}^{D-2}$. If $R<D$, $B \in \mathcal{H}_{\mathcal{J}_N}^{R} \cup \mathcal{ H}_{\mathcal{J}_N}^{R-2}$. 
    Now suppose there exists $B\in \mathcal{H}_{\mathcal{J}_N}^{R} \cup \mathcal{ H}
    _{\mathcal{J}_N}^{R-2} $. Define the matrix 
   \[
\bar{B}=\left( \begin{array}{c|c} 
0 & v^t \\ \hline
-v & B 
\end{array} \right),  \quad B\in \mathrm{Skew}(D-1\times D-1,\mathbb{R}). 
\] 
by choosing $v$ as in the following table
\[
\begin{array}{c|c|c}
\text{Case}
& \operatorname{rank}B & \text{Choice of }v\\ \hline
R=D & D-2 & v\notin\operatorname{Im}B \\
R<D,\; B\in\mathcal H_{\mathcal J_N}^{R} & R & v=0 \\
R<D,\; B\in\mathcal H_{\mathcal J_N}^{R-2} & R-2 & v\notin\operatorname{Im}B
\end{array}
\]

We have then for all cases that $\mathrm{rank}(\bar{B})=R$. The resulting matrix defines an $\omega \in \Omega^2_{R,\,\text{closed}}(\mathfrak{g})$. 
\end{proof}
The exact shape of the set $\mathcal{H}_{\mathcal{J}_N}^{K}$ is described in the next section and the exact conditions for $\mathcal{J}_{N}$ so that this set is nonzero are stated in Theorem \ref{Main1}. 

\subsubsection{The moduli space for almost abelian Lie algebras.}
We study the quotient space in $(\ref{definitionofmodulispace})$. We will consider only the subgroup of $\mathbb{R}^{\times}\mathrm{Aut}(\mathfrak{g})$  described by the following matrices $\bar{A}  \in \mathbb{R}^{\times}\mathrm{Aut}(\mathfrak{g})$, written as 
\[\bar{A} =   \begin{pmatrix}
\alpha & 0 \\
v & A
\end{pmatrix},
\]
with  
\[
A \in \mathcal{T}_{\mathcal{J}_{N}}:= \left\{ A \in \mathrm{GL}(N,\mathbb{R})\;\middle|\;  A\mathcal{J}_{N}-\mathcal{J}_{N}A=0
               \right\}.
\]
Let $\bar{B}$ be as in Equation  $(\ref{skewsymmetricmatrix})$. Consider the product 
\begin{align}\label{productformoduli}
      \bar{A}^t\bar{B}\bar{A}&=\left( \begin{array}{c|c} 
0 & (\alpha b^t+v^tB)A \\ \hline
A^t(Bv-b\alpha) & A^tBA
\end{array} \right) 
 \end{align}

We have two cases depending on the possible ranks of $B$: $R$ or  $R-2$, these correspond to $b\in \mathrm{Im}(B)$ or  $b\notin \mathrm{Im}(B)$ respectively. When $R=D$ only the latter case is possible. By varying $v$ the term $Bv=\text{Im}\; B$. In the former case the term $(A^t(Bv-b\alpha))$ can always be made 0. So we get the set 
\[
E_{1,R}:=\left\{ B \in \mathrm{Skew}(N\times N,\mathbb{R}) \;\middle|\; \begin{aligned}
 \mathrm{Rank}(B)=R 
\end{aligned}
 \right\}/_{\mathrm{cong}} \mathcal{T}_{\mathcal{J}_{N}}.
\]
In the latter case we need to consider also the term  $A^t(Bv-b\alpha)$ so we get the set

\begin{align*}
E_{2,R}&:=    
   \left\{(B,[b]) \;\middle|\; \begin{aligned}
B \in\mathrm{Skew}(N\times N,\mathbb{R}),  \\
\mathrm{Rank}(B)=R-2, \\
0\neq[b] \in \mathbb{R}^{N}/\mathrm{Im}(B). 
\end{aligned}
 \right\}\Big/ \sim \\
 \end{align*}

where $(B,[b])\sim(B^{\prime}, [b^{\prime}])$ if and only if there exist $\alpha \in \mathbb{R}^{\times}$ and $A \in  \mathcal{T}_{\mathcal{J}_{N}}$ such that 
\[
B^{\prime}=A^tBA  \; \;  \text{and}  \;\; \  [b^{\prime}]=\alpha A^t[b]. 
\]

Therefore we can write the surjections
\begin{align}
E_{2,D} &\twoheadrightarrow\mathfrak{P}\Omega^2_D(\mathfrak{g})  \label{modulifordimensionD} \\
E_{1,R} \,\cup \, E_{2,R}&\twoheadrightarrow\mathfrak{P}\Omega^2_R(\mathfrak{g}), \;  \;  R\neq D.      
\end{align}
 
 For $\mathfrak{P}\Omega^2_D(\mathfrak{g})$ we in fact have the following.

\begin{prop}\label{bijjectiontomatricesforsymplecticcase}
    We have the surjection 
\[
 \left\{ B \in \mathrm{Skew}(N\times N,\mathbb{R}) \;\middle|\; \begin{aligned}
 \mathrm{Rank}(B)=N-1 
\end{aligned}
 \right\}/_{\mathrm{cong}}  \mathcal{T}_{\mathcal{J}_{N}} \twoheadrightarrow  \mathfrak{P}\Omega^2_{D}(\mathfrak{g}).
\]
    \begin{proof}
        We refer to the Equation $(\ref{modulifordimensionD})$. We have $\mathrm{rank}(B)=D-2$ and  $0\neq w \notin \mathrm{Im}(B)$. From the dimensions of this case $\dim(\mathbb{R}^{N}/\mathrm{Im}(B))=1$, so that there is only one equivalence class for the second element of the set $E_{2,D}$.  
    \end{proof}
\end{prop}
 \begin{coro} \label{corollaryconditionsforpermutation}
     Let $[\omega] \in \mathfrak{P}\Omega^2_{D}$.  Let $\omega$ correspond to a skew-symmetric matrix $\bar{B}$ as in $(\ref{skewsymmetricmatrix})$ of rank $D$. Then $[\omega]=[\bar{P}.\omega_0]$ for some $\bar{P}\in \mathrm{Per}(D,\mathbb{R})_{e_1}$ if there exists $A\in \mathcal{T}_{\mathcal{J}_{N}} $ such that  
     \[
     A^tBA=P^t(J_{(D,D)})_{1,1}P .  
     \]
    
 \end{coro}
\begin{proof}
    
Just need to notice that for any 
\begin{align*}
  \bar{P}^{-1}&=\left( \begin{array}{c|c} 
1 & 0 \\ \hline
0 & P^{-1} 
\end{array} \right) \in \mathrm{Per}(D,\mathbb{R})_{e_1}.
\end{align*}
The $2$-form  $\bar{P}^{-1}.\omega_0$ corresponds to the skew-symmetric matrix 

\[
\left( \begin{array}{c|c} 
0 & * \\ \hline
* & P^t(J_{(D,D)})_{1,1}P
\end{array} \right). 
\]
As the lower-right block $B$ of $\omega$ and
$\bar P^{-1}.\omega_0$ belong to the same congruence class over
$\mathcal T_{\mathcal J_N}$. Proposition
\ref{bijjectiontomatricesforsymplecticcase} implies that
\[
[\omega]=[\bar P^{-1}.\omega_0].
\]
\end{proof}

If we restrict ourselves to closed $2$-forms then from Proposition \ref{conditionsforexistenceclosedform}, we just need to add the condition that the block $B$ in $\ref{skewsymmetricmatrix}$ is in $\mathcal{ H}_{\mathcal{J}_N}^{K}$, with $K=R$ or $K=R-2$, appropriately. For example, we get the surjection

\begin{equation}    
\mathcal{H}_{\mathcal{J}_N}^{D-2}/\text{``congruence over $ \mathcal{T}_{\mathcal{J}_{N}}$''} \twoheadrightarrow
\mathfrak{P}\Omega^2_{D, \;\text{closed}}.
\end{equation} 

In Theorem \ref{main2} we will show that Corollary \ref{corollaryconditionsforpermutation} can always be applied for $\mathfrak{P}\Omega^2_{D,\, closed}(\mathfrak{g})$. In the next section we study the sets $\mathcal{H}_{\mathcal{J}_N}^{K}$, $ \mathcal{T}_{\mathcal{J}_{N}}$ and their congruence.

\section{A Matrix problem}\label{matrixproblemsection}
Again, in this section, let $N=D-1$, where $D$ is the dimension of the Almost abelian Lie algebra we are considering. Here we study in more detail the matrix equations obtained in the previous section for the study of closed $2$-forms.
Let

\[
\mathcal{J}_N
=
\mathcal{J}_{n_1}(\lambda_1)\oplus\cdots\oplus
\mathcal{J}_{n_k}(\lambda_k)
\oplus
\mathcal{C}_{n_{k+1}}(a_1,b_1)\oplus\cdots\oplus
\mathcal{C}_{n_p}(a_p,b_p) \in \mathrm{M}(N\times N,\mathbb{R}).
\]

be a matrix in real Jordan normal form (see \cite{matrixanalysis},Theorem 3.4.1.5. ). We recall the definition of the following sets
\begin{equation} \label{matrixsets}
    \begin{aligned}
  \mathcal{T}_{\mathcal{J}_{N}}&:= \left\{ A \in \mathrm{GL}(N,\mathbb{R})\;\middle|\;  A\mathcal{J}_{N}-\mathcal{J}_{N}A=0   
               \right\},  \\
  \mathcal{ H}_{\mathcal{J}_N}^{K}&:=\left\{ B \in \mathrm{Skew}(N\times N,\mathbb{R}) \;\middle|\; \begin{aligned}
B\mathcal{J}_N+\mathcal{J}_N^tB=0,  \\ \mathrm{Rank}(B)=K 
\end{aligned}
 \right\}.
    \end{aligned}
\end{equation}

First we need to describe the elements of these sets. Then we consider the following quotient space:
 \begin{align}\label{quotient space}     
 \mathcal{H}^K_{\mathcal{J}_N}/_{\mathrm{cong}}  \mathcal{T}_{\mathcal{J}_{N}}. 
\end{align}
In the following, we will use the notation $C=[C_{ij}]^k_{i=1,j=1} \in \mathrm{M}(N\times N,\mathbb{R})$ to denote a block matrix, with the same block shape as $\mathcal{J}_N$. Each block denoted by $C_{ij} \in \mathrm{M}(n_i\times n_j,\mathbb{R})$.
We need to study the equations 
\begin{equation}\label{permutationequation}
  C\mathcal{J}_{N}-\mathcal{J}_{N}C=0,
\end{equation}
\begin{equation}\label{transposepermutationequation}
    C\mathcal{J}_N+\mathcal{J}_N^tC=0.
\end{equation}
We search for a description of the solution space for these equations. This type of equations has been discussed in a more general setting (See \cite{topicsmatrixanalysis}, Section 4.4). The first one is just a commutativity equation and the second one is a special case of Lyapunov's equation. 
We can think of the matrix $\mathcal{J}_N$ as a block matrix 
\begin{equation}\label{jornanrealpluscomplex}
    \mathcal{J}_N=\mathcal{J}_{N_{\mathbb{R}}}\oplus \mathcal{J}_{N_{\mathbb{C}}}=\left( \begin{array}{c|c} 
\text{Real eigenvalues} & 0 \\ \hline
0 & \text{Complex eigenvalues}  
\end{array} \right),
\end{equation}
$N_{\mathbb{R}}+N_{\mathbb{C}}=N$.

From Theorem 4.4.6 in \cite{topicsmatrixanalysis} we know that a matrix $C$ satisfying Equation (\ref{permutationequation}) or Equation (\ref{transposepermutationequation})  will have the same diagonal shape. Therefore, we can write
\begin{align}
    \mathcal{T}_{\mathcal{J}_{N}}&= \mathcal{T}_{\mathcal{J}_{N_{\mathbb{R}}}}\oplus  \mathcal{T}_{\mathcal{J}_{N_{\mathbb{C}}}} , \\
     \mathcal{H}_{\mathcal{J}_{{N}}}^{K} &= \bigcup_{
     Y+W=K} \left(\mathcal{ H}_{\mathcal{J}_{N_{\mathbb{R}}}}^{Y} \oplus \mathcal{ H}_{\mathcal{J}_{N_{\mathbb{C}}}}^{W}\right).
\end{align}
The sets on the right hand side are defined as
\begin{align}
   \mathcal{T}_{\mathcal{J}_{N_{\mathbb{R}}}}&:= \left\{ A \in \mathrm{GL}(N_{\mathbb{R}},\mathbb{R})\;\middle|\;  A\mathcal{J}_{N_{\mathbb{R}}}-\mathcal{J}_{N_{\mathbb{R}}}A=0   
               \right\}, \\
   \mathcal{T}_{\mathcal{J}_{N_{\mathbb{C}}}}&:= \left\{ A \in \mathrm{GL}(N_{\mathbb{C}},\mathbb{R})\;\middle|\;  A\mathcal{J}_{N_{\mathbb{C}}}-\mathcal{J}_{N_{\mathbb{C}}}A=0   
               \right\}, \\ 
   \mathcal{ H}_{\mathcal{J}_{N_{\mathbb{R}}}}^{Y}&:=\left\{ B \in \mathrm{Skew}(N_{\mathbb{R}}\times N_{\mathbb{R}},\mathbb{R}) \;\middle|\; \begin{aligned}
B\mathcal{J}_{N_{\mathbb{R}}}+\mathcal{J}_{N_{\mathbb{R}}}^tB=0,  \\ \mathrm{Rank}(B)=Y 
\end{aligned}
 \right\},  \\
   \mathcal{H}_{\mathcal{J}_{N_{\mathbb{C}}}}^{W}&:=\left\{ B \in \mathrm{Skew}(N_{\mathbb{C}}\times N_{\mathbb{C}},\mathbb{R}) \;\middle|\; \begin{aligned}
B\mathcal{J}_{N_{\mathbb{C}}}+\mathcal{J}_{N_{\mathbb{C}}}^tB=0,  \\ \mathrm{Rank}(B)=W
\end{aligned}
 \right\}.
\end{align}

We now mention two facts about matrices that will be used in the following sections. The next proposition is known (it is used in Lemma 4.1 of \cite{Horn}, which as mentioned before inspired some of the proofs of the following chapters). It is a direct result of [\cite{existenceofroots}, Theorem 7]. 

\begin{prop} \label{existenceofroots}
Let $A$ be a real nonsingular matrix that has no negative real eigenvalues, then there exists a polynomial $p(A)$ in $A$ with real coefficients,  such that 
\[
p(A)^2=A^{-1}.
\]
\end{prop}
\begin{proof}
$A^{-1}$ has no negative real eigenvalues, so from [\cite{existenceofroots}, Theorem 7] there exists a polynomial $p_1$ such that 
\[
p_1(A^{-1})^2=A^{-1}.
\]
Also it is well known that there exists a polynomial $p_2$  such that $p_2(A)=A^{-1}$.  Setting $p:=p_1\circ p_2$, which is again a
polynomial with real coefficients, we obtain our desired polynomial.
\end{proof}
The next proposition is not a surprising fact. The proof follows easily by using a result in \cite{permutationmatrices}.
\begin{prop}
    \label{skewsymmetricpermutation}
     Consider a skew-symmetric matrix $A \in \matrices{N}{N}$ of rank $R$, such that the matrix $\abs{A} $ (obtained by substituting each nonzero element in $A$ by its absolute value) is a partial permutation matrix (so that some rows might be zero). Then there exists a permutation matrix $P$ such that
     \begin{align*}
     A=\begin{cases}
P^tJ_NP & \text{if} \;R=N  \\
P^tJ_{(N,R)}P &\text{if} \; R<N.
\end{cases}        
             \end{align*}        
  \end{prop}
 
 \begin{proof}
     Consider first the case $A$ is of even dimension and maximal rank.  $\abs{A} $ is a symmetric  ($\abs{A}^t=\abs{A}$)  permutation matrix . Define the matrix
     \[
     Q_{N}:=Q_2\oplus  \cdots  \oplus Q_2, \; \; Q_2=\left(\begin{array}{c|c} 
0&1\\ \hline
1&0
\end{array} \right).
     \]
     
     From Theorem 1 in \cite{permutationmatrices}, we can find a permutation matrix $P_1$ such that 
     \[
     \abs{A}=P_1^tQ_{N}P_1.
     \]

 Now define the matrix 
 \begin{equation}\label{9}
  S_N:=J_2\oplus  \cdots  \oplus J_2, \; \; J_2=\left(\begin{array}{c|c} 
0&1\\ \hline
-1&0
\end{array} \right).
 \end{equation}
 Clearly we have
 \[
\abs{ P_1^tS_{N}P_1} = \abs{A}
 \]
 $P_1^tS_{N}P_1$ is a skew-symmetric matrix such that the nonzero elements are in the same positions as those of $A$, but some elements could have the wrong sign.  We can find another permutation matrix $P_2$ that exchanges the necessary elements such that
 \[
 A=P_2^tP_1^tS_{N}P_1P_2.
 \]
  Finally, it is well known that there is a permutation matrix $P_3$  such that $S_N=P_3^tJ_{N}P_3$.
 So finally we get
 \[
 A=P_2^tP_1^tP_3^tJ_NP_3P_1P_2.
 \]
 Now let $A$ be of any rank. There exists a permutation matrix $P$ such that
 
 \[
 P^tAP=\left(\begin{array}{c|c} 
A^{\prime}&0\\ \hline
0&0
\end{array} \right),
 \]
 where $A^{\prime}$ is skew-symmetric of maximal rank. We  can apply the same argument as before using matrices of the shape 
 \[
 \left(\begin{array}{c|c} 
P_i&0\\ \hline
0&I \end{array} \right).
 \]
 \end{proof}
 
\subsection{The sets $\mathcal{T}_{\mathcal{J}_{N_{\mathbb{R}}}}$ and $\mathcal{H}_{\mathcal{J}_{{N}_{\mathbb{R}}}}^{K}$: real eigenvalues. }
In this section we suppose that the real Jordan normal form is given by 
\begin{equation}  \label{startingjordanformreal}
\mathcal{J}_{N_{\mathbb{R}}}=\mathcal{J}_{n_1}(\lambda_1)\oplus  \cdots  \oplus \mathcal{J}_{n_p}(\lambda_p).
\end{equation}
We first introduce some matrices and their properties. 
\subsubsection{Some special matrices and their properties} \label{sectionrealdefinitionmatrices}

We recall some of the notation used in \cite{Horn}, where, among other things, Equation (\ref{permutationequation}) is studied. We introduce some new definitions too.   
We say that a matrix $C_{ij} \in \mathrm{M}(n_i\times n_j,\mathbb{R})$ is \textit{upper Toeplitz} if it is of the form
 
\begin{align}\label{formtoeplitz}
C_{ij}=
\begin{cases}
\left( \begin{array}{ccccc} 
&c_{ij}&c_{ij}^{(2)}&\cdots& c_{ij}^{(n_i)} \\ 
&&c_{ij}&\ddots&\vdots \\
&&&\ddots&c_{ij}^{(2)} \\
0&&&& c_{ij}
\end{array} \right) \; \text{if} \; n_i\leq n_j,  \\
\left( \begin{array}{cccc} 
c_{ij}&c_{ij}^{(2)}&\cdots& c_{ij}^{(n_j)} \\ 
&c_{ij}&\ddots&\vdots \\
&&\ddots&c_{ij}^{(2)} \\
&&& c_{ij} \\
0&&&
\end{array} \right) \; \text{if} \; n_j\leq n_i.
\end{cases}
\end{align}

In a similar way we say a matrix is \textit{lower Hankel} if it is of the form 
\begin{align}\label{formhankel}
C_{ij}=
\begin{cases}
\left( \begin{array}{ccccc} 
0&&&&c_{ij}  \\ 
&&&\iddots&c_{ij}^{(2)} \\
&&c_{ij}&\iddots&\vdots \\
&c_{ij}&c_{ij}^{(2)}&\cdots& c_{ij}^{(n_i)}
\end{array} \right) \; \text{if} \; n_i\leq n_j,  \\
\left( \begin{array}{cccc} 
0&&&  \\ 
&&&c_{ij}  \\ 
&&\iddots&c_{ij}^{(2)} \\
&c_{ij}&\iddots&\vdots \\
c_{ij}&c_{ij}^{(2)}&\cdots& c_{ij}^{(n_j)}
\end{array} \right) \; \text{if} \; n_j\leq n_i.
\end{cases}
\end{align}
A matrix $C=[C_{ij}]^p_{i=1,j=1}$ is \textit{${N_{\mathbb{R}}}$-lower Hankel} if each block is lower Hankel. Define the following matrices having the same block partition as the matrix $C$. 
\begin{align}
    \mathcal{P}_{{N_{\mathbb{R}}}}=\mathcal{P}_{n_1}\oplus \cdots \oplus \mathcal{P}_{n_p}& \;\;, \text{with} \; \mathcal{P}_{n_i}=\left( \begin{array}{cccc}
    &&&1\\
    &&1&\\
    &\iddots&&\\
    1&&&
    \end{array}
    \right)\in \mathrm{M}(n_i\times n_i,\mathbb{R}), \\
    I_{{N_{\mathbb{R}}}}^{\pm}=I_{n_1}^{\pm}\oplus \cdots \oplus I_{n_p}^{\pm}& \;\;, \text{with} \; I_{n_i}^{\pm}=\left( \begin{array}{cccc}
    1&&&\\
    &-1&&\\
    &&1&\\
    &&&\ddots
    \end{array}
    \right)\in \mathrm{M}(n_i\times n_i,\mathbb{R}). 
\end{align}
We say that a matrix $C^{\prime}$ is \textit{${N_{\mathbb{R}}}$-upper alternating Toeplitz} if $C^{\prime}=I_{{N_{\mathbb{R}}}}^{\pm}C$ for some ${N_{\mathbb{R}}}$-upper Toeplitz matrix $C$. In the same way we define \textit{${N_{\mathbb{R}}}$-lower alternating Hankel} matrices.   Notice that if $C$ is ${N_{\mathbb{R}}}$-lower (alternating) Hankel matrix then 
\[
C=\mathcal{P}_{{N_{\mathbb{R}}}}C^{\prime}
\]
is a N-upper (alternating) Toeplitz matrix and vice versa. 

We recall here the definition of \textit{N-block star} of a block matrix.  

\[
C^{\boxstar}:=\mathcal{P}_{N_{\mathbb{R}}}C^{T}\mathcal{P}_{N_{\mathbb{R}}}.
\]
When referring to one single block $C_{ij}$ we will also use, unless there is some risk of confusion, the following notation 
\[
(C_{ij})^{\boxstar}:=\mathcal{P}_{n_j}C_{ij}^t\mathcal{P}_{n_i}.
\]

We will say that a matrix $C$ is \textit{N-skew-symmetric} if 
\[
C^{\boxstar}=-C.
\]
That is 
   
\begin{equation*}\label{blockofNskewsymetric}
C_{ji}=-(C_{ij})^{\boxstar}=-\mathcal{P}_{n_j}C_{ij}^t\mathcal{P}_{n_i}.
\end{equation*}

For $C_{ij}\in \matrices{n_i}{n_j}$, define
\[
\hat{C}_{ij}:=\altidentity{i}C_{ij}\altidentity{j}.
\]
It is easy to show that if $C$ is N-skew-symmetric, then $\mathcal{P}_{N_{\mathbb{R}}}C$ is skew-symmetric in the usual sense. 
If ${C}_{ij}$ is upper Toeplitz, then $\hat{C}_{ij}$ is again upper Toeplitz. Obviously 
\[
C_{ij}\altidentity{j}=\altidentity{i}\hat{C}_{ij}.
\]
In particular (for square $C_{ij}$),  if $p(C_{ij})$ is a polynomial in $C_{ij}$ and $p(\hat{C}_{ij})$ is the same polynomial in $\hat{C}_{ij}$ then 
\begin{equation}\label{commutationwithhat}
p(\hat{C}_{ij})\altidentity{j}=\altidentity{i}p(C_{ij}).
\end{equation}

If in addition to being $N$-upper alternating Toeplitz $C$ is also N-skew-symmetric we actually have 

\begin{equation}\label{conmmuteswithaltidentity}
\altidentity{i}C_{ii}=\begin{cases} 
C_{ii}\altidentity{i} & \text{if} \; n_i \; \text{is even}, \\
-C_{ii}\altidentity{i}& \text{if} \; n_i \; \text{is odd}.
\end{cases}
\end{equation}
\begin{proof}
    Notice that 
    \[
\altidentity{i}\mathcal{P}_{n_i}=\begin{cases}
-\mathcal{P}_{n_i}\altidentity{i} & \text{if} \; n_i \; \text{is even}, \\
\mathcal{P}_{n_i}\altidentity{i}  & \text{if} \; n_i \; \text{is odd}.
\end{cases}
\]

$\altidentity{i}C_{ii}$ is Toeplitz, therefore it is Persymmetric, that is 
\[
\altidentity{i}C_{ii}=\mathcal{P}_{n_i}C_{ii}^t\altidentity{i}\mathcal{P}_{n_i}.
\]
From our supposition $C_{ii}^t=-\mathcal{P}_{n_i}C_{ii}\mathcal{P}_{n_i}.$ Therefore,
\begin{equation} 
    \altidentity{i}C_{ii}=-\mathcal{P}_{n_i}(\mathcal{P}_{n_i}C_{ii}\mathcal{P}_{n_i})\altidentity{i}\mathcal{P}_{n_i}=\begin{cases}
C_{ii}\altidentity{i} & \text{if} \; n_i \; \text{is even}, \\
-C_{ii}\altidentity{i} & \text{if} \; n_i \; \text{is odd}.
\end{cases} 
\end{equation}
\end{proof}
\begin{remark} \label{diagonalsarezero}
This previous observation can also be understood as a consequence of the fact that for a $N$-skew-symmetric  N-upper Toeplitz  matrix $C$, the diagonals of $C_{ii}$ that have an odd number of elements have all their elements equal to $0$. So necessarily a nonzero diagonal will have an even number of entries. 
    \end{remark}

    \subsubsection{The solutions of equations (\ref{permutationequation}) and (\ref{transposepermutationequation})}
For reference in the following statements we state here two extra conditions a given block matrix $C=[C_{ij}]$ could satisfy with respect to $\mathcal{J}_N$:
\begin{align}    \label{conditionblocksrealtoeplitz}
     C_{ij}&=0 \;\; \text{whenever} \; \; \lambda_i \neq \lambda_j. \\ 
     \label{conditionblockrealankel}
     C_{ij}&=0 \;\; \text{whenever} \;\;\lambda_i\neq -\lambda_j.
     \end{align}

The next proposition is known (Lemma 4.4.11 in \cite{topicsmatrixanalysis} and the discussion following it).
\begin{prop}\label{conmmutationwithjordan}
 A block matrix $C$ satisfies 
\begin{align}\label{conmmutationwithjordanequation}
    C\mathcal{J}_{N_{\mathbb{R}}}-\mathcal{J}_{N_{\mathbb{R}}}C=0
\end{align}
if and only if
 $C$ is $N$-upper Toeplitz matrix and satisfies condition (\ref{conditionblocksrealtoeplitz}).

\end{prop}

For equation (\ref{transposepermutationequation}) , we just slightly modify the proof of the previous proposition. First we consider the case of only one Jordan block.

\begin{lemma}\label{onejordanblocklemmatransposecommute}
Let $\mathcal{J}_r(0) \in \mathrm{M}(r\times r,\mathbb{R})$, $\mathcal{J}_s(0) \in \mathrm{M}(s\times s,\mathbb{R})$ and $X \in \mathrm{M}(r\times s,\mathbb{R})$. Then 
\[
X\mathcal{J}_s+\mathcal{J}_r^tX=0
\]
if and only if 
\begin{align*}
    X&=\left( \begin{array}{c} 
0\\ 
Y
\end{array} \right), Y\in \mathrm{M}(s\times s,\mathbb{R})\;\;\text{if}\; r\geq s\; \text{, or} \\ 
    X&=\left( \begin{array}{cc} 0&Y
\end{array} \right), Y\in \mathrm{M}(r\times r,\mathbb{R}) \; \; \text{if}\; r\leq s,
\end{align*}
where $Y$ is a lower alternating Hankel matrix.
\end{lemma}
\begin{proof}
Let $X=(x_{ij})$. 
\begin{align*}
    (XJ(0))_{ij}=x_{ij-1} \\
    (J(0)^tX)_{ij}=x_{i-1j}
\end{align*}
So the condition in the statement becomes 
\[x_{ij-1}+x_{i-1j}=0.\] 
\end{proof}
\begin{prop}\label{Propositionconditioncommuteswithtranspose}
 A block matrix $C$ satisfies 
\begin{align}\label{conmmutationwithtranspose}
    C\mathcal{J}_{N_{\mathbb{R}}}+\mathcal{J}_{N_{\mathbb{R}}}^tC=0
\end{align}
if and only if
 $C$ is $N_{\mathbb{R}}$-lower alternating Hankel matrix and satisfies condition (\ref{conditionblockrealankel}).
\end{prop}
\begin{proof}
Equation (\ref{conmmutationwithtranspose}) is equivalent to the set 
\[
C_{ij}\mathcal{J}_{n_j}(\lambda_j)+\mathcal{J}_{n_i}(\lambda_i)^tC_{ij}=0.
\]
As mentioned in the previous section, each of these equations has a nontrivial solution if and only if $\lambda_j+\lambda_i=0$. So we only consider equations of the form
\[
C_{ij}\mathcal{J}_{n_j}(\lambda_j)+\mathcal{J}_{n_i}(-\lambda_j)^tC_{ij}=0.
\]
Using the identity $\mathcal{J}_{n_k}(\lambda_k)=\lambda_k I_{n_k}+\mathcal{J}_{n_k}(0)$ we get
\[
C_{ij}\mathcal{J}_{n_j}(0)+\mathcal{J}_{n_i}(0)^tC_{ij}=0.
\]

This is the type of equation we got in Lemma \ref{onejordanblocklemmatransposecommute}.  This completes the proof.
\end{proof}
We can now rewrite the matrix sets in (\ref{matrixsets}) as
\begin{equation} \label{matrixsetsrealrewrite}
    \begin{aligned}
  \mathcal{T}_{\mathcal{J}_{N_{\mathbb{R}}}}&:= \left\{ [A_{ij}] \in \mathrm{GL}({N_{\mathbb{R}}},\mathbb{R})\;\middle|\; \begin{aligned}
 &{N_{\mathbb{R}}}\text{-upper Toeplitz} \\ 
&\text{Satisfies} \;(\ref{conditionblocksrealtoeplitz}) 
\end{aligned}
               \right\}  \\
              \mathcal{H}_{\mathcal{J}_{N_{\mathbb{R}}}}^K&:=\left\{ [B_{ij}] \in \mathrm{Skew}({N_{\mathbb{R}}}\times {N_{\mathbb{R}}},\mathbb{R}) \;\middle|\; \begin{aligned}
 &{N_{\mathbb{R}}}\text{-lower alternating Hankel}  \\ 
&\text{Satisfies} \;(\ref{conditionblockrealankel}) \\
&Rank=K
\end{aligned}
 \right\}\
    \end{aligned}
\end{equation}

In the next section we answer the question: for a given $\mathcal{J}_{N_{\mathbb{R}}}$, for which $K$ is $\mathcal{H}_{\mathcal{J}_{N_{\mathbb{R}}}}^K\neq \emptyset$.  Then we examine the shape of the quotient space $(\ref{quotient space})$.  In fact we will start with the latter, as it will let us answer the former more easily.

\subsection{The quotient space} \label{realquotientspacesection}

First we show that we can study an equivalent problem. 
\begin{lemma} \label{equivalentmatrixproblemreal}
Define the sets 
\begin{equation} \label{matrixsetsrealrewritewithstar}
    \begin{aligned}
 \dot{\mathcal{T}}_{\mathcal{J}_{N_{\mathbb{R}}}}&:= \left\{ [A_{ij}] \in \mathrm{GL}({N_{\mathbb{R}}},\mathbb{R})\;\middle|\; \begin{aligned}
 &{N_{\mathbb{R}}}\text{-upper Toeplitz} \\ 
&\text{Satisfies} \;(\ref{conditionblocksrealtoeplitz}) 
\end{aligned}
               \right\}  \\
               \dot{\mathcal{H}}_{\mathcal{J}_{N_{\mathbb{R}}}}^{K}&:=\left\{ [B_{ij}] \in {N_{\mathbb{R}}}\text{-Skew}({N_{\mathbb{R}}}\times {N_{\mathbb{R}}},\mathbb{R}) \;\middle|\; \begin{aligned}
 &{N_{\mathbb{R}}}\text{-upper alternating Toeplitz}  \\ 
&\text{Satisfies} \;(\ref{conditionblockrealankel}) \\
&Rank=K
\end{aligned}
 \right\}\
    \end{aligned}
\end{equation}
Then there exists a bijection 

\begin{equation}\label{bijectionquotient space}     
 \mathcal{H}_{\mathcal{J}_{N_{\mathbb{R}}}}^{K}/_{\mathrm{cong}}\mathcal{T}_{\mathcal{J}_{N_{\mathbb{R}}}} \xrightarrow{bij}  \dot{\mathcal{H}}_{\mathcal{J}_{N_{\mathbb{R}}}}^{K}/_{\text{N-blockstar cong}}\dot{\mathcal{T}}_{\mathcal{J}_{N_{\mathbb{R}}}}
\end{equation} 
Here the right side means congruence with respect to the $\boxstar$ operator.
\end{lemma}
    \begin{proof}
   Let $B,C \in \mathcal{H}_{\mathcal{J}_{N_{\mathbb{R}}}}^{K} $. recall from Section (\ref{sectionrealdefinitionmatrices}) that $\mathcal{P}_{N_{\mathbb{R}}}B$ and $\mathcal{P}_{N_{\mathbb{R}}}C$ are $N$-skew-symmetric $N$-upper Toeplitz. Then we just have to notice that  for any $A\in \dot{\mathcal{T}}_{\mathcal{J}_{N_{\mathbb{R}}}}$
        \begin{align*}
             &A^tCA=B \\
             \Leftrightarrow  &\mathcal{P}_{N_{\mathbb{R}}}\mathcal{P}_{N_{\mathbb{R}}}A^t\mathcal{P}_{N_{\mathbb{R}}}\mathcal{P}_{N_{\mathbb{R}}}CA=B \\
              \Leftrightarrow  &A^{\boxstar}(\mathcal{P}_{N_{\mathbb{R}}}C)A=\mathcal{P}_{N_{\mathbb{R}}}B.
        \end{align*}

    \end{proof}

From now on we focus on the right hand quotient space of $(\ref{bijectionquotient space})$.  We mention one more time that inspiration for the proof comes from the proof of Lemma 4.1 in \cite{Horn}, the equivalent statement to that lemma will be Proposition \ref{canonicalformsreal}. Before we give a proof of this proposition we will introduce again some more matrices and notation. 

\subsubsection{Some more matrices and notation.} \label{Some more matrices and notation.}
For the use in the rest of this section and unclutter notation we define the following 
\begin{align*}
    n_{l,m}:=\mathrm{min}(n_l,n_m), \\
    s_{(l,m;k)}:=n_{l,m}-k+1.
\end{align*}
 Consider again an upper Toeplitz block $C_{lm}$. We denote by $C_{lm}^{\langle u \rangle} \in \matrices{s_{(l,m;u)}}{s_{(l,m;u)}}$ the submatrix formed by only considering the diagonals with elements $c_{lm}^{(t)}$, with $t\geq u$ that is 
\[
C_{lm}^{\langle u \rangle}:=\left( \begin{array}{cccc} 
c_{lm}^{(u)}&c_{lm}^{(u+1)}&\cdots& c_{lm}^{(n_{l,m})} \\ 
&c_{lm}^{(u)}&\ddots&\vdots \\
&&\ddots&c_{lm}^{(u+1)} \\
&&& c_{lm}^{(u)}\\
\end{array} \right) \in \matrices{s_{(l,m;u)}}{s_{(l,m;u)}}.
\]
If $c_{lm}^{(u)}\neq 0$, then $C_{lm}^{\langle u \rangle}$ is upper Toeplitz matrix with nonzero diagonal. Another kind of matrix that will appear later is the matrix obtained by erasing the submatrix $C_{lm}^{\langle u \rangle}$ from the original matrix $C_{lm}$, we will denote this matrix by $C_{lm}^{<u}$. That is this matrix is defined by 
\begin{equation} \label{matrixwithoutcorner}
\matrices{n_l}{n_m}\ni C_{lm}^{<u}:=\left( \begin{array}{cccccc} 
&c_{lm}&\cdots &c_{lm}^{(u-1)}& & 0 \\ 
&&c_{lm}&\cdots&\ddots& \\
&&&\ddots&&c_{lm}^{(u-1)} \\
&&&& c_{lm}& \vdots \\
0&&&& &c_{lm}
\end{array} \right)    
\end{equation}
The matrix depicted is in the case $n_l\leq n_m$, but the equivalent construction is obvious when $n_l\geq n_m$. Now suppose we have an upper Toeplitz matrix $A\in \matrices{k}{k}$, we use the notation similar as in (\ref{formtoeplitz}) to denote diagonals of the matrix by $a^{(i)}$. For any $n,m$, such that $k\leq\min(n,m)$ we can \say{extend} the matrix $A$ to another upper Toeplitz matrix  in $\matrices{n}{m}$ by a map

\[
\mathrm{ToeplitzExt}_{n,m}:\matrices{k}{k}\rightarrow \matrices{n}{m}
\]
To define this map, define $C:=\mathrm{ToeplitzExt}_{n,m}(A)$ and again use the notation $c^{(i)}$ for the diagonals, then we can define the map by  
\begin{equation}
   c^{(i)}:=\begin{cases}
       a^{(i)} & 1\leq i\leq k \\
       0&  \text{other cases}.
   \end{cases}
\end{equation}

The important property of this extended matrix is that, for example in the case $n_l\leq n_m$, it has the shape
\begin{equation}\label{blockappearstwice}
    \mathrm{ToeplitzExt}_{n,m}(A)= \left( \begin{array}{c|cc} 
0&*&* \\ 
0&0&A 
\end{array} \right)=\left( \begin{array}{c|cc} 
0&A &* \\ 
0&0&*
\end{array} \right),\end{equation}
where the $0$ and $*$ are selected so that it agrees with the definition. That is the matrix $A$ appears both in the lower and upper corner of the square part of each matrix. Of course we have  a similar situation when $n_m\leq n_l$ in the obvious way. 

For example, for  $n_l\leq n_m$, if we are considering again an upper Toeplitz block $C_{lm}^{\langle u \rangle}$ then 

\[
\mathrm{ToeplitzExt}_{n_l,n_m}(C_{lm}^{\langle u \rangle}):=\left( \begin{array}{cccccc} 
&c_{lm}^{(u)}&\cdots &c_{lm}^{(n_l)}& & 0 \\ 
&&c_{lm}^{(u)}&\cdots&\ddots& \\
&&&\ddots&&c_{lm}^{(n_l)} \\
&&&& c_{lm}^{(u)}& \vdots \\
0&&&& &c_{lm}^{(u)}
\end{array} \right) 
\]

We define one more matrix

\[
 I_{lm}^{[k]\pm}:=\left(  \begin{array}{c|c} 
0&I_{s_{(l,m;k)}}^{\pm}\\ \hline
0&0
\end{array} \right) \in \matrices{n_l}{n_m}.
 \]
 For convenience, we extend this notation by setting
\[
I_{lm}^{[k]\pm}:=0
\qquad\text{whenever}\qquad
k>\min(n_l,n_m).
\]
Finally notice the following 
 \begin{align}\label{conjugationofE}
(I_{lm}^{[k]\pm})^{\boxstar}:=
 \begin{cases}
-I_{ml}^{[k]\pm} &\text{if} \; s_{(l,m;k)} \; \text{is even}, \\
I_{ml}^{[k]\pm} &\text{if} \; s_{(l,m;k)} \; \text{is odd}.
\end{cases}
\end{align}

\subsubsection{The set $\dot{\mathcal{H}}_{\mathcal{J}_{N_{\mathbb{R}}}}^{K}/_{\text{N-blockstar cong}}\dot{\mathcal{T}}_{\mathcal{J}_{N_{\mathbb{R}}}}$}

\begin{lemma}\label{step1}
     Let $C=[C_{ij}]\in \dot{\mathcal{H}}_{\mathcal{J}_{N_{\mathbb{R}}}}^{K}$. Suppose $C_{lm}$ is such that $c_{lm}^{(k)} \neq 0$ and $c_{lm}^{(q)}=0$ for all $q<k$.  There exists  $S \in\dot{\mathcal{T}}_{\mathcal{J}_{N_{\mathbb{R}}}}$ such that 
   
\[
S^{\boxstar}CS=\begin{cases}
\left( \begin{array}{ccccc} 
*&&&&\\ 
&\ddots&&*&\\
&&tI_{lm}^{[k]\pm}&& \\
&*&&\ddots &\\
&&& & *
\end{array} \right) \;\; \text{if} \; l=m, \\
\left(\begin{array}{ccccccc} 
*&&&&&&  \\ 
&\ddots&&&&*& \\
&&*&\cdots&tI_{lm}^{[k]\pm}&& \\
&&\vdots&*&\vdots&& \\
&&-t(I_{lm}^{[k]\pm})^{\boxstar}&\cdots&*&& \\
&*&&&&\ddots& \\
&&&&&&*
\end{array} \right) \;\; \text{if} \; l\neq m.
\end{cases} 
\]          
That is $S^{\boxstar}CS \in \dot{\mathcal{H}}_{\mathcal{J}_{N_{\mathbb{R}}}}^{K} $ is such that the $l,m$  component is $tI_{lm}^{[k]\pm}$.

 \end{lemma}
\begin{proof}

Define the matrix
\[
C_{lm}^{\langle k \rangle,\mathrm{na}}:=tI^{\pm}_{s_{(l,m;k)}}C_{lm}^{\langle k \rangle} ,
\]
where $t\in \lbrace 1,-1\rbrace$ is selected such that the diagonal elements of $C_{lm}^{\langle k \rangle,\mathrm{na}}$ are positive, the matrix $C_{lm}^{\langle k \rangle,\mathrm{na}}$ is upper Toeplitz (the \say{na} comes from non alternating, suggestive of what the operation does). From Proposition \ref{existenceofroots}, we can find a polynomial $p(C_{lm}^{\langle k \rangle,\mathrm{na}})$ such that $p(C_{lm}^{\langle k \rangle,\mathrm{na}})^2=(C_{lm}^{\langle k \rangle,\mathrm{na}})^{-1}$.
In particular, we have 
\[
p(C_{lm}^{\langle k \rangle,\mathrm{na}})C_{lm}^{\langle k \rangle,\mathrm{na}}p(C_{lm}^{\langle k \rangle,\mathrm{na}})=I_{s_{(l,m;k)}}.
\]
\textbf{Case $l=m$}. (Notice that for this case we necessarily have $\lambda_l=0$).  $C_{ll}^{\langle k \rangle}$ is an even dimensional matrix (see Remark \ref{diagonalsarezero}), so from (\ref{conmmuteswithaltidentity})  $C_{ll}^{\langle k \rangle,\mathrm{na}}I^{\pm}_{s_{(l,l;k)}}=I^{\pm}_{s_{(l,l;k)}}C_{ll}^{\langle k \rangle,\mathrm{na}}$. Therefore, we also have 
\[
I_{s_{(l,l;k)}}^{\pm}p(C_{ll}^{\langle k \rangle,\mathrm{na}})=p(C_{ll}^{\langle k \rangle,\mathrm{na}})I_{s_{(l,l;k)}}^{\pm}
\]
Now define the matrix 
\[
 \matrices{n_l}{n_l} \ni S_{ll}:=\mathrm{ToeplitzExt}_{n_l,n_l}(p(C_{ll}^{\langle k \rangle,\mathrm{na}})). 
 \] 
For such a matrix we have 
\begin{align*}
S_{ll}C_{ll}S_{ll}&=\left(\begin{array}{c|c} 
0&p(C_{ll}^{\langle k \rangle,\mathrm{na}})C_{ll}^{\langle k \rangle}p(C_{ll}^{\langle k \rangle,\mathrm{na}}) \\ \hline
0&0
\end{array} \right) \\
&=t\left(\begin{array}{c|c} 
0&p(C_{ll}^{\langle k \rangle,\mathrm{na}})I^{\pm}_{s_{(l,l;k)}}C_{ll}^{\langle k \rangle,\mathrm{na}}p(C_{ll}^{\langle k \rangle,\mathrm{na}}) \\ \hline
0&0
\end{array} \right) \\
&=t\left(\begin{array}{c|c} 
0&I^{\pm}_{s_{(l,l;k)}}p(C_{ll}^{\langle k \rangle,\mathrm{na}})C_{ll}^{\langle k \rangle,\mathrm{na}}p(C_{ll}^{\langle k \rangle,\mathrm{na}}) \\ \hline
0& 0
\end{array} \right)=t I_{ll}^{[k]\pm}.
\end{align*}
Finally, define the nonsingular diagonal matrix  $ S=[S_{ij}] \in \dot{\mathcal{T}}_{\mathcal{J}_{N_{\mathbb{R}}}}$ by 
\[
S_{ij} :=
\begin{cases}
I_{n_i} & i = j,\ i \neq l,\\
\mathrm{ToeplitzExt}_{n_l,n_l}(p(C_{ll}^{\langle k \rangle,\mathrm{na}})) & i = j = l,\\
0 & i \neq j.
\end{cases}
\]
This matrix gives our desired result.

\textbf{Case $l\neq m$.} As in (\ref{commutationwithhat}) we have 
\begin{align*}
    p(\hat{C}_{lm}^{\langle k \rangle,\mathrm{na}})C_{lm}^{\langle k \rangle}p(C_{lm}^{\langle k \rangle,\mathrm{na}})&=tp(\hat{C}_{lm}^{\langle k \rangle,\mathrm{na}})I^{\pm}_{s_{(l,m;k)}}C_{lm}^{\langle k \rangle,\mathrm{na}}p(C_{lm}^{\langle k \rangle,\mathrm{na}}) \\
      &=tI^{\pm}_{s_{(l,m;k)}}.
\end{align*}
Similarly as in the previous lemma, define the upper Toeplitz matrices

\[
 \matrices{n_m}{n_m} \ni S_{mm}:=\mathrm{ToeplitzExt}_{n_m,n_m}\left(p(C_{lm}^{\langle k \rangle,\mathrm{na}})\right),
 \]
 \[
 \matrices{n_l}{n_l} \ni S_{ll}:=\mathrm{ToeplitzExt}_{n_l,n_l}\left(p(\hat{C}_{lm}^{\langle k \rangle,\mathrm{na}})\right).
\] 
Then, we have
\begin{align*}
S_{ll}C_{lm}S_{mm}&=\left(\begin{array}{c|c} 
0&p(\hat{C}_{lm}^{\langle k \rangle,\mathrm{na}})C_{lm}^{\langle k \rangle}p(C_{lm}^{\langle k \rangle,\mathrm{na}}) \\ \hline
0& 0
\end{array} \right)=tI_{lm}^{[k]\pm}.
\end{align*}
Finally, define the diagonal nonsingular matrix  $ S=[S_{ij}] \in \dot{\mathcal{T}}_{\mathcal{J}_{N_{\mathbb{R}}}}$ by 
\[
S_{ij} :=
\begin{cases}
I_{n_i} & i = j,\ i \neq l, i \neq m,\\
\mathrm{ToeplitzExt}_{n_l,n_l}\left(p(\hat{C}_{lm}^{\langle k \rangle,\mathrm{na}})\right) & i = j = l,\\
\mathrm{ToeplitzExt}_{n_m,n_m}\left(p(C_{lm}^{\langle k \rangle,\mathrm{na}})\right) & i = j = m,\\
0 & i \neq j.
\end{cases}
\]
This matrix gives our desired result. 

\end{proof}   

Next using the \say{alternating Identity} matrices we obtained in the previous Lemmas we eliminate elements in their corresponding columns and rows.

 \begin{lemma}\label{firstlemma}   Let $C=[C_{ij}]\in \dot{\mathcal{H}}_{\mathcal{J}_{N_{\mathbb{R}}}}^{K}$.  Suppose $C_{ll}$ is such that $c_{ll}^{(k)} \neq 0$ and $c_{ll}^{(q)}=0$ for all $q<k$. 
 Then, there exists  $S \in\dot{\mathcal{T}}_{\mathcal{J}_{N_{\mathbb{R}}}}$ and $t\in \lbrace 1,-1\rbrace $ such that

  \[
B=[B_{ij}]:=S^{\boxstar}CS
\]        
satisfies 
\begin{enumerate}
    \item $B_{ll}=tI_{ll}^{[k]\pm},$
    \item $b_{il}^{(u)}=b_{li}^{(u)}=0$ for all   $u\geq k$ and all $i\neq l$. 
\end{enumerate}

 \end{lemma}
 \begin{proof}
         First, from Lemma \ref{step1} we can without loss of generality suppose that $C_{ll}=tI_{ll}^{[k]\pm}$ for some $t\in \lbrace 1,-1\rbrace $. For $C_{lm}$, $m\neq l$, define the following matrix  (selecting $t$ as we did in Lemma $\ref{step1})$ 
\begin{equation}\label{matrixtoeliminate}
C_{lm}^{\langle k \rangle,\mathrm{na}}:=tI^{\pm}_{s_{(l,m;k)}}C_{lm}^{\langle k \rangle}.
\end{equation}
Then define the matrix 
    
\begin{equation}\label{completethematrix}
 \matrices{n_l}{n_m} \ni S_{lm} := \mathrm{ToeplitzExt}_{n_l,n_m}(C_{lm}^{\langle k \rangle,\mathrm{na}}).
\end{equation} Notice that 
\[
(tI_{ll}^{[k]\pm})(-S_{lm})
=
-\left(C_{lm}-C_{lm}^{<k}\right).
\]

Finally, define $ S=[S_{ij}] \in \dot{\mathcal{T}}_{\mathcal{J}_{N_{\mathbb{R}}}}$ as 

\[
S_{ij} :=
\begin{cases}
I_{n_i} & i = j,\\
-\mathrm{ToeplitzExt}_{n_l,n_j}(C_{lj}^{\langle k \rangle,\mathrm{na}}) & \text{all} \;\; i=l, \; j\neq l ,\\
0 & \text{other cases}.
\end{cases}
\]
Written explicitly this matrix looks like
\[
S=\left( \begin{array}{ccccccc} 
I_{n_1}&&& &&&\\ 
&\ddots&&&&&\\
&&\ddots&&&&\\
-S_{l1}&\cdots&-S_{l\,l-1}&I_{n_l}&-S_{l\,l+1}&\cdots &-S_{l\,p} \\
&&&&\ddots&& \\
&&&&&\ddots&\\
&&&&& &I_{n_p}
\end{array} \right).
\]
Then consider the product 
\[
S^{\boxstar}CS,
\]
which gives our desired answer. Notice that there is no problem defining the previous matrix $S$, because if $C_{ll}\neq 0$,  condition (\ref{conditionblockrealankel}) implies $\lambda_l=0$. Then for all $i\neq l$, if $C_{li} \neq 0$ we have also that $\lambda_i =0$. So there is no problem with having $S_{li}\neq0$, as it satisfies condition (\ref{conditionblocksrealtoeplitz}).
 \end{proof}
 We can immediately obtain the following Corollary. 
\begin{coro}\label{corollarytofirstlemma}
   Let $C=[C_{ij}]\in \dot{\mathcal{H}}_{\mathcal{J}_{N_{\mathbb{R}}}}^{K}$.  Suppose $C_{ll}$ is such that $c_{ll} \neq 0$.  Then, there exists  $S \in\dot{\mathcal{T}}_{\mathcal{J}_{N_{\mathbb{R}}}}$ and $t\in \lbrace 1,-1\rbrace $ such that

  \[
S^{\boxstar}CS=
\left( \begin{array}{ccccccc} 
&&&0&&&\\ 
&*&&\vdots&&*&\\ 
&&&0&&&\\
0&\cdots&0&tI_{n_l}^\pm&0&\cdots& 0 \\
&&&0& &&\\
&*&&\vdots& &*&\\
&&&0& &&

\end{array} \right).
\]           
\end{coro}

 \begin{lemma}\label{secondlemma}
     Let $C=[C_{ij}]\in \dot{\mathcal{H}}_{\mathcal{J}_{N_{\mathbb{R}}}}^{K}$. Suppose that $C_{lm}$, with $l\neq m$ and $n_l=n_m$, is such that $c_{lm}^{(k)} \neq 0$ and $c_{lm}^{(q)}=0$ for all $q<k$. Suppose that  $c_{ll}^{(q)}=c_{mm}^{(q)}=0$ for $q\leq k$.
   Then, there exists  $S \in\dot{\mathcal{T}}_{\mathcal{J}_{N_{\mathbb{R}}}}$ and $t\in \lbrace 1,-1\rbrace $ such that

  \[
B=[B_{ij}]:=S^{\boxstar}CS
\]         
satisfies
\begin{enumerate}
    \item $B_{lm}=tI_{lm}^{[k]\pm},$
    \item $B_{ll}=B_{mm}=0$.
    \item $b_{il}^{(u)}=b_{li}^{(u)}=b_{im}^{(u)}=b_{mi}^{(u)}=0$ for all   $u\geq k$  for all $i\neq m,l$ ( That is for such blocks the matrix has the same shape as described in $(\ref{matrixwithoutcorner})$).
\end{enumerate}
 \end{lemma}
 \begin{proof}
     From Lemma \ref{step1} we can without loss of generality suppose that $C_{lm}=t I_{lm}^{[k]\pm}$, $t\in \lbrace 1,-1\rbrace $. The proof of this case has two parts. First, we eliminate the elements $C_{ll}$ and $C_{mm}$. Similarly as in previous lemmas define the matrices
     
\begin{align*}
    C_{mm}^{\langle k \rangle,\mathrm{na}}&:=tI^{\pm}_{s_{(m,m;k)}}C_{mm}^{\langle k \rangle}, \\
   C_{ll}^{\langle k \rangle,\mathrm{na}}&:=tI^{\pm}_{s_{(l,l;k)}}C_{ll}^{\langle k \rangle}.
\end{align*}
  Then we define the corresponding matrices 
  \begin{align*}
    \matrices{n_l}{n_m} \ni S_{lm}&:=\mathrm{ToeplitzExt}_{n_l,n_m}(C_{mm}^{\langle k \rangle,\mathrm{na}}) \\
    \matrices{n_m}{n_l} \ni S_{ml}&:=\mathrm{ToeplitzExt}_{n_m,n_l}(C_{ll}^{\langle k \rangle,\mathrm{na}})
\end{align*}

If either $C_{ll}\neq 0$ or $C_{mm}\neq 0$, from condition (\ref{conditionblockrealankel}) we have $\lambda_l=0$ or $\lambda_m=0$ respectively. So  $C_{lm}\neq 0 $ implies that, in either case, $\lambda_l=\lambda_m=0$. So in any case we have no problem defining the following matrix $S=[S_{ij}]$ satisfying $(\ref{conditionblocksrealtoeplitz})$  and with blocks given by 
\[
S_{ij} :=
\begin{cases}
I_{n_i} & i = j,\\
-\frac{1}{2}\mathrm{ToeplitzExt}_{n_l,n_m}(C_{mm}^{\langle k \rangle,\mathrm{na}}) & i =l, j=m,\\
-\frac{1}{2}\mathrm{ToeplitzExt}_{n_m,n_l}(C_{ll}^{\langle k \rangle,\mathrm{na}}) & i = m, j= l,\\
0 & \text{others} .
\end{cases}
\]

\[
S= \left( \begin{array}{ccccccc} 
I_{n_1}&&&&&& \\ 
&\ddots&&&&& \\
&&I_{n_l}&\cdots&-\frac{1}{2}S_{lm}&& \\
&&\vdots&\ddots&\vdots& \\
&&-\frac{1}{2}S_{ml}&\cdots&I_{n_m}&& \\
&&&&&\ddots& \\
&&&&&&I_{n_z}
\end{array} \right).
\]
Consider the product
\begin{align}
S^{\boxstar}CS=:B=[B_{ij}].
\end{align}
We analyze the blocks $B_{ll},B_{mm},B_{lm}, B_{ml}$. 

\begin{align}  
B_{ll}&=C_{ll}-\frac{1}{2}tI_{lm}^{[k]\pm}S_{ml}+\frac{1}{2}t(S_{ml})^{\boxstar}(I_{lm}^{[k]\pm})^{\boxstar}+\frac{1}{4}(S_{ml})^{\boxstar}C_{mm}S_{ml}, \\ 
B_{mm}&=C_{mm}-\frac{1}{2}t(S_{lm})^{\boxstar}I_{lm}^{[k]\pm}+\frac{1}{2}t(I_{lm}^{[k]\pm})^{\boxstar}S_{lm}+\frac{1}{4}(S_{lm})^{\boxstar}C_{ll}S_{lm},\\ \label{eq2}
B_{lm}&=tI_{lm}^{[k]\pm}-\frac{1}{2}C_{ll}S_{lm}-\frac{1}{2}(S_{ml})^{\boxstar}C_{mm}-t\frac{1}{4}(S_{ml})^{\boxstar}(I_{lm}^{[k]\pm})^{\boxstar}S_{lm}, \\
B_{ml}&=-t(I_{lm}^{[k]\pm})^{\boxstar}-\frac{1}{2}C_{mm}S_{ml}-\frac{1}{2}(S_{lm})^{\boxstar}C_{ll}+t\frac{1}{4}(S_{lm})^{\boxstar}I_{lm}^{[k]\pm}S_{ml}.
\end{align}
Let us analyze the first equation. For the second term we have
\[
-\frac{1}{2}tI_{lm}^{[k]\pm}S_{ml}=-\frac{1}{2}C_{ll}.
\]
For the third term we have from (\ref{conjugationofE})  and (\ref{conmmuteswithaltidentity}).
\[
\frac{1}{2}t(S_{ml})^{\boxstar}(I_{lm}^{[k]\pm})^{\boxstar}= -\frac{1}{2}C_{ll}.
\]
So we get that 
\[
B_{ll}=\frac{1}{4}(S_{ml})^{\boxstar}C_{mm}S_{ml}
\]
A similar calculation gives that the second equation becomes
\[
B_{mm}=\frac{1}{4}(S_{lm})^{\boxstar}C_{ll}S_{lm}.
\]
Notice that all the terms in both $B_{ll}$ and $B_{mm}$ are strictly upper triangular (recall that $c_{ll}^{(k)}=c_{mm}^{(k)}=0$), therefore the result of the product is a smaller upper triangular matrix than the original corresponding blocks of $C$. Next we analyze the blocks $B_{lm}, B_{ml}$. Notice that all the elements in the right hand side, except the first one in both $B_{lm}$ and $B_{ml}$ 
are again strictly upper triangular matrices because they all involve multiplication of strictly upper triangular matrices. We conclude that the matrix $B$ has the same properties as the initial matrix $C$, but the elements  $B_{ll}$ and $B_{mm}$ are smaller strictly upper triangular matrices (the first nonzero diagonal is smaller). We can repeat the procedure so far several times, until the resulting matrices $B_{ll}$ and $B_{mm}$ are zero. From this we can now assume, without loss of generality that $C_{ll}=C_{mm}=0$, so that the matrix $C$ is of the shape
 
\[
C=\left(\begin{array}{ccccccc} 
&&&&&&  \\ 
&*&&&&*& \\
&&0&\cdots&tI_{lm}^{[k]\pm}&& \\
&&\vdots&*&\vdots&& \\
&&-t(I_{lm}^{[k]\pm})^{\boxstar}&\cdots&0&& \\
&*&&&&*& \\
&&&&&&
\end{array} \right).
\]

 Now we eliminate the rest of elements in the $l$-th  and $m$-th rows and columns. As before define the matrices
 
\begin{align*}
    C_{mq}^{\langle k \rangle,\mathrm{na}}&:=tI^{\pm}_{s_{(m,q;k)}}C_{mq}^{\langle k \rangle},\\
    C_{lq}^{\langle k \rangle,\mathrm{na}} &:=tI^{\pm}_{s_{(l,q;k)}}C_{lq}^{\langle k \rangle}.        
\end{align*}
 Then we define the corresponding matrices 
  \begin{align*}
    \matrices{n_l}{n_q} \ni R_{lq}&:=\mathrm{ToeplitzExt}_{n_l,n_q}(C_{mq}^{\langle k \rangle,\mathrm{na}}) \\
    \matrices{n_m}{n_q} \ni R_{mq}&:=\mathrm{ToeplitzExt}_{n_m,n_q}(C_{lq}^{\langle k \rangle,\mathrm{na}})
\end{align*}

Notice the change of indices in the extension. From our initial conditions $C_{lm}\neq 0$, so we have $\lambda_{l}+\lambda_{m}=0$. Suppose that   $C_{lq}\neq 0$ we have $\lambda_{l}+\lambda_{q}=0$. Combining both of these equations we get $\lambda_{m}-\lambda_{q}=0$. So we can define the following matrix $R=[R_{ij}]$ satisfying condition (\ref{conditionblocksrealtoeplitz}) by
 
\[
R_{ij} :=
\begin{cases}
I_{n_i} & i = j,\\
-\mathrm{ToeplitzExt}_{n_l,n_j}(C_{mj}^{\langle k \rangle,\mathrm{na}})  & i =l, j\neq l,m,\\
-\mathrm{ToeplitzExt}_{n_m,n_j}(C_{lj}^{\langle k \rangle,\mathrm{na}}) & i = m, j\neq m,l ,\\
0 & \text{others} .
\end{cases}
\] 
If we expand this, the matrix looks like 
 \[
R=
\left(\begin{smallmatrix}
I_{n_1}&&&0&&&&0&&& \\
&\ddots&&\vdots&&0&&\vdots&&0&  \\ 
&&I_{n_{l-1}}&0&&&&0&&&  \\ 
-R_{l1}&\cdots&-R_{l\, l-1}&I_{n_l}&-R_{l\, l+1}&\cdots&-R_{l\, m-1}&0&-R_{l\, m+1}&\cdots&-R_{lp} \\
&&&0&I_{n_{l+1}}&&&0&&& \\
&0&&\vdots&&\ddots&&\vdots&&0& \\
&&&0&&&I_{n_{m-1}}&0&&& \\
-R_{m1}&\cdots&-R_{m\, l-1}&&-R_{m\, l+1}&\cdots&-R_{m\, m-1}&I_{n_m}&-R_{m\, m+1}&\cdots&-R_{mp} \\ \\
&&&0&&&&0&I_{n_{m+1}}&& \\
&0&&\vdots&&0&&\vdots&&\ddots& \\
&&&0&&&&0&&&I_{n_p}
\end{smallmatrix} \right).
\]    

Finally consider the product
\[
R^{\boxstar}CR,
\]
which gives our desired result. 
\end{proof}
Again, we immediately have the next Corollary. 
\begin{coro}\label{corollarytosecondlemma}
     Let $C=[C_{ij}]\in \dot{\mathcal{H}}_{\mathcal{J}_{N_{\mathbb{R}}}}^{K}$. Suppose that for some $C_{lm}$, with $l\neq m$ and $n_l=n_m$, we have $c_{lm} \neq 0$.  Also suppose that $c_{ll}=0$ and $c_{mm}= 0$.
   Then, there exists  $S \in\dot{\mathcal{T}}_{\mathcal{J}_{N_{\mathbb{R}}}}$ and $t\in \lbrace 1,-1\rbrace $ such that

  \[
S^{\boxstar}CS=
\left(\begin{array}{ccccccccccc}
&&&0&&&&0&&& \\
&*&&\vdots&&*&&\vdots&&*&  \\ 
&&&0&&&&0&&&  \\ 
0&\cdots&0&0&0&\cdots&0&tI_{n_l}^{\pm}&0&\cdots&0 \\
&&&0&&&&0&&& \\
&*&&\vdots&&*&&\vdots&&*& \\

&&&0&&&&0&&& \\
0&\cdots&0&-t(I_{n_l}^{\pm})^{\boxstar}&0&\cdots&0&0&0&\cdots&0 \\
&&&0&&&&0&&& \\
&*&&\vdots&&*&&\vdots&&*& \\
&&&0&&&&0&&&

\end{array} \right).
\]          
\end{coro}
\begin{proof}
    Follows easily from previous Lemma. 
\end{proof}

\begin{remark}\label{whentheblocksdonthavethesamesize}
    
Lemma $(\ref{secondlemma})$ starts with a block $C_{lm}$ with $n_l=n_m$ and from there we eliminated elements in the corresponding column and row.  We can ask why not  consider, for example,  the case $n_l > n_m$, but the procedure does not work as it is. Consider, for example, the second term in Equation $(\ref{eq2})$: $\frac{1}{2}t(S_{lm})^{\boxstar}I_{lm}^{[k]\pm}$. Notice it is a product of the form

 \[
\left(\begin{array}{c|c} 
0&*\\ 
\end{array} \right)\left(\begin{array}{c} 
*\\ \hline
0
\end{array} \right)
 \]
 (The $0$ parts here can have different sizes) So if the difference between $n_l$ and $n_m$ is big enough we have that $\frac{1}{2}t(S_{lm})^{\boxstar}I_{lm}^{[k]\pm}=0$.
 \end{remark}
 We consider next only one particular case of such a block with $n_l>n_m$, where the problem is avoided. It will be necessary for later results.  
 \begin{lemma}\label{special4dimensionlemma}
     Let $p=2$ (in Equation $(\ref{startingjordanformreal})$), that is $\mathcal{J}_{N_{\mathbb{R}}}=\mathcal{J}_{n_1}(\lambda_1)\oplus \mathcal{J}_{n_2}(\lambda_2)$. Consider a matrix $C=[C_{ij}]\in \dot{\mathcal{H}}_{\mathcal{J}_{N_{\mathbb{R}}}}^{ N_{\mathbb{R}}-1}$.  Suppose also that $n_1=n_2+1$ and that $c_{11}=c_{22}= 0$, but $c_{12}\neq 0$.
 Then, there exists  $S \in\dot{\mathcal{T}}_{\mathcal{J}_{N_{\mathbb{R}}}}$ and $t\in \lbrace 1,-1\rbrace $ such that
    \[
    S^{\boxstar}CS=\left(\begin{array}{c|c} 
0&tI_{12}^{[1]\pm}\\ \hline
-t(I_{12}^{[1]\pm})^{\boxstar}& 0
\end{array} \right)
    \]
 \end{lemma}
 \begin{proof}
This situation is not covered by neither Lemma \ref{firstlemma} nor Lemma \ref{secondlemma}, but still from Lemma \ref{step1} we can suppose without loss of generality that 
     \[
     C=\left(\begin{array}{c|c} 
C_{11}&tI_{12}^{[1]\pm}\\ \hline
-t(I_{12}^{[1]\pm})^{\boxstar}&C_{22}
\end{array} \right)
     \]
Similarly as in the previous Lemmas define the following matrices

\begin{equation*}  
 C_{11}^{\langle 2 \rangle,\mathrm{na}}:=tI^{\pm}_{n_1-1}C_{11}^{\langle 2 \rangle},
\end{equation*}
\begin{equation*}
     C_{22}^{\langle 2 \rangle,\mathrm{na}}:=tI^{\pm}_{n_2-1}C_{22}^{\langle 2 \rangle}.
\end{equation*}
Then define the corresponding matrices
\begin{align*}
    \matrices{n_1}{n_2} \ni S_{12}  &:=\mathrm{ToeplitzExt}_{n_1,n_2}(C_{22}^{\langle 2 \rangle,\mathrm{na}}), \\
     \matrices{n_2}{n_1} \ni S_{21} &:=\mathrm{ToeplitzExt}_{n_2,n_1}( C_{11}^{\langle 2 \rangle,\mathrm{na}})
\end{align*}

As in Lemma \ref{secondlemma}, define the matrix
\[
S=\left(\begin{array}{c|c} 
I_{n_1}&-\frac{1}{2}S_{12}\\ \hline
-\frac{1}{2}S_{21}&I_{n_2}
\end{array} \right),
\]
and calculate $S^{\boxstar} CS$. The fact that $n_1-n_2=1$ and $c_{11}=c_{22}=0$ avoids the problem mentioned in Remark \ref{whentheblocksdonthavethesamesize}. 
Therefore the rest of the arguments is the same as in Lemma \ref{secondlemma}. 
\end{proof}

The previous Lemmas give an idea of the shape of the quotient space in the right hand side of (\ref{bijectionquotient space}). We get a particularly nice description of such quotient space for the case of maximal possible rank in Proposition \ref{canonicalformsreal}. But before proving that Proposition we answer the following question mentioned at the end of previous Section: what Jordan normal forms $\mathcal{J}_{N_{\mathbb{R}}}$ admit solutions of a given rank $R$ for equation (\ref{transposepermutationequation})?    First in light of the previous results it will be convenient decompose the Jordan normal and ordering it as follows:
\begin{equation}\label{jordannormalforminspecialorderreal} 
\mathcal J_{N_{\mathbb R}}
=
\mathcal J_{\mathbb R}^{0}
\oplus
\mathcal J_{\mathbb R}^{\neq 0}.
\end{equation}
Where the terms are defined as follows
\[
\mathcal J_{\mathbb R}^{0}
:=
\bigoplus_{r=1}^{k_{0}}
\mathcal J_{n_{r}^{0}}(0),
\qquad
n_{1}^{0}
\geq
n_{2}^{0}
\geq
\cdots
\geq
n_{k_{0}}^{0}.
\]
\[
\mathcal J_{\mathbb R}^{\neq 0}
:=
\bigoplus_{\mu=1}^{s}
\mathfrak J_{\mu}^{\neq 0}.
\]
\[
\mathfrak J_{\mu}^{\neq 0}
:=
\mathfrak J_{\mu}^{+}
\oplus
\mathfrak J_{\mu}^{-}.
\]
We choose $\lambda_{\mu}>0$ for all $\mu$, then 
\[
\mathfrak J_{\mu}^{+}
:=
\bigoplus_{r=1}^{k_{\mu}^{+}}
\mathcal J_{n_{\mu,r}^{+}}(\lambda_{\mu}),
\qquad
n_{\mu,1}^{+}
\geq
n_{\mu,2}^{+}
\geq
\cdots
\geq
n_{\mu,k_{\mu}^{+}}^{+}.
\]
\[
\mathfrak J_{\mu}^{-}
:=
\bigoplus_{r=1}^{k_{\mu}^{-}}
\mathcal J_{n_{\mu,r}^{-}}(-\lambda_{\mu}),
\qquad
n_{\mu,1}^{-}
\geq
n_{\mu,2}^{-}
\geq
\cdots
\geq
n_{\mu,k_{\mu}^{-}}^{-}.
\]
We have the partial sizes 
\[
N_{\mu}^{+}
:=
\sum_{r=1}^{k_{\mu}^{+}}n_{\mu,r}^{+},
\qquad
N_{\mu}^{-}
:=
\sum_{r=1}^{k_{\mu}^{-}}n_{\mu,r}^{-}.
\] 
Then we have the total sizes 
\[
N_{\mathbb R}^{0}
:=
\sum_{r=1}^{k_{0}}n_{r}^{0},
\qquad
N_{\mathbb R}^{\neq 0}
=
\sum_{\mu=1}^{s}
\left(
N_{\mu}^{+}
+
N_{\mu}^{-}
\right).
\]

Of course we have $ N_{\mathbb R} = N_{\mathbb R}^{0} + N_{\mathbb R}^{\neq 0}$. For an eigenvalue $\lambda_{\mu}>0$, define the following vectors in  $\mathbb{R}^{\max(k_\mu^+,k_\mu^-)}$ as 
\begin{equation}\label{vectorsrealcase}
\begin{aligned} 
\vec L_{\lambda_{\mu}} &= (n_{\mu,1}^{+},\ldots,n_{\mu,k_{\mu}^{+}}^{+},0,\ldots,0), \\
\vec L_{-\lambda_{\mu}} &= (n_{\mu,1}^{-},\ldots,n_{\mu,k_{\mu}^{-}}^{-},0,\ldots,0).
\end{aligned}
\end{equation}
(The $0$ are added as needed to complete the size of the vectors). Let $d_M(\vec{L}_{\lambda_{\mu}}, \vec{L}_{-\lambda_{\mu}})$ be the Manhattan distance between these two vectors. Recall the Manhattan distance is just the sum of the absolute values of the differences of corresponding components, that is if we write $\vec{L}_{\lambda_{\mu}}=((l_{\lambda_{\mu}})_k)$ and   $\vec{L}_{-\lambda_{\mu}}=((l_{-\lambda_{\mu}})_k)$ then 
\[
d_M(\vec{L}_{\lambda_{\mu}}, \vec{L}_{-\lambda_{\mu}})=\sum_{k=1}^{{\max(k_{\mu}^+,k_{\mu}^-)}}\abs{(l_{\lambda_{\mu}})_k-(l_{-\lambda_{\mu}})_k}.
\]
Using the same notation as in \cite{complexclassification}, let $N_{\mathcal{J}_{N_{\mathbb{R}}}}(m, \lambda)$ denote the number of Jordan blocks of size $m$ of $\mathcal{J}_{N_{\mathbb{R}}}$ with  eigenvalue $\lambda \in \mathbb{R}$.  Now we define the following numbers
\begin{equation}\label{maximalranksrealcasepartial}
    \begin{aligned}
        \mathcal R_{\mathcal J_{\mathbb R}^{0}}
:=&  N_{\mathbb R}^{0}
-  \sum_{\ell\ \mathrm{odd}} \left(  N_{\mathcal J_{\mathbb R}^{0}}(\ell,0)\; \mathrm{mod}\;2
\right), \\
\mathcal R_{\mathfrak J_{\mu}^{\neq 0}} :=& N_{\mu}^{+} + N_{\mu}^{-} - d_M \left( \vec L_{\lambda_{\mu}},  \vec L_{-\lambda_{\mu}}
\right), \\
\mathcal R_{\mathcal J_{\mathbb R}^{\neq 0}} :=&
N_{\mathbb R}^{\neq 0} - \sum_{\mu=1}^{s} d_M \left( \vec L_{\lambda_{\mu}}, \vec L_{-\lambda_{\mu}} \right),  \\
     \end{aligned}
\end{equation}
We have the sum of this numbers
\begin{equation} \label{maximalrankreal}
     \mathcal{R}_{\mathcal{J}_{N_{\mathbb{R}}}}:=N_{\mathbb{R}}-\left( \sum_{\ell \;\text{}odd} {(N_{\mathcal{J}_{N_{\mathbb{R}}}}(\ell,0)}\;\mathrm{mod}\,2)  +  \sum_{\mu=1}^s d_M(\vec{L}_{\lambda_{\mu}}, \vec{L}_{-\lambda_{\mu}}) \right). 
\end{equation}
Now using this numbers define the following sets 

\begin{equation}\label{possibleranksrealsets}
\begin{aligned}
R_{\mathcal J_{\mathbb R}^{0}} &:= \left\{ r \;\middle|\; r\in 2\mathbb Z,\ 0\leq r\leq \mathcal R_{\mathcal J_{\mathbb R}^{0}} \right\} = \left\{
0,2,\ldots,\mathcal R_{\mathcal J_{\mathbb R}^{0}} \right\}, 
\\
R_{\mathfrak J_{\mu}^{\neq 0}}
&:=
\left\{
r \;\middle|\; r\in 2\mathbb Z,\ 0\leq r\leq
\mathcal R_{\mathfrak J_{\mu}^{\neq 0}}
\right\}
=
\left\{
0,2,\ldots,\mathcal R_{\mathfrak J_{\mu}^{\neq 0}}
\right\},
\\
R_{\mathcal J_{\mathbb R}^{\neq 0}}
&:= \sum_{\mu=1}^{s}
R_{\mathfrak J_{\mu}^{\neq 0}}= 
\left\{
r \;\middle|\; r\in 2\mathbb Z,\ 0\leq r\leq
\mathcal R_{\mathcal J_{\mathbb R}^{\neq 0}}
\right\}
= \left\{ 0,2,\ldots,\mathcal R_{\mathcal J_{\mathbb R}^{\neq 0}}
\right\}, \\ 
R_{\mathcal J_{N_{\mathbb R}}}
&:= R_{\mathcal J_{\mathbb R}^{0}}
+R_{\mathcal J_{\mathbb R}^{\neq 0}}=\left\{
r \;\middle|\; r\in 2\mathbb Z,\ 0\leq r\leq
\mathcal R_{\mathcal{J}_{N_{\mathbb{R}}}}
\right\}=\left\{ 0,2,\ldots,\mathcal R_{\mathcal{J}_{N_{\mathbb{R}}}}
\right\}.
\end{aligned}
\end{equation}

If the Jordan normal form is given in this order, every matrix in
$\dot{\mathcal H}^{K}_{\mathcal J_{N_{\mathbb R}}}$ is block
diagonal, with components corresponding to the zero eigenvalue and
to each pair $\pm\lambda_{\mu}$. Therefore, when studying the ranks
of the solutions, we may consider the different components of
\eqref{jordannormalforminspecialorderreal} separately.

Before that we state next Lemma, that states the existence of solutions with a simple form, where in each row and corresponding column there is only one nonzero diagonal block. This will help us to think and simplify about the later lemmas. 
\begin{lemma}\label{simpleformofsolutions}
    If $A=[A_{ij}] \in \dot{\mathcal{H}}_{\mathcal{J}_{N_{\mathbb{R}}}}^{K} $ , then there also exist $B=[B_{ij}] \in \dot{\mathcal{H}}_{\mathcal{J}_{N_{\mathbb{R}}}}^{K} $ with the following properties
    \begin{enumerate}
    \item For each $i$ there exists at most one $j$ such that $B_{ij}\neq 0$,  for all other $l\neq j$, $B_{il}= 0$.
    \item For each $B_{ij}\neq 0$ with $i\leq j$, we have
    \[
B_{ij}=I_{ij}^{[k]\pm},
\qquad
B_{ji}=-\left(B_{ij}\right)^{\boxstar},
\]
for some appropriate $k$.        
    \end{enumerate}
    \end{lemma}

    \begin{proof}
Delete all zero scalar rows of $A$ and the corresponding scalar columns,
and denote the resulting matrix by $B=[B_{ij}]$. Since only zero rows and
columns are deleted,
\[
\operatorname{rank}B=\operatorname{rank}A=K.
\]
Also, 
\[
B\in\dot{\mathcal H}_{\mathcal J'_{N'}}^{K},
\]
where $N'\leq N$ and $\mathcal J'_{N'}$ is obtained from
$\mathcal J_{N_{\mathbb R}}$ by decreasing, if necessary, the sizes of some
Jordan blocks.

If $B=0$, then $K=0$, and there is nothing to prove. Suppose that $B\neq0$. If one of its diagonal blocks has nonzero main coefficient, we apply Corollary \ref{corollarytofirstlemma} to that block. Otherwise, all the diagonal main coefficients are zero. In the latter case, choose a block row of maximal size. Since $B$ has no zero scalar rows or columns, this block row contains a square off-diagonal block $B_{ij}$, $i\neq j$, with nonzero main coefficient. Since $b_{ii}=b_{jj}=0$, Corollary \ref{corollarytosecondlemma} applies to the pair of block rows and columns indexed by $i$ and $j$.

In either case, the corresponding block row and column, or pair of block
rows and columns, is isolated from the remaining matrix. We then consider
the complementary block submatrix. This submatrix may contain zero scalar
rows and corresponding columns and we delete them and repeat the preceding
argument. At each step, at least one block row or column is removed from consideration. Therefore, the procedure ends after finitely many steps. We obtain a matrix $C$ of rank $K$ such that every nonzero block row contains exactly one nonzero block and every nonzero block is of the form
\[
t I_m^\pm,
\qquad
t\in\{-1,1\},
\]
for some positive integer $m$.

Finally, we reinsert all the deleted zero rows and columns in their original
positions, extending the blocks of $C$ back to the original block
sizes of $\mathcal J_{N_{\mathbb R}}$. Every nonzero
reduced block $tI_m^\pm$ becomes a block of the form
\[
tI_{ij}^{[k]\pm}
\]
for an appropriate $k$. We obtain in the end a matrix
\[
D=[D_{ij}]\in\dot{\mathcal H}_{\mathcal J_{N_{\mathbb R}}}^{K}.
\]

Since only zero rows and columns have been inserted,
\[
\operatorname{rank}D=K.
\]

By construction, every block row of $D$ contains only one nonzero block and has almost the required form except for the constant $t$. Since only the existence of a matrix of rank $K$ with the stated form is
required, we may replace each nonzero block $tI_{ij}^{[k]\pm}$, for all $i\leq j$, by
$I_{ij}^{[k]\pm}$, making the corresponding replacement in the opposite
block according to the $N$-skew-symmetry relation. These replacements give us a matrix with the required form. 
\end{proof}

Therefore when exploring existence of a certain rank questions, we only need to consider matrices of such  form. First, we study the block $\mathcal J_{\mathbb R}^{0}$ of the decomposition \eqref{jordannormalforminspecialorderreal}.

\begin{lemma} \label{ranksforzeroeigenvaluesreal}
We have that 
\[
\dot{\mathcal{H}}_{\mathcal J_{\mathbb R}^{0}}^{R}\neq \emptyset
\]
if and only if $R \in  R_{\mathcal J_{\mathbb R}^{0}} $. 
\end{lemma}
\begin{proof}
For simplicity, throughout this proof we write
\[
n_i:=n_i^0,
\qquad
1\leq i\leq k_0.
\]
     First we show that
     \[
     \dot{\mathcal{H}}_{\mathcal J_{\mathbb R}^{0}}^{\mathcal R_{\mathcal J_{\mathbb R}^{0}}}\neq \emptyset.
     \]
    We will construct a block matrix $A=[A_{ij}]_{1\leq i,j\leq k_0}$ in this set. We use the notation introduced at the beginning of this section . We will construct it with the properties described in Lemma $\ref{simpleformofsolutions}$, so that we only define one nonzero block in each row and such nonzero blocks will only have one nonzero diagonal.  For all $i$ such that $n_i$ is even, just make $A_{ii}=I_{ii}^{[1]\pm}$.  Next consider an $i$ such that $n_i$ is odd. In such case recall that (from Remark \ref{diagonalsarezero}) we have $a_{ii}$ must be $0$. If you can find $j\neq i$ with $n_j=n_i$,  define $A_{ij}=I_{ij}^{[1]\pm}$ and also the corresponding $A_{ji}$ according to the N-skew-symmetry relation. If such a $j$ does not exist, make $A_{ii}=I_{ii}^{[2]\pm}$. Notice that in this last case the corresponding odd dimension diagonal block has rank $n_i-1$.  Repeat this procedure for all block rows not yet considered.  We obtain at the end a matrix of rank  $\mathcal R_{\mathcal J_{\mathbb R}^{0}}$, with the shape described in Lemma $\ref{simpleformofsolutions}$.
    
 Next we construct matrices with all the rest of smaller ranks in  $R_{\mathcal J_{\mathbb R}^{0}}$. The idea is to \say{slide} up matrix of a given rank to obtain a new one with smaller rank.  Let 
 \[
 C=[C_{ij}] \in \dot{\mathcal{H}}_{\mathcal J_{\mathbb R}^{0}}^{K}
 \]
 for some $K$, be as in Lemma \ref{simpleformofsolutions}. Consider a block row of $C$ that has a nonzero block. Let it be the $i$-th row. Such a nonzero block has the shape $I_{ij}^{[k]\pm}$ for some $k$. Define a new matrix by making the change 

\[
I_{ij}^{[k]\pm} \to \begin{cases}I_{ij}^{[k+2]\pm} &j=i, \\
I_{ij}^{[k+1]\pm}&j\neq i.
    
\end{cases}
\]

In the latter case do the appropriate change in the corresponding column too, according to the N-skew-symmetry relation.  Call this resulting matrix $C^{\prime}$.  $C^{\prime}$ is still  $N$-skew-symmetric $N$-upper alternating Toeplitz and satisfies (\ref{conditionblockrealankel}).  Also it has exactly two extra zero rows (and corresponding columns), so it has rank $K-2$. So we have
     \[
     C^{\prime} \in \dot{\mathcal{H}}_{\mathcal J_{\mathbb R}^{0}}^{K-2}
     \]
      Repeat the procedure if necessary. The odd ranks are not possible immediately from $N$-skew-symmetry.

      Finally we show that 
      \[
     \dot{\mathcal{H}}_{\mathcal J_{\mathbb R}^{0}}^{R}= \emptyset,
     \]
     for all $R>\mathcal R_{\mathcal J_{\mathbb R}^{0}}$. Again the odd ranks are immediately not possible. Let 
     \[
     B \in \dot{\mathcal{H}}_{\mathcal J_{\mathbb R}^{0}}^{S},
     \]
     for some $S\geq\mathcal R_{\mathcal J_{\mathbb R}^{0}}$. Suppose without loss of generality that  it has the shape given in Lemma \ref{simpleformofsolutions}. We show that $S= \mathcal{R}_{{\mathcal J_{\mathbb R}^{0}}}$. The idea is that every time an odd block cannot be \say{paired} with another odd block of the same size, it always decreases the rank of the matrix.  We will transform the matrix $B$ in several steps, without decreasing the rank, until we obtain a matrix similar to the matrix constructed at the beginning of this proof. After each transformation we will still refer to the matrix as $B$, so as to not introduce more notation.  First we make each nonzero block \say{as big as possible}. Consider a nonzero block $B_{ij}$ such that $B_{ij}= I_{ij}^{[k]\pm}$, for some $k$. Substitute such block using the following rules
    \begin{itemize}
        \item If $i=j$ and $n_i$ is odd instead make $B_{ii}=I_{ii}^{[2]\pm} $ (recall once again that $b_{ii}=0$ necessarily for odd sized blocks in the diagonal). For the case $n_i=1$ just leave the block as zero matrix. 
        \item For the rest of cases make $B_{ij}=I_{ij}^{[1]\pm}$.
    \end{itemize}
     The resulting matrix $B$ has bigger or equal rank. Now for this new matrix $B$ consider all the nonzero blocks $B_{ij}$ with $i<j$, except those such that $n_i=n_j$ odd,  make both $B_{ij}=B_{ji}=0$. Then make $B_{ii}=I_{ii}^{[1]\pm}$ if $n_i$ is even or $B_{ii}=I_{ii}^{[2]\pm}$ if $n_i$ is odd, similarly make $B_{jj}=I_{jj}^{[1]\pm}$ if $n_j$ is even or $B_{jj}=I_{jj}^{[2]\pm}$ if $n_j$ is odd.  Again the resulting matrix $B$ has the same or bigger rank. For this new matrix $B$, consider a nonzero block $B_{ii}$, with $n_i$ odd. If there is another nonzero block $B_{jj}$, $j\neq i$ , but $n_j=n_i$, make both $B_{ii}=B_{jj}=0$ and instead make $B_{ij}=I_{ij}^{[1]\pm}$ and define $B_{ji}$ according to the N-skew-symmetry relation . Repeat the procedure for all  the rest of such pair of  blocks you can find. The resulting matrix is again of the same or bigger rank. Moreover, it is of the same shape (up to some conjugation of the blocks), and therefore it has the same rank as the matrix $A$ constructed at the beginning of this proof.   
    
\end{proof}
Next we study a block $\mathfrak J_{\mu}^{\neq 0}$ of the decomposition \eqref{jordannormalforminspecialorderreal}.
\begin{lemma} \label{ranksfornonzeroeigenvaluesreal}
We have that 
\[
\dot{\mathcal{H}}_{\mathfrak J_{\mu}^{\neq 0}}^{R}\neq \emptyset
\]
if and only if $R \in  R_{\mathfrak J_{\mu}^{\neq 0}} $. 
\end{lemma}
\begin{proof}
The proof follows the same order as previous Lemma. 
For simplicity, for this proof we enumerate the Jordan blocks in
$\mathfrak J_\mu^{\neq0}$ consecutively and write
\[
n_i:=
\begin{cases}
n_{\mu,i}^{+},
&1\leq i\leq k_\mu^{+},\\[1mm]
n_{\mu,i-k_\mu^{+}}^{-},
&k_\mu^{+}<i\leq k_\mu^{+}+k_\mu^{-}.
\end{cases}
\]
So the indices $1,\ldots,k_\mu^{+}$ correspond to the blocks with eigenvalue $\lambda_\mu$ and the indices $
k_\mu^{+}+1,\ldots,k_\mu^{+}+k_\mu^{-}
$ correspond to the blocks with eigenvalue $-\lambda_\mu$.

First we show that
     \[
\dot{\mathcal H}_{\mathfrak J_{\mu}^{\neq 0}}^{
\mathcal R_{\mathfrak J_{\mu}^{\neq 0}}
}
\neq\emptyset.
\]
    We will construct a block matrix $A=[A_{ij}]_{1\leq i,j\leq k_\mu^{+}+k_\mu^{-}}$ in this set. Again we suppose it has the shape of Lemma $\ref{simpleformofsolutions}$, so that we only define one nonzero block in each row and such nonzero blocks will only have one nonzero diagonal. 
    For each $i$ such that
$1\leq i\leq\min(k_\mu^+,k_\mu^-)$, define $A_{i\,k_\mu^++i}=I_{i\,k_\mu^++i}^{[1]\pm}$ and the corresponding $A_{k_\mu^++i\,\,i}$  according to the N-skew-symmetry relation. 

Each such pairing introduces
\[
\abs{n_i-n_{k_\mu^++i}}=\left|n_{\mu,i}^+-n_{\mu,i}^-\right|
\]
zero rows. All blocks not paired in the previous way make them zero. The resulting rank of the matrix A is $\mathcal R_{\mathfrak J_{\mu}^{\neq 0}}$. 
    
 Next we construct matrices with all the rest of smaller ranks in  $ R_{\mathfrak J_{\mu}^{\neq 0}}$. Here the idea is again \say{slide} up matrix of a given rank to obtain a new one with smaller rank.  As the proof of this is exactly the same as in the previous Lemma we omit it.
 
       Finally we show that 
             
      \[
     \dot{\mathcal H}_{\mathfrak J_{\mu}^{\neq 0}}^{R}= \emptyset,
     \]
     for all $R>\mathcal R_{\mathfrak J_{\mu}^{\neq 0}}$. Let 
     \[
     B \in \dot{\mathcal{H}}_{\mathfrak J_{\mu}^{\neq 0}}^{S},
     \]
     for some $S\geq\mathcal R_{\mathfrak J_{\mu}^{\neq 0}}$. Suppose without loss of generality that  it has the shape given in Lemma \ref{simpleformofsolutions}. We show that $S= \mathcal{R}_{\mathfrak J_{\mu}^{\neq 0}}$. The idea is that the \say{pairing} constructed at the beginning for the matrix $A$ at the beginning of this proof gives the maximum possible rank. We will transform the matrix $B$ in several steps, without decreasing the rank, until we obtain a matrix similar to the matrix constructed at the beginning of this proof. After each transformation we will still refer to the matrix as $B$, so as to not introduce more notation.  First make each nonzero block as big as possible. That is if for a nonzero block $B_{ij}$, we have $B_{ij}=I_{ij}^{[k]\pm}$, for some $k>1$, substitute it for $B_{ij}=I_{ij}^{[1]\pm}$ (recall in this case we necessarily have that for nonzero blocks $i\neq j$). The resulting matrix $B$ has equal or higher rank.  Any such resulting matrix $B$ is formed by \say{pairing} blocks corresponding to a couple of positive-negative eigenvalues as these are the only blocks that can be nonzero. That is if $B_{ij}\neq 0 $ with $i<j$, then $1\leq i\leq k_{\mu}^+$ and $k_{\mu}^+ < j\leq k_{\mu}^++k_{\mu}^-$. There are of course blocks that will be left without a pair, such block rows will be necessarily zero.  Every nonzero block $B_{ij}$ introduces necessarily $\abs{n_i-n_j}$ zero rows. Therefore the rank of this matrix $B$ can be described by 
  \[
  N_{\mu}^{+} + N_{\mu}^{-} - d_M \left( \vec L_{\lambda_{\mu}},  P\vec L_{-\lambda_{\mu}}
\right)
  \]
  for some $P\in \mathrm{Per}(\max(k_{\mu}^+,k_{\mu}^-),\mathbb{R})$.  So the matrix with highest rank we could have is   
  \[
  N_{\mu}^{+} + N_{\mu}^{-} -\min_{P\in \mathrm{Per}(\max(k_{\mu}^+,k_{\mu}^-),\mathbb{R})} d_M(\vec{L}_{\lambda_{\mu}}, P\vec{L}_{-\lambda_{\mu}})
  \]
    It is well known that this is achieved when the vectors are paired in descending order, that is exactly when $P=I$, which is the previous construction.  This completes our proof.
\end{proof}
Now we can combine the previous two Lemmas and give the first main result of this section. 

\begin{prop} \label{sufficientneccesaryconditionsexistence}
  We have that 
\[
\dot{\mathcal{H}}_{\mathcal J_{N_{\mathbb R}}}^{R}\neq \emptyset
\]
if and only if $R \in  R_{\mathcal J_{N_{\mathbb R}}} $.
\end{prop}
\begin{proof}
    As mentioned before if $A \in \dot{\mathcal{H}}_{\mathcal J_{N_{\mathbb R}}}^{R}$, then $A$ is a block diagonal matrix with respect to the decomposition  \eqref{jordannormalforminspecialorderreal}. Therefore, we just need to consider the Minkowski sum of the possible ranks obtained in the previous Lemmas over all possible eigenvalues, and this gives our desired result. 
    
\end{proof}
If we fix $R$, several matrices ${\mathcal{J}_{N_{\mathbb{R}}}}$ can satisfy $\dot{\mathcal{H}}_{\mathcal J_{N_{\mathbb R}}}^{R}\neq \emptyset$. We describe those possible $\mathcal{J}_{N_{\mathbb{R}}}$, when $R$ is maximal.  
\begin{coro}\label{maximaldimensioncaseevenreal}
     Let $N_{\mathbb{R}}$ be even. $\dot{\mathcal{H}}_{\mathcal{J}_{N_{\mathbb{R}}}}^{N_{\mathbb{R}}}\neq \emptyset$ if and only if  
       \begin{enumerate}
    \item $N_{\mathcal{J}_{N_{\mathbb{R}}}}(m,0)$ is even for all odd $m\in 
            \mathbb{N}^+$.
            \item $N_{\mathcal{J}_{N_{\mathbb{R}}}}(m, \lambda)
=N_{\mathcal{J}_{N_{\mathbb{R}}}}(m, -\lambda)$ for all
$\lambda \in \mathbb{R}\setminus\{0\}$ and all $m\in\mathbb N^+$.   
     \end{enumerate}   
     \end{coro}
\begin{proof}
  From Proposition \ref{sufficientneccesaryconditionsexistence} we get 
    \begin{equation*}
    \begin{aligned}          
       \sum_{l \;\text{}odd} {(N_{\mathcal{J}_{N_{\mathbb{R}}}}(l,0)}\;\mathrm{mod}\,2)  +  \sum_{i} d_M(\vec{L}_{\lambda_i}, \vec{L}_{-\lambda_i})=0.
        \end{aligned}
    \end{equation*}
All the terms involved are positive, so we get that for all odd $l$ and all $i$    
\begin{enumerate}
    \item $N_{\mathcal{J}_{N_{\mathbb{R}}}}(l,0)\;\mathrm{mod}\,2=0$,
    \item $d_M(\vec{L}_{\lambda_i}, \vec{L}_{-\lambda_i})=0$,
    \end{enumerate}
These two correspond, in order, to the conditions stated in this corollary. 

\end{proof}

\begin{coro}\label{maximaldimensioncaseoddreal}
    Let $N_{\mathbb{R}}$ be odd and $R=N_{\mathbb{R}}-1$, then $\dot{\mathcal{H}}_{\mathcal{J}_{N_{\mathbb{R}}}}^{N_{\mathbb{R}}-1}\neq \emptyset$ if and only if  one of the following
    
    \begin{enumerate}
            \item $N_{\mathcal{J}_{N_{\mathbb{R}}}}(m, \lambda)=N_{\mathcal{J}_{N_{\mathbb{R}}}}(m, -\lambda)$ for all $\lambda \in \mathbb{R}\setminus\lbrace 0 \rbrace$, $m\in \mathbb{N}^+$  and there exists only one odd $l\in \mathbb{N}^+$ such that $N_{\mathcal{J}_{N_{\mathbb{R}}}}(l,0)$ is odd. 
            \item $N_{\mathcal{J}_{N_{\mathbb{R}}}}(m,0)$ is even for all odd $m\in 
            \mathbb{N}^+$. One of the following:
            \begin{enumerate}
                
                \item There exists a unique $\alpha \in \mathbb R\setminus\{0\}$ such that $N_{\mathcal J_{N_{\mathbb R}}}(1,\alpha)- N_{\mathcal J_{N_{\mathbb R}}}(1,-\alpha)=1$.
For every pair $(m,\lambda)\notin\{(1,\alpha),(1,-\alpha)\}$ with $\lambda\in\mathbb R\setminus\{0\}$, we have 
\[
N_{\mathcal J_{N_{\mathbb R}}}(m,\lambda)
=
N_{\mathcal J_{N_{\mathbb R}}}(m,-\lambda).
\]

                 \item  There exist a unique $l\in \mathbb{N}^+$ and a unique $\alpha \in \mathbb{R}\setminus\lbrace 0 \rbrace$  such that $N_{\mathcal{J}_{N_{\mathbb{R}}}}(l,\alpha)-N_{\mathcal{J}_{N_{\mathbb{R}}}}(l,-\alpha)=1$ and $N_{\mathcal{J}_{N_{\mathbb{R}}}}(l+1,-\alpha)-N_{\mathcal{J}_{N_{\mathbb{R}}}}(l+1,\alpha)=1$.  For every pair $(m,\lambda)$ different from $(l,\alpha)$, 
$(l+1,\alpha)$, $(l,-\alpha)$ and 
$(l+1,-\alpha)$ with $\lambda\in\mathbb R\setminus\{0\}$, we have $N_{\mathcal J_{N_{\mathbb R}}}(m,\lambda)
=
N_{\mathcal J_{N_{\mathbb R}}}(m,-\lambda).$   
                \end{enumerate} 
        \end{enumerate}
    \end{coro}

\begin{proof}
From Proposition \ref{sufficientneccesaryconditionsexistence} we get the condition $ N_{\mathbb{R}}-1\leq \mathcal{R}_{\mathcal{J}_{N_{\mathbb{R}}}} \leq  N_{\mathbb{R}}$. As $\mathcal{R}_{\mathcal{J}_{N_{\mathbb{R}}}}$ is even then the condition becomes 
    \begin{equation*}
        \sum_{l \;\text{}odd} {(N_{\mathcal{J}_{N_{\mathbb{R}}}}(l,0)}\;\mathrm{mod}\,2) + \sum_{i} d_M(\vec{L}_{\lambda_i}, \vec{L}_{-\lambda_i}) = 1.
    \end{equation*}         
       We get that one of the following must hold
                \begin{enumerate}
              \item \[\sum_{l \;\text{}odd} {(N_{\mathcal{J}_{N_{\mathbb{R}}}}(l,0)}\;\mathrm{mod}\,2)=1,  \;\;  \sum_{i} d_M(\vec{L}_{\lambda_i}, \vec{L}_{-\lambda_i})=0.\]
            
             \item \[\sum_{l \;\text{}odd} {(N_{\mathcal{J}_{N_{\mathbb{R}}}}(l,0)}\;\mathrm{mod}\,2)=0,  \;\;  \sum_{i} d_M(\vec{L}_{\lambda_i}, \vec{L}_{-\lambda_i})=1.\]      
            
          \end{enumerate}
      
       These correspond in order to the conditions $(1),(2)$ in the statement of this corollary.  Notice that there is only one way to satisfy
       \[\sum_{l \;\text{}odd} {(N_{\mathcal{J}_{N_{\mathbb{R}}}}(l,0)}\;\mathrm{mod}\,2)=1.\]
  However there are two possible ways to satisfy 
  \[
  \sum_{i} d_M(\vec{L}_{\lambda_i}, \vec{L}_{-\lambda_i})=1,
  \]
  which correspond to the conditions $(2-a)$ and $(2-b)$ in the statement of this corollary.
\end{proof}
\begin{example} {\label{examplereal}}
    Consider the case $N_{\mathbb{R}}=5$ and 
    \[
    \mathcal{J}_5= \mathcal{J}_3(\lambda_1)\oplus \mathcal{J}_2(-\lambda_1), \;\; \lambda_1 \neq 0.
    \]
    For this Jordan normal form we have $\mathcal{R}_{\mathcal{J}_{N_{\mathbb{R}}}}=4$. Therefore according to Proposition $\ref{sufficientneccesaryconditionsexistence}$, we must have  $\dot{\mathcal{H}}_{\mathcal{J}_{N_{\mathbb{R}}}}^{R}\neq \emptyset$ for all even $R\leq 4$. Indeed we can have that 
\[
A_1=\left(\begin{array}{ccc|cc} 
0&0&0&1&0\\ 
0&0&0&0&-1\\ 
0&0&0&0&0\\ \hline 
0&1&0&0&0\\
0&0&-1&0&0
\end{array} \right) \in \dot{\mathcal{H}}_{\mathcal{J}_{N_{\mathbb{R}}}}^{ 4}.
\]
By using the procedure we used in Lemma ($\ref{ranksfornonzeroeigenvaluesreal}$) of \say{sliding up} we obtain the matrix 
\[
A_2=
\left(\begin{array}{ccc|cc} 
0&0&0&0&1\\ 
0&0&0&0&0\\ 
0&0&0&0&0\\ \hline 
0&0&-1&0&0\\
0&0&0&0&0
\end{array} \right) \in \dot{\mathcal{H}}_{\mathcal{J}_{N_{\mathbb{R}}}}^{2}.
\]

\end{example}

Finally, we show when the rank is maximal the quotient space in the right side of \eqref{bijectionquotient space} is finite and has a very nice description.
\begin{prop}\label{canonicalformsreal}
Let $C=[C_{ij}]\in \dot{\mathcal{H}}_{\mathcal{J}_{N_{\mathbb{R}}}}^{N_{\mathbb{R}}}$ if $N_{\mathbb{R}}$ is even or $C=[C_{ij}]\in \dot{\mathcal{H}}_{\mathcal{J}_{N_{\mathbb{R}}}}^{ N_{\mathbb{R}}-1}$ if $N_{\mathbb{R}}$ is odd. Then there exists $S \in\dot{\mathcal{T}}_{\mathcal{J}_{N_{\mathbb{R}}}}$ and a permutation matrix $P \in \mathrm{Per}(N_{\mathbb{R}},\mathbb{R})$ such that
\[
S^{\boxstar}CS=\begin{cases}
\mathcal{P}_{N_{\mathbb{R}}}P^tJ_{N_{\mathbb{R}}}P  &\text{if} \; {N_{\mathbb{R}}} \; \text{is even}, \\
\mathcal{P}_{N_{\mathbb{R}}}P^tJ_{({N_{\mathbb{R}}},{N_{\mathbb{R}}}-1)}P  & \text{if} \; {N_{\mathbb{R}}}\; \text{is odd}.
\end{cases}
\]
\end{prop}
\begin{proof}
Suppose first that $N_{\mathbb R}$ is even. Then
$\mathcal J_{N_{\mathbb R}}$ is as described in Corollary
\ref{maximaldimensioncaseevenreal}, and $C$ is nonsingular.

If one of the diagonal blocks has nonzero main coefficient, we apply
Corollary \ref{corollarytofirstlemma} to that block. Otherwise, all
the diagonal main coefficients are zero. In this case, choose a block
row of maximal size. There must exist in this block row a nonzero
square off-diagonal block $C_{ij}$, with $i<j$, whose main coefficient
is nonzero. If no such block existed, then, from the shape of
the blocks in
$\dot{\mathcal H}_{\mathcal J_{N_{\mathbb R}}}$, the block row would
contain a zero scalar row. The matrix $C$ would then have a zero row,
contradicting its nonsingularity. Since all the diagonal main
coefficients are zero, we have $c_{ii}=c_{jj}=0$, and hence
Corollary \ref{corollarytosecondlemma} applies to $C_{ij}$. In either case, the corresponding block row and column, or pair of
block rows and columns, is isolated from the remaining part of the
matrix, in the sense that subsequent congruence transformations
do not affect those previously isolated block rows and columns.

If some block rows and columns have not yet been isolated, we repeat
the same argument for those rows and columns.  Then after finitely many steps we obtain
$S_1,\ldots,S_r\in
\dot{\mathcal T}_{\mathcal J_{N_{\mathbb R}}}$. Set
$S:=S_1\cdots S_r$ and $C':=S^{\boxstar}CS$. Then
$S\in\dot{\mathcal T}_{\mathcal J_{N_{\mathbb R}}}$, and every
scalar row and column of $\mathcal P_{N_{\mathbb R}}C'$ contains
exactly one nonzero entry. Therefore,
\[
\left|\mathcal P_{N_{\mathbb R}}C'\right|
\]
is a permutation matrix. Also $\mathcal P_{N_{\mathbb R}}C'$ is skew-symmetric. Therefore, Proposition \ref{skewsymmetricpermutation} gives a
permutation matrix
$P\in\operatorname{Per}(N_{\mathbb R},\mathbb R)$ such that
\[
\mathcal P_{N_{\mathbb R}}C'
=
P^tJ_{N_{\mathbb R}}P.
\]
It follows that
\[
S^{\boxstar}CS
=
\mathcal P_{N_{\mathbb R}}P^tJ_{N_{\mathbb R}}P.
\]

Now suppose that $N_{\mathbb R}$ is odd. We apply the same procedure
as in the even-dimensional case for as long as one of Corollaries
\ref{corollarytofirstlemma} and
\ref{corollarytosecondlemma} can be applied. More precisely, if one
of the remaining diagonal blocks has nonzero main coefficient, we
apply Corollary \ref{corollarytofirstlemma}. Otherwise, all the
remaining diagonal main coefficients are zero, and we choose, among
the remaining block rows, one of maximal size. If this block row
contains a square off-diagonal block with nonzero main coefficient,
then its corresponding diagonal main coefficients are zero, and
Corollary \ref{corollarytosecondlemma} applies.

As in the even-dimensional case, every successive
congruence transformation leaves the previously isolated components
unchanged. Moreover, every component isolated by the two corollaries
is nonsingular. Since $\operatorname{rank}C=N_{\mathbb R}-1$, the
part that remains after these reductions has rank deficiency one.

By Corollary \ref{maximaldimensioncaseoddreal}, the Jordan matrix
$\mathcal J_{N_{\mathbb R}}$ satisfies one of the conditions
$(1)$, $(2\text{-a})$, or $(2\text{-b})$. All the matched
full-rank components occurring in these conditions have already been
isolated by the preceding procedure. So only the
exceptional block or blocks described in the corresponding condition
can remain.

Suppose first that condition $(1)$ holds. Then the remaining part
consists of one zero-eigenvalue Jordan block
$\mathcal J_{n_r}(0)$ of odd size. Its corresponding diagonal block
has rank $n_r-1$. Since $n_r$ is odd, its main coefficient satisfies
$c_{rr}=0$, and the maximality of its rank implies
$c_{rr}^{(2)}\neq0$. Therefore, Lemma \ref{firstlemma} applies to
this block and transforms it into
$\varepsilon I_{rr}^{[2]\pm}$, where
$\varepsilon\in\{-1,1\}$.

Suppose next that condition $(2\text{-a})$ holds. After all the
matched blocks have been isolated, there remains one unmatched
one-dimensional Jordan block with eigenvalue $\alpha\neq0$. By
condition \eqref{conditionblockrealankel}, the corresponding scalar
row and column are zero. They account for the unique rank deficiency,
so no further transformation is necessary.

Finally, suppose that condition $(2\text{-b})$ holds. After all the
matched blocks have been isolated, there remain two Jordan blocks of
sizes $l$ and $l+1$ with opposite nonzero eigenvalues. After
reversing the order of these two blocks if necessary, this is
precisely the situation of Lemma
\ref{special4dimensionlemma}. Applying that lemma, with its
congruence matrix completed by identity blocks on all the previously
isolated components, gives the required form of the remaining part.

Consequently, similarly to the even-dimensional case, in every case
there exists
$S\in\dot{\mathcal T}_{\mathcal J_{N_{\mathbb R}}}$ such that, for
$C':=S^{\boxstar}CS$, the matrix
\[
\left|\mathcal P_{N_{\mathbb R}}C'\right|
\]
is a partial permutation matrix of rank $N_{\mathbb R}-1$.

Furthermore, $\mathcal P_{N_{\mathbb R}}C'$ is skew-symmetric.
Therefore, Proposition \ref{skewsymmetricpermutation} gives a
permutation matrix
$P\in\operatorname{Per}(N_{\mathbb R},\mathbb R)$ such that
\[
\mathcal P_{N_{\mathbb R}}C'
=
P^tJ_{(N_{\mathbb R},N_{\mathbb R}-1)}P.
\]
Finally,
\[
S^{\boxstar}CS
=
\mathcal P_{N_{\mathbb R}}
P^tJ_{(N_{\mathbb R},N_{\mathbb R}-1)}P.
\]
\end{proof}

We would like to obtain similar results for any rank, but as mentioned in Remark \ref{whentheblocksdonthavethesamesize}, our procedure does not work for all possible matrices we would encounter. 
\section{The sets $\mathcal{T}_{\mathcal{J}_{N_{\mathbb{C}}}}$ and $\mathcal{ H}_{\mathcal{J}_{N_{\mathbb{C}}}}^{K}$: Complex eigenvalues}
The study of complex eigenvalues will follow the same order as the real eigenvalues case, most of the proofs are very similar, but there are important differences, so we treat it separately.   Consider a matrix in real Jordan normal form  $\mathcal{J}_{N_{\mathbb{C}}}$ with only complex eigenvalues $\alpha_j=a_j+ib_j$, $b_j>0$. We write it as 
\[
\mathcal{J}_{N_{\mathbb{C}}}=\mathcal{C}_{n_1}(a_1,b_1)\oplus \cdots \oplus \mathcal{C}_{n_p}(a_p,b_p). 
\]

\subsection{Some special matrices and their properties} \label{sectioncomplexdefinitionmatrices}
We say that a matrix $B_{ij} \in \mathrm{M}(2n_i\times 2n_j,\mathbb{R})$ is \textit{2-upper Toeplitz} if it is of the same form as in equation (\ref{formtoeplitz}) but each $b_{ij}^{(t)} \in \matrices{2}{2}$.
We say that a block matrix $B=[B_{ij}]^p_{i=1,j=1} \in \mathrm{M}(N_{\mathbb{C}}\times N_{\mathbb{C}},\mathbb{R})$ is \textit{2-N-upper Toeplitz} if each block is $2$-upper Toeplitz. 
In a similar way we say a matrix is \textit{2-lower Hankel} if it is of the same form as in equation (\ref{formhankel})  but each $b_{ij}^{(t)} \in \matrices{2}{2}$. A matrix $B=[B_{ij}]^p_{i=1,j=1}$ is \textit{2-N-lower Hankel} if each block is $2$-lower Hankel. Let us define the following matrices having the same block partition as the matrix $B$.
\begin{align}
    \mathcal{P}_{N_{\mathbb{C}},2}=\mathcal{P}_{n_1,2}\oplus \cdots \oplus \mathcal{P}_{n_p,2}& \;\;, \;\mathcal{P}_{n_i,2}:=\left( \begin{array}{cccc}
    &&&I_2\\
    &&I_2&\\
    &\iddots&&\\
    I_2&&&
    \end{array}
    \right)\in \mathrm{M}(2n_i\times 2n_i,\mathbb{R}), \\
    I_{N_{\mathbb{C}},2}^{\pm}=I_{n_1,2}^{\pm}\oplus \cdots \oplus I_{n_p,2}^{\pm}& \;\;,  \; I_{n_i,2}^{\pm}:=\left( \begin{array}{cccc}
    I_2&&&\\
    &-I_2&&\\
    &&I_2&\\
    &&&\ddots
    \end{array}
    \right)\in \mathrm{M}(2n_i\times 2n_i,\mathbb{R}).
\end{align}

 A matrix $B^{\prime}$ is \textit{2-$N_{\mathbb{C}}$-upper alternating Toeplitz (lower Hankel)} if $B^{\prime}=I_{N_{\mathbb{C}},2}^{\pm}B$ for some $2$-$N_{\mathbb{C}}$-upper Toeplitz (lower Hankel) matrix $B$.  Notice that if $B^{\prime}$ is $2$-$N_{\mathbb{C}}$-lower (alternating) Hankel then 
\[
B=\mathcal{P}_{N_{\mathbb{C}},2}B^{\prime}
\]
is $2$-$N_{\mathbb{C}}$-upper (alternating) Toeplitz and vice versa.
Similarly to the previous section we can define the matrix 
\[
\hat{B}_{ij}:=I_{n_i,2}^{\pm}B_{ij}I_{n_j,2}^{\pm}.
\] 
This satisfies 
\[
B_{ij}I_{n_j,2}^{\pm}=I_{n_i,2}^{\pm}\hat{B}_{ij}.
\]
And (for $B_{ij}$ square) if $p(B_{ij})$ is a polynomial in $B_{ij}$ and $p(\hat{B}_{ij})$ is the same polynomial in $\hat{B}_{ij}$ then 
\begin{equation}\label{commutationwithhatcomplexcase}
p(\hat{B}_{ij})I_{n_j,2}^{\pm}=I_{n_i,2}^{\pm}p(B_{ij}).
\end{equation}
We define the \textit{2-N-block star} of a block matrix.  

\[
B^{\boxstartwo}:=\mathcal{P}_{N_{\mathbb{C}},2}B^{T}\mathcal{P}_{N_{\mathbb{C}},2}.
\]
For a single block  this means
\[
B_{ij}^{\boxstartwo}=\mathcal{P}_{n_j,2}B_{ij}^t\mathcal{P}_{n_i,2}.
\]
We will say that a matrix $B$ is \textit{2-N-skew-symmetric} if 
\[
B^{\boxstartwo}=-B.
\]
That is 
\begin{equation*}
B_{ji}=-\mathcal{P}_{n_j,2}B_{ij}^t\mathcal{P}_{n_i,2}.
\end{equation*}

We list some extra properties such a matrix $B$ could have. These conditions will appear in the Lemmas of next section. 
\begin{align}
    B_{ij}=0,&\;\; \text{whenever} \;\; (a_i,b_i)\neq (a_j,b_j). \label{conditioneigenvaluescomplextoeplitz} \\ 
    B_{ij}=0,&\;\,\text{whenever} \,(-a_i,b_i)\neq(a_j,b_j) . \label{conditioneigenvaluescomplexhankel}  \\
    b_{ij}^{(t)} J_2-J_2b_{ij}^{(t)}=0,& \;\; \text{for all} \; t. \label{8} 
    \end{align}

A block $b_{ij}^{(t)} \in \matrices{2}{2}$ satisfying Equation $(\ref{8})$  has the following shape 
\[
b_{ij}^{(t)}=\begin{pmatrix}
    a &b \\
    -b & a
\end{pmatrix},
\]
for some $a,b \in \mathbb{R}$. Recall such a block has eigenvalues $a\pm bi$. Also we have the trivial observation that such a matrix if it is not zero has full rank 2.

If $B$ is a skew-symmetric (in the usual sense) $2$-$N$-lower alternating Hankel matrix, satisfying property $(\ref{8})$, then the matrix $\mathcal{P}_{N_{\mathbb{C}},2}B$ is a $2$-$N$-skew-symmetric $2$-N-upper alternating Toeplitz matrix satisfying $(\ref{8})$. 
Let $B$ be a $2$-$N$-skew-symmetric $2$-N-upper alternating Toeplitz matrix satisfying $(\ref{8})$ . Equation $\eqref{conmmuteswithaltidentity}$ and Remark \ref{diagonalsarezero} are not satisfied exactly in the same way for blocks $B_{ii}$, but we have some similar properties. We don't have zero diagonals as in the real case instead we have that all its elements are of the following shape 
\begin{equation}\label{diagonalelemtsincomplexcase}   
b_{ii}^{(k)}=\begin{cases}
    \begin{pmatrix}
        a & 0  \\
        0&a
    \end{pmatrix} & \text{belongs to a diagonal with an even number of blocks}, \\
    \begin{pmatrix}
        0 & a  \\
        -a & 0 
    \end{pmatrix} & \text{belongs to a diagonal with an odd number of blocks}.
\end{cases}
\end{equation}
For some $a\in \mathbb{R}$. Define the following matrix 
\begin{equation}\label{thematrixthatcommutescomplexcase}
W_{n_i}^{\pm}=\begin{cases} 
I_{n_i,2}^{\pm} & \text{if} \; n_i \; \text{is even}, \\
I_{n_i,2}^{\pm}S_{2n_i}& \text{if} \; n_i \; \text{is odd}.
\end{cases}
\end{equation}
$S_{2n_i}$ is as in \eqref{9}. The matrix $C_{ii}:=(W_{n_i}^{\pm})^{-1}B_{ii}$ is $2$-N-upper Toeplitz matrix. We have the following commutation relation that is the equivalent of \eqref{conmmuteswithaltidentity}

\begin{equation}\label{conmmuteswith2altidentity}
C_{ii}^{\boxstartwo}W_{n_i}^{\pm}=W_{n_i}^{\pm}C_{ii}
\end{equation}
\begin{proof}
 First $B_{ii}^{\boxstartwo}=-B_{ii}$ because $B$ is $2$-$N$-skew-symmetric. Also we can easily show that for both of parities, $(W_{n_i}^{\pm})^{\boxstartwo}=-W_{n_i}^{\pm}$. Now we can calculate

   \begin{equation}
                  C_{ii}^{\boxstartwo}W_{n_i}^{\pm}=((W_{n_i}^{\pm})^{-1}B_{ii})^{\boxstartwo}W_{n_i}^{\pm}=B_{ii}^{\boxstartwo}((W_{n_i}^{\pm})^{-1})^{{\boxstartwo}}W_{n_i}^{\pm}=B_{ii}=W_{n_i}^{\pm}C_{ii}.
          \end{equation}
\end{proof}

Notice that the matrices we just described behave in practice in a very similar way to the real case. Therefore, by doing the appropriate modifications many of the proofs of previous section translate to the complex case. For example by substituting in those proofs $\altidentity{i}$ for $W_{n_i}^{\pm}$, etc.

\subsection{The solutions of equations (\ref{permutationequation}) and (\ref{transposepermutationequation})}
 As in the real case we first consider the case of only one Jordan block. 

\begin{lemma}\label{onecomplexjordanblockcommute}
Let $\mathcal{C}_r(a,b) \in \mathrm{M}(2r\times 2r,\mathbb{R})$, $\mathcal{C}_s(a,b) \in \mathrm{M}(2s\times 2s,\mathbb{R})$ and $X=(x_{ij}) \in \mathrm{M}(2r\times 2s,\mathbb{R})$, with each $x_{ij}\in \matrices{2}{2}$. Then 
\begin{equation}\label{1}
X\mathcal{C}_s(a,b)-\mathcal{C}_r(a,b)X=0
\end{equation}
if and only if $X$ is 2-upper Toeplitz matrix and satisfies condition (\ref{8}).
\end{lemma}
\begin{proof}
    First we write
    \begin{align*}
        \mathcal{C}_s(a,b)=aI_{2s}+\mathcal{C}_s(0,0)+b\hat{J}_s, \\
        \mathcal{C}_r(a,b)=aI_{2r}+\mathcal{C}_r(0,0)+b\hat{J}_r,
    \end{align*}
    with
    \[
 \matrices{2i}{2i}\ni \hat{J}_i={J}_{2}\oplus \cdots \oplus {J}_{2} \;\;, \text{with} \; {J}_{2}=\left( \begin{array}{cc}
    &1\\
    -1&
    \end{array}
    \right).
\]
    So (\ref{1}) becomes
    \begin{equation}\label{2}
    X\mathcal{C}_s(0,0)-\mathcal{C}_r(0,0)X+bX\hat{J}_s-b\hat{J}_rX=0.
    \end{equation}
    Next we show that we can decouple the first and second half of this equation. We will show that equation (\ref{2}) implies
    \[
    bX\hat{J}_s-b\hat{J}_rX=0.
    \]
    First some algebra. 
    Multiply (\ref{2}) by $\hat{J}_r$ from the left and by $\hat{J}_s$ from the right, to obtain 
    \begin{align}\label{3}
        \hat{J}_r\left(X\mathcal{C}_s(0,0)-\mathcal{C}_r(0,0)X\right)\hat{J}_s-b\hat{J}_rX+bX\hat{J}_s=0.
    \end{align}
    
    Calculate (\ref{2})+(\ref{3}) 
 
   \begin{multline}\label{4}
         (X\mathcal{C}_s(0,0)-\mathcal{C}_r(0,0)X)+\hat{J}_r(X\mathcal{C}_s(0,0)-\mathcal{C}_r(0,0)X)\hat{J}_s\\+2bX\hat{J}_s-2b\hat{J}_rX=0.
            \end{multline}
  It is easy to calculate that
    \begin{equation}\label{5}
        \begin{aligned} 
    \left((X\mathcal{C}_s(0,0)-\mathcal{C}_r(0,0)X)+\hat{J}_r(X\mathcal{C}_s(0,0)-\mathcal{C}_r(0,0)X)\hat{J}_s\right)_{ij} \\=x_{ij-1}+J_2x_{ij-1}J_2-(x_{i+1j}+J_2x_{i+1j}J_2)
     \end{aligned} 
  \end{equation}
  We proceed by induction over the diagonals. Recall that the $k$-th diagonal of, for example, the matrix $X$,  is usually denoted by $\mathrm{diag}_k(X)$, defined as 
  \[  
    \mathrm{diag}_k(X):= \lbrace x_{ij} \mid   j-i=k \rbrace,    
    \;\; \; 1-r \leq k \leq s-1.        
      \]
  
   We begin with the $(1-r)$-th block diagonal, which consists only of the
block in position $(r,1)$. From equation (\ref{5}) we get that 
   \[
   \left((X\mathcal{C}_s(0,0)-\mathcal{C}_r(0,0)X)+\hat{J}_r(X\mathcal{C}_s(0,0)-\mathcal{C}_r(0,0)X)\hat{J}_s\right)_{r\:1}=0.
   \]
   So then from equation (\ref{4}) we get 
   \[
   (bX\hat{J}_s-b\hat{J}_rX)_{r\:1}=0.
   \]
   So it holds for the $(1-r)$-th block diagonal. Now suppose that our hypothesis is true for the $q$-th diagonal and that 
   \[(bX\hat{J}_s-b\hat{J}_rX)_{lm}
   \]
   belongs to the $(q+1)$-th diagonal. From the induction hypothesis about the $q$-th diagonal 
    \[
    x_{lm-1}J_2-J_2x_{lm-1}=0
     \Longrightarrow x_{lm-1}+J_2x_{lm-1}J_2=0,
    \]
    \[
    x_{l+1m}J_2-J_2x_{l+1m}=0
     \Longrightarrow x_{l+1m}+J_2x_{l+1m}J_2=0.
    \]
      Therefore from equation (\ref{4}) and (\ref{5}) we get
    \[
    (bX\hat{J}_s-b\hat{J}_rX)_{lm}=0
    \]
    This completes the induction, and shows that (\ref{2}) becomes two separate equations  \begin{equation}
        \begin{aligned}
          X\mathcal{C}_s(0,0)-\mathcal{C}_r(0,0)X=0   \\
        bX\hat{J}_s-b\hat{J}_rX=0
        \end{aligned}        
    \end{equation}
    The first equation requires $X$ to be 2-upper Toeplitz (compare with Lemma 4.4.11 in \cite{topicsmatrixanalysis}), and the second equation implies (\ref{8}).   
    
    Conversely, if $X$ is $2$-upper Toeplitz and satisfies \eqref{8}, then both of the separated equations hold, and therefore \eqref{1} holds.

\end{proof}
Now we have the equivalent of the complex case for Proposition \ref{conmmutationwithjordan}.
\begin{prop}\label{conmmuteswithjordancomplex}
A block matrix $B=[B_{ij}]^k_{i=1,j=1}$ commutes with $\mathcal{J}_{N_{\mathbb{C}}}$, i.e 
\begin{equation}\label{commutesjordancomplexequation}
  B\mathcal{J}_{N_{\mathbb{C}}}-\mathcal{J}_{N_{\mathbb{C}}}B=0,  
\end{equation}

if and only if $B$ is 2-$N$-upper Toeplitz, satisfies property (\ref{conditioneigenvaluescomplextoeplitz}) and each block satisfies property (\ref{8}). 
\end{prop}
\begin{proof}
Equation (\ref{commutesjordancomplexequation}) is equivalent to the set of equations
\[
B_{ij}\mathcal{C}_{n_j}(a_j,b_j)-\mathcal{C}_{n_i}(a_i,b_i)B_{ij}=0.
\]
Each of these equations has a nontrivial solution if and only if $(a_i,b_i)=(a_j,b_j)$. So we only consider equations of the form
\[
B_{ij}\mathcal{C}_{n_j}(a_i,b_i)-\mathcal{C}_{n_i}(a_i,b_i)B_{ij}=0.
\]
These are the equations considered in previous Lemma.  This completes the proof.
\end{proof}
\begin{lemma}\label{2upperhankel}
Let $\mathcal{C}_r(-a,b) \in \mathrm{M}(2r\times 2r,\mathbb{R})$, $\mathcal{C}_s(a,b) \in \mathrm{M}(2s\times 2s,\mathbb{R})$ and $X=(x_{ij}) \in \mathrm{M}(2r\times 2s,\mathbb{R})$. Then 
\begin{equation}\label{6}
X\mathcal{C}_s(a,b)+\mathcal{C}_r(-a,b)^tX=0
\end{equation}
if and only if $X$ is $2$-lower alternating Hankel matrix and satisfies property (\ref{8}).

\end{lemma}

\begin{proof}
First we write
    \begin{align*}
        \mathcal{C}_s(a,b)=aI_{2s}+\mathcal{C}_s(0,0)+b\hat{J}_s, \\
        \mathcal{C}_r(-a,b)^t=-aI_{2r}+\mathcal{C}_r(0,0)^t-b\hat{J}_r,
    \end{align*}
So (\ref{6})  becomes 
\[
X\mathcal{C}_s(0,0)+\mathcal{C}_r(0,0)^tX+bX\hat{J}_s-b\hat{J}_rX=0.
\]
This can be separated into the following two equations
\begin{equation}
        X\mathcal{C}_s(0,0)+\mathcal{C}_r(0,0)^tX=0         
    \end{equation}
    \begin{equation}
        X\hat{J}_s-\hat{J}_rX=0.
    \end{equation}
    The first one requires $X$ to be $2$-lower alternating Hankel (compare with  Lemma \ref{onejordanblocklemmatransposecommute} ), the second one  implies (\ref{8}).  The proof is similar to Lemma \ref{onecomplexjordanblockcommute} so we omit it. 
 \end{proof}
The next proposition is the equivalent of Proposition \ref{Propositionconditioncommuteswithtranspose}.
   \begin{prop}\label{transposeconmmuteswithjordancomplex}
A block matrix $B=[B_{ij}]^k_{i=1,j=1}$ satisfies 
\begin{equation} \label{commutesjordancomplextranspose}
  B\mathcal{J}_{N_{\mathbb{C}}}+\mathcal{J}_{N_{\mathbb{C}}}^tB=0  
\end{equation}

if and only if $B$ is 2-$N$-lower alternating Hankel, satisfies (\ref{conditioneigenvaluescomplexhankel})  and each block satisfies property (\ref{8}). 

\end{prop}
   \begin{proof}
Equation (\ref{commutesjordancomplextranspose}) is equivalent to the set 
\[
B_{ij}\mathcal{C}_{n_j}(a_j,b_j)+\mathcal{C}_{n_i}(a_i,b_i)^tB_{ij}=0.
\]
Each of these equations has a non trivial solution if and only if $(-a_i,b_i)=(a_j,b_j)$. So we only consider equations of the form
\[
B_{ij}\mathcal{C}_{n_j}(a_i,b_i)+\mathcal{C}_{n_i}(-a_i,b_i)^tB_{ij}=0.
\]
These are the equations considered in the previous lemma, so this completes the proof.
\end{proof} 

We can now rewrite the matrix sets in (\ref{matrixsets}) for the complex case
\begin{equation} \label{matrixsetscomplexrewrite}
    \begin{aligned}
  \mathcal{T}_{\mathcal{J}_{{N_{\mathbb{C}}}}}&:= \left\{ [A_{ij}] \in \mathrm{GL}(N_{\mathbb{C}},\mathbb{R})\;\middle|\; \begin{aligned}
 &2\text{-}N_{\mathbb{C}}\text{-upper Toeplitz} \\ 
&\text{Satisfies} \;(\ref{conditioneigenvaluescomplextoeplitz}) \; \text{and} \; (\ref{8}).
\end{aligned}
               \right\}  \\
               \mathcal{H}_{\mathcal{J}_{N_{\mathbb{C}}}}^{K}&:=\left\{ [B_{ij}] \in \mathrm{Skew}(N_{\mathbb{C}}\times N_{\mathbb{C}},\mathbb{R}) \;\middle|\; \begin{aligned}
 &2\text{-}N_{\mathbb{C}}\text{-lower alternating Hankel}  \\ 
&\text{Satisfies} \;(\ref{conditioneigenvaluescomplexhankel}) \; \text{and} \; (\ref{8}) \\
&Rank=K.
\end{aligned}
 \right\}\
    \end{aligned}
\end{equation}

In the next section, as in the real case, we answer the question: for a given $\mathcal{J}_{N_{\mathbb{C}}}$, for which $K$ is $\mathcal{H}_{\mathcal{J}_{N_{\mathbb{C}}}}^K\neq \emptyset$.  Then we examine the shape of the quotient space $(\ref{quotient space})$.  In fact we will start with the latter, as it will let us answer the former more easily.

\subsection{The quotient space}

First as in the real case, we show that we can study an equivalent problem. 
\begin{lemma} \label{equivalentmatrixproblemcomplex}
Define the sets 
\begin{equation} \label{matrixsetscomplexrewritewithstar}
    \begin{aligned}
  \dot{\mathcal{T}}_{\mathcal{J}_{N_{\mathbb{C}}}}&:= \left\{ [A_{ij}] \in \mathrm{GL}(N_{\mathbb{C}},\mathbb{R})\;\middle|\; \begin{aligned}
 &2\text{-}N_{\mathbb{C}}\text{-upper Toeplitz} \\ 
&\text{Satisfies}  \;(\ref{conditioneigenvaluescomplextoeplitz}) \; \text{and} \; (\ref{8}).
\end{aligned}
               \right\}  \\
               \dot{\mathcal{H}}_{\mathcal{J}_{N_{\mathbb{C}}}}^{K}&:=\left\{ [B_{ij}] \in 2\text{-}N_{\mathbb{C}}\text{-Skew}(N_{\mathbb{C}}\times N_{\mathbb{C}},\mathbb{R}) \;\middle|\; \begin{aligned}
 &2\text{-}N_{\mathbb{C}}\text{-upper alternating Toeplitz}  \\ 
&\text{Satisfies}  \;(\ref{conditioneigenvaluescomplexhankel}) \; \text{and} \; (\ref{8}). \\
&Rank=K
\end{aligned}
 \right\}\
    \end{aligned}
\end{equation}
Then, there exists a bijection 

\begin{equation}\label{bijectionquotientspacecomplex}     
 \mathcal{H}_{\mathcal{J}_{N_{\mathbb{C}}}}^{K}/_{\mathrm{cong}}\mathcal{T}_{\mathcal{J}_{N_{\mathbb{C}}}} \xrightarrow{bij}  \dot{\mathcal{H}}_{\mathcal{J}_{N_{\mathbb{C}}}}^{K}/_{\text{2-N-blockstar cong }} \dot{\mathcal{T}}_{\mathcal{J}_{N_{\mathbb{C}}}}.
\end{equation} 
\end{lemma}
    \begin{proof}
   Let $B,C \in \mathcal{H}_{\mathcal{J}_{N_{\mathbb{C}}}}^K $. Recall from Section (\ref{sectioncomplexdefinitionmatrices}) that $\mathcal{P}_{N_{\mathbb{C}},2}B$ and $\mathcal{P}_{N_{\mathbb{C}},2}C$ are $2$-$N_{\mathbb C}$-skew-symmetric $2$-$N_{\mathbb C}$-upper Toeplitz. Then we just have to notice that 
        \begin{align*}
             &A^tCA=B \\
             \Leftrightarrow  &\mathcal{P}_{N_{\mathbb{C}},2}(\mathcal{P}_{N_{\mathbb{C}},2}A^t\mathcal{P}_{N_{\mathbb{C}},2})\mathcal{P}_{N_{\mathbb{C}},2}CA=B \\
              \Leftrightarrow  &A^{\boxstartwo}(\mathcal{P}_{N_{\mathbb{C}},2}C)A=\mathcal{P}_{N_{\mathbb{C}},2}B.
        \end{align*}
        
           \end{proof}

We concentrate then in the latter quotient space.  Before describing such quotient space we first introduce some new matrices and notation. 
\subsubsection{Some more matrices and notation.} 
This section is the analogue of Section \ref{Some more matrices and notation.} but for the complex eigenvalues case. We have a similar notation as the one introduced there 
\begin{align*}
    n_{l,m}:=\mathrm{min}(n_l,n_m), \\
    s_{(l,m;k)}:=n_{l,m}-k+1.
\end{align*}
We define the equivalent matrices in the obvious ways. For a $2$-upper Toeplitz block $B_{lm}$, we denote by $B_{lm}^{\langle u\rangle} \in \matrices{2s_{(l,m;u)}}{2s_{(l,m;u)}}$ the submatrix formed by only considering the diagonals with elements $b_{lm}^{(t)}$, with $t\geq u$ that is

\[
B_{lm}^{\langle u\rangle}:=\left( \begin{array}{cccc} 
b_{lm}^{(u)}&b_{lm}^{(u+1)}&\cdots& b_{lm}^{(n_{l,m})} \\ 
&b_{lm}^{(u)}&\ddots&\vdots \\
&&\ddots&b_{lm}^{(u+1)} \\
&&& b_{lm}^{(u)}\\
\end{array} \right)
\]
If $b_{lm}^{(u)}\neq 0$, then $B_{lm}^{\langle u\rangle}$ is a $2$-upper Toeplitz matrix with nonzero block diagonal. Then again the matrix formed by erasing the submatrix $B_{lm}^{\langle u\rangle}$ from the original matrix $B_{lm}$, we will denote this matrix by $B_{lm}^{<u}$. This matrix is just as described in Equation \eqref{matrixwithoutcorner}, in the obvious way.

Now we define the equivalent \say{extension} operation.  Suppose we have a $2$-upper Toeplitz matrix $A\in \matrices{2k}{2k}$, we use the notation as in (\ref{formtoeplitz}) to denote the nonzero diagonals of the matrix by $a^{(i)}$. For any $n,m$, such that $k\leq\min(n,m)$ we can \say{extend} the matrix $A$ to another $2$-upper Toeplitz matrix  in $\matrices{2n}{2m}$ by a map (we use the same notation as in the real case, as no confusion will arise)

\[
\mathrm{ToeplitzExt}_{n,m}:\matrices{2k}{2k}\rightarrow \matrices{2n}{2m}
\]
To define this map, define $B:=\mathrm{ToeplitzExt}_{n,m}(A)$ and again use the notation $b^{(i)}$ for the diagonals, then we can define the map by  
\begin{equation}
   b^{(i)}:=\begin{cases}
       a^{(i)} & 1\leq i\leq k \\
       0&  \text{other cases}.
   \end{cases}
\end{equation}

Of course again such a matrix has the important property described in Equation $(\ref{blockappearstwice})$.  We define a couple of more matrices, in which we embed some of the previous matrices in the upper corner
\[
I_{lm,2}^{[k]\pm}:=
\left(  \begin{array}{c|c} 
0&I_{s_{(l,m;k)},2}^{\pm}\\ \hline
0&0
\end{array} \right) \in \matrices{2n_l}{2n_m}.
\]
\[
 W_{ll}^{[k]\pm}:=\left(  \begin{array}{c|c} 
0&W_{s_{(l,l;k)}}^{\pm}\\ \hline
0&0
\end{array} \right) \in \matrices{2n_l}{2n_l}.
 \]
$W$ is as in \eqref{thematrixthatcommutescomplexcase}.
For convenience, we extend the notation by setting
\begin{align*}
    I_{lm,2}^{[k]\pm}&:=0
\qquad\text{whenever}\qquad
k>\min(n_l,n_m), \\
W_{ll}^{[k]\pm}&:=0
\qquad\text{whenever}\qquad
k>n_l.
\end{align*}

Finally notice the following 
 \begin{align}\label{conjugationofEcomlex}
(I_{lm,2}^{[k]\pm})^{\boxstartwo}=
 \begin{cases}
-I_{ml,2}^{[k]\pm} &\text{if} \; s_{(l,m;k)} \; \text{is even}, \\
I_{ml,2}^{[k]\pm} &\text{if} \; s_{(l,m;k)} \; \text{is odd}.
\end{cases}
\end{align}

\subsubsection{The set $\dot{\mathcal{H}}_{\mathcal{J}_{N_{\mathbb{C}}}}^{K}/_{\text{2-N-blockstar cong }} \dot{\mathcal{T}}_{\mathcal{J}_{N_{\mathbb{C}}}}$.}

The following Lemmas are the equivalent of the ones in section \ref{realquotientspacesection}. 

\begin{lemma}\label{step1complexcase}
     Let $B=[B_{ij}] \in \dot{\mathcal{H}}_{\mathcal{J}_{N_{\mathbb{C}}}}^{K}$.  Suppose that $B_{lm}$ is such that $b_{lm}^{(k)} \neq 0$ and $b_{lm}^{(q)}=0$ for all $q<k$.    
    Then there exists $S=[S_{ij}] \in \dot{\mathcal{T}}_{\mathcal{J}_{N_{\mathbb{C}}}}$  such that 

\[
S^{\boxstartwo}BS=\begin{cases}
\left( \begin{array}{ccccc} 
*&&&&*\\ 
&\ddots&&&\\
&&tW_{ll}^{[k]\pm}&& \\
&*&&\ddots &\\
&&& &*
\end{array} \right) \;\; \text{if} \; l=m, \\
\left(\begin{array}{ccccccc} 
*&&&&&&  \\ 
&\ddots&&&&*& \\
&&*&\cdots&tI_{lm,2}^{[k]\pm}&& \\
&&\vdots&*&\vdots&& \\
&&-t(I_{lm,2}^{[k]\pm})^{\boxstartwo}&\cdots&*&& \\
&*&&&&\ddots& \\
&&&&&&*
\end{array} \right) \;\; \text{if} \; l\neq m,
\end{cases} 
\]          
for some $t\in\{1,-1\}$. That is $S^{\boxstartwo}BS \in \dot{\mathcal{H}}_{\mathcal{J}_{N_{\mathbb{C}}}}^{K} $ is such that the $l,l$ component is \(tW_{ll}^{[k]\pm}\) if \(l=m\), and the
$l,m$ component is \(tI_{lm,2}^{[k]\pm}\) if \(l\neq m\).
\end{lemma}
\begin{proof}

The proof is very similar to Lemma \ref{step1}. Define
\[
B_{lm}^{\langle k \rangle,\mathrm{na}}:=\begin{cases}
t(W^{\pm}_{s_{(l,l;k)}})^{-1}B_{ll}^{\langle k\rangle} & l=m \\
tI_{s_{(l,m;k)},2}^{\pm}B_{lm}^{\langle k\rangle} & l\neq m.
\end{cases}
\]
where $t\in \lbrace 1,-1\rbrace$ is selected such that $B_{lm}^{\langle k \rangle,\mathrm{na}}$  has no negative real eigenvalues. $B_{lm}^{\langle k \rangle,\mathrm{na}}$ is $2$-upper Toeplitz. From Proposition \ref{existenceofroots}, we can find a polynomial $p(B_{lm}^{\langle k \rangle,\mathrm{na}})$ such that $p(B_{lm}^{\langle k \rangle,\mathrm{na}})^2=(B_{lm}^{\langle k \rangle,\mathrm{na}})^{-1}$.
In particular, we have 
\begin{align*}
    p(B_{lm}^{\langle k \rangle,\mathrm{na}})B_{lm}^{\langle k \rangle,\mathrm{na}}p(B_{lm}^{\langle k \rangle,\mathrm{na}})=I_{2s_{(l,m;k)}}
\end{align*}

\textbf{Case $l=m$}. 
First notice that 
\begin{align*}
    p(B_{ll}^{\langle k \rangle,\mathrm{na}})^{\boxstartwo}B_{ll}^{\langle k\rangle}p(B_{ll}^{\langle k \rangle,\mathrm{na}})&=tp(B_{ll}^{\langle k \rangle,\mathrm{na}})^{\boxstartwo}W^{\pm}_{s_{(l,l;k)}}B_{ll}^{\langle k \rangle,\mathrm{na}}p(B_{ll}^{\langle k \rangle,\mathrm{na}})\\
&=tW^{\pm}_{s_{(l,l;k)}}p(B_{ll}^{\langle k \rangle,\mathrm{na}})B_{ll}^{\langle k \rangle,\mathrm{na}}p(B_{ll}^{\langle k \rangle,\mathrm{na}}) \\
&=tW^{\pm}_{s_{(l,l;k)}}.
\end{align*}
Here we used the commutation in \eqref{conmmuteswith2altidentity}. 
Define the matrix 
\[
 \matrices{2n_l}{2n_l} \ni S_{ll}:=\mathrm{ToeplitzExt}_{n_l,n_l}\left(p(B_{ll}^{\langle k \rangle,\mathrm{na}})\right). 
 \] 
For such a matrix we have 
\begin{align*}
S_{ll}^{\boxstartwo}B_{ll}S_{ll}&=\left(\begin{array}{c|c} 
0&p(B_{ll}^{\langle k \rangle,\mathrm{na}})^{\boxstartwo}B_{ll}^{\langle k\rangle}p(B_{ll}^{\langle k \rangle,\mathrm{na}}) \\ \hline
0& 0
\end{array} \right)\\
&=t W_{ll}^{[k]\pm}.
\end{align*}

Finally, define the nonsingular diagonal matrix  $ S=[S_{ij}] \in \dot{\mathcal{T}}_{\mathcal{J}_{N_{\mathbb{C}}}}$ by 
\[
S_{ij} :=
\begin{cases}
I_{2n_i} & i = j,\ i \neq l,\\
\mathrm{ToeplitzExt}_{n_l,n_l}\left(p(B_{ll}^{\langle k \rangle,\mathrm{na}})\right) & i = j = l,\\
0 & i \neq j.
\end{cases}
\]
This matrix gives our desired result.

\textbf{Case $l\neq m$.} From (\ref{commutationwithhatcomplexcase}) we have 
\begin{align*}
    p(\hat{B}_{lm}^{\langle k \rangle,\mathrm{na}})B_{lm}^{\langle k \rangle}p(B_{lm}^{\langle k \rangle,\mathrm{na}})&=tp(\hat{B}_{lm}^{\langle k \rangle,\mathrm{na}})I^{\pm}_{s_{(l,m;k)},2}B_{lm}^{\langle k \rangle,\mathrm{na}}p(B_{lm}^{\langle k \rangle,\mathrm{na}}) \\
      &=tI^{\pm}_{s_{(l,m;k)},2}.
\end{align*}
Then \say{extend} these matrices

\begin{align*}
 \matrices{2n_m}{2n_m} \ni T_{mm}&:=\mathrm{ToeplitzExt}_{n_m,n_m} \left(p(B_{lm}^{\langle k \rangle,\mathrm{na}})\right), \\
 \matrices{2n_l}{2n_l} \ni T_{ll}&:=\mathrm{ToeplitzExt}_{n_l,n_l}\left(p(\hat{B}_{lm}^{\langle k \rangle,\mathrm{na}})\right).
\end{align*}
We have
\begin{align*}
T_{ll}B_{lm}T_{mm}&=\left(\begin{array}{c|c} 
0&p(\hat{B}_{lm}^{\langle k \rangle,\mathrm{na}})B_{lm}^{\langle k\rangle}p(B_{lm}^{\langle k \rangle,\mathrm{na}}) \\ \hline
0& 0
\end{array} \right)=tI_{lm,2}^{[k]\pm}.
\end{align*}
Finally, define the diagonal nonsingular matrix  $ S=[S_{ij}] \in \dot{\mathcal{T}}_{\mathcal{J}_{N_{\mathbb{C}}}}$ by 
\[
S_{ij} :=
\begin{cases}
I_{2n_i} & i = j,\ i \neq l, i \neq m,\\
(\mathrm{ToeplitzExt}_{n_l,n_l}\left(p(\hat{B}_{lm}^{\langle k \rangle,\mathrm{na}})\right))^{\boxstartwo} & i = j = l,\\
\mathrm{ToeplitzExt}_{n_m,n_m} \left(p(B_{lm}^{\langle k \rangle,\mathrm{na}})\right) & i = j = m,\\
0 & i \neq j.
\end{cases}
\]
This matrix gives our desired result. 
\end{proof}  
Next the analogue of Lemma (\ref{firstlemma}).
\begin{lemma}\label{firstlemmacomplex}   Let $B=[B_{ij}]\in \dot{\mathcal{H}}_{\mathcal{J}_{N_{\mathbb{C}}}}^{K}$.  Suppose $B_{ll}$ is such that $b_{ll}^{(k)} \neq 0$ and $b_{ll}^{(q)}=0$ for all $q<k$. Also suppose that for all other $i\neq l$ with $n_i=n_l$, we have $b_{il}^{(q)}=b_{li}^{(q)}=0$ for all $q< k$.  Then, there exists  $S \in \dot{\mathcal{T}}_{\mathcal{J}_{N_{\mathbb{C}}}}$ and $t\in \lbrace 1,-1\rbrace $ such that

  \[
D=[D_{ij}]:=S^{\boxstartwo}BS
\]        
satisfies 
\begin{enumerate}
    \item $D_{ll}=tW_{ll}^{[k]\pm},$
    \item $D_{il}=D_{li}=0$ for all $i\neq l$ such that $n_i=n_l$,
\end{enumerate}
 \end{lemma}
 \begin{proof} The proof is similar to that of Lemma \ref{firstlemma}, so we skip some details. First, from Lemma \ref{step1complexcase} we can without loss of generality suppose that $B_{ll}=tW_{ll}^{[k]\pm}$ for some $t\in \lbrace 1,-1\rbrace $. For $B_{lj}$, $j\neq l$,  
\begin{equation}\label{matrixtoeliminatecomplex}
B_{lj}^{\langle k \rangle,\mathrm{na}}:=t(W_{s_{(l,j;k)}}^{\pm})^{-1}B_{lj}^{\langle k\rangle}.
\end{equation}
Then define the matrix 
    
\begin{equation}\label{completethematrixcomplex}
 \matrices{2n_l}{2n_j} \ni S_{lj} := \mathrm{ToeplitzExt}_{n_l,n_j}(B_{lj}^{\langle k \rangle,\mathrm{na}}).
\end{equation} 
We have the following
\[ (tW_{ll}^{[k]\pm})(-S_{lj}) = -\left(B_{lj}-B_{lj}^{<k}\right). \]

Finally, define $ S=[S_{ij}] \in \dot{\mathcal{T}}_{\mathcal{J}_{N_{\mathbb{C}}}}$ as 

\[
S_{ij} :=
\begin{cases}
I_{2n_i} & i = j,\\
-\mathrm{ToeplitzExt}_{n_l,n_j}(B_{lj}^{\langle k \rangle,\mathrm{na}}) & \text{all} \;\; i=l, \; j\neq l ,\\
0 & \text{other cases}.
\end{cases}
\]
Then consider the product 
\[
S^{\boxstartwo}BS,
\]
which gives our desired answer. Notice that there is no problem defining the previous matrix $S$, because if $B_{ll}\neq 0$,  condition (\ref{conditioneigenvaluescomplexhankel}) implies $\lambda_l=ib$ for some $b\in \mathbb{R}$. Then for all $m\neq l$, if $B_{lm} \neq 0$ we have also that $\lambda_m =ib$. So there is no problem with having $S_{lm}\neq0$, as it satisfies condition (\ref{conditioneigenvaluescomplextoeplitz}). 
 \end{proof}

 \begin{coro}
    \label{firstcorollarycomplexcase}  Let $B=[B_{ij}] \in \dot{\mathcal{H}}_{\mathcal{J}_{N_{\mathbb{C}}}}^{K}$. Suppose that $B_{ll}$, is such that $b_{ll} \neq 0$.   Then there exists a nonsingular $2$-$N$-upper Toeplitz matrix  $S\in \dot{\mathcal{T}}_{\mathcal{J}_{N_{\mathbb{C}}}}$ such that

  \[
S^{\boxstartwo}BS=
\left( \begin{array}{ccccccc} 
&&&0&&&\\ 
&*&&\vdots&&*&\\ 
&&&0&&&\\
0&\cdots&0&tW_{n_l}^{\pm}&0&\cdots& 0 \\
&&&0& &&\\
&*&&\vdots& &*&\\
&&&0& &&

\end{array} \right).
\]          
For some $t\in \{1,-1\}$. 
 \end{coro}
 \begin{proof}
         This follows easily from Lemma \ref{firstlemmacomplex}.    
 \end{proof}
Next the equivalent of Lemma \ref{secondlemma}.
 \begin{lemma}\label{secondlemmacomplex}
     Let $B=[B_{ij}]\in \dot{\mathcal{H}}_{\mathcal{J}_{N_{\mathbb{C}}}}^{K}$. Suppose that $B_{lm}$, with $l\neq m$ and $n_l=n_m$, is such that $b_{lm}^{(k)} \neq 0$ and $b_{lm}^{(q)}=0$ for all $q<k$.  Also suppose that for all other $z$ such that $n_z=n_l$, but $z\neq l$ and $z\neq m$ we have $b_{zl}^{(q)}=b_{lz}^{(q)}=b_{zm}^{(q)}=b_{mz}^{(q)}=0$ for all $q<k $ . Finally suppose that  $b_{ll}^{(q)}=b_{mm}^{(q)}=0$ for $q\leq k$.
   Then, there exists  $S \in \dot{\mathcal{T}}_{\mathcal{J}_{N_{\mathbb{C}}}}$ and $t\in \lbrace 1,-1\rbrace $ such that

  \[
D=[D_{ij}]:=S^{\boxstartwo}BS
\]         
satisfies
\begin{enumerate}
    \item $D_{lm}=tI_{lm,2}^{[k]\pm}$,
    \item $D_{ll}=D_{mm}=0$, 
    \item $D_{il}=D_{li}=D_{im}=D_{mi}=0$ for all $i\neq l,m$ such that $n_i=n_l$,
    \item $d_{il}^{(u)}=d_{li}^{(u)}=d_{im}^{(u)}=d_{mi}^{(u)}=0$ for all   $u\geq k$ and for all $i\neq l,m$ with $n_i\neq n_l$ ( That is for such blocks the matrix has the same shape as described in $(\ref{matrixwithoutcorner})$).
\end{enumerate}

 \end{lemma}
 \begin{proof}
The proof is similar to Lemma \ref{secondlemma}, so we omit some details.  From Lemma \ref{step1complexcase} we can without loss of generality suppose that $B_{lm}=t I_{lm,2}^{[k]\pm}$, $t\in \lbrace 1,-1\rbrace $. The proof of this case has two parts. First, we eliminate the elements $B_{ll}$ and $B_{mm}$. Define the matrices
    
\begin{align*}
   B_{mm}^{\langle k \rangle,\mathrm{na}}&:=tI^{\pm}_{s_{(m,m;k)},2}B_{mm}^{\langle k \rangle}, \\
  B_{ll}^{\langle k \rangle,\mathrm{na}}&:=tI^{\pm}_{s_{(l,l;k)},2}B_{ll}^{\langle k \rangle}.
\end{align*}
  Then define the corresponding matrices 
  \begin{align*}
    \matrices{2n_l}{2n_m} \ni S_{lm}&:=\mathrm{ToeplitzExt}_{n_l,n_m}(B_{mm}^{\langle k \rangle,\mathrm{na}}) \\
    \matrices{2n_m}{2n_l} \ni S_{ml}&:=\mathrm{ToeplitzExt}_{n_m,n_l}(B_{ll}^{\langle k \rangle,\mathrm{na}})
\end{align*}

Suppose $B_{ll}\neq 0$, from condition (\ref{conditioneigenvaluescomplexhankel}) we have that ( recall the notation for eigenvalues used in such condition) $ a_l=0$. Given that  $B_{lm}\neq 0 $ we must have $a_l=-a_m$, so we have also $a_m=0$. The same analysis for the case $B_{mm}\neq 0$. Therefore in any case we have no problem defining the following matrix $S=[S_{ij}]$ satisfying $(\ref{conditioneigenvaluescomplextoeplitz})$  with blocks given by 
\[
S_{ij} :=
\begin{cases}
I_{2n_i} & i = j,\\
-\frac{1}{2}\mathrm{ToeplitzExt}_{n_l,n_m}(B_{mm}^{\langle k \rangle,\mathrm{na}}) & i =l, j=m,\\
-\frac{1}{2}\mathrm{ToeplitzExt}_{n_m,n_l}(B_{ll}^{\langle k \rangle,\mathrm{na}}) & i = m, j= l,\\
0 & \text{others} .
\end{cases}
\]

We consider  the product
\begin{align}
S^{\boxstartwo}BS=:D=[D_{ij}].
\end{align}
We analyze the blocks $D_{ll},D_{mm},D_{lm}, D_{ml}$ and the same analysis as in Lemma $\ref{secondlemma}$, modified in the obvious ways give us that if we repeat this procedure several times we can suppose without loss of generality that the matrix 
$B$ is of the shape
 
\[
B=\left(\begin{array}{ccccccc} 
&&&&&&  \\ 
&*&&&&*& \\
&&0&\cdots&tI_{lm,2}^{[k]\pm}&& \\
&&\vdots&*&\vdots&& \\
&&-t(I_{lm,2}^{[k]\pm})^{\boxstartwo}&\cdots&0&& \\
&*&&&&*& \\
&&&&&&
\end{array} \right).
\]

 Now we eliminate the rest of elements in the $l$-th  and $m$-th rows and columns. Define, for $q\neq l,m$, the matrices
 
\begin{align*}
    B_{mq}^{\langle k \rangle,\mathrm{na}}&:=tI^{\pm}_{s_{(m,q;k)},2}B_{mq}^{\langle k \rangle},\\
    B_{lq}^{\langle k \rangle,\mathrm{na}} &:=tI^{\pm}_{s_{(l,q;k)},2}B_{lq}^{\langle k \rangle}.         
\end{align*}
 Then define the corresponding matrices 
  \begin{align*}
    \matrices{2n_l}{2n_q} \ni S_{lq}&:=\mathrm{ToeplitzExt}_{n_l,n_q}(B_{mq}^{\langle k \rangle,\mathrm{na}}) \\
    \matrices{2n_m}{2n_q} \ni S_{mq}&:=\mathrm{ToeplitzExt}_{n_m,n_q}(B_{lq}^{\langle k \rangle,\mathrm{na}})
\end{align*}

From our initial conditions $B_{lm}\neq 0$, this means from condition $(\ref{conditioneigenvaluescomplexhankel})$ (and with the notation used in such condition for the eigenvalues)  we have $(a_l,b_l)=(-a_m,b_m)$. If $B_{lq}\neq 0$, $q\neq m,l$, we have $(a_l,b_l)=(-a_q,b_q)$. Combining both of these equations we get $(a_m,b_m)=(a_q,b_q)$. Similarly, if $B_{mq}\neq0$, then $
(a_l,b_l)=(a_q,b_q)$. So we can define the following matrix $S=[S_{ij}]$ satisfying condition (\ref{conditioneigenvaluescomplextoeplitz}) by
 
\[
S_{ij} :=
\begin{cases}
I_{2n_i} & i = j,\\
-\mathrm{ToeplitzExt}_{n_l,n_j}(B_{mj}^{\langle k \rangle,\mathrm{na}}) & i =l, j\neq l,m,\\
-\mathrm{ToeplitzExt}_{n_m,n_j}(B_{lj}^{\langle k \rangle,\mathrm{na}}) & i = m, j\neq m,l ,\\
0 & \text{others} .
\end{cases}
\] 
Finally consider the product
\[
S^{\boxstartwo}BS,
\]
which gives our desired result. 
\end{proof}
\begin{coro}\label{secondlemmacomplexcasecorollary}
Let $B=[B_{ij}] \in \dot{\mathcal{H}}_{\mathcal{J}_{N_{\mathbb{C}}}}^{K}$. Suppose that for $B_{lm}$, $l\neq m$, $n_l=n_m$ , we have that $b_{lm} \neq 0$.  Also, suppose that $b_{ll}=0$ and $b_{mm}= 0$. Then there exists $S=[S_{ij}] \in \dot{\mathcal{T}}_{\mathcal{J}_{N_{\mathbb{C}}}}$ and $t\in\{1,-1\}$  such that 

  \[
S^{\boxstartwo}BS=
\left(\begin{array}{ccccccccccc}
&&&0&&&&0&&& \\
&*&&\vdots&&*&&\vdots&&*&  \\ 
&&&0&&&&0&&&  \\ 
0&\cdots&0&0&0&\cdots&0&tI_{n_l,2}^{\pm}&0&\cdots&0 \\
&&&0&&&&0&&& \\
&*&&\vdots&&*&&\vdots&&*& \\

&&&0&&&&0&&& \\
0&\cdots&0&-t(I_{n_l,2}^{\pm})^{\boxstartwo}&0&\cdots&0&0&0&\cdots&0 \\
&&&0&&&&0&&& \\
&*&&\vdots&&*&&\vdots&&*& \\
&&&0&&&&0&&&

\end{array} \right).
\]          
The remaining part of the matrix, represented by the $*$, is again $2$-$N_{\mathbb{C}}$-upper alternating Toeplitz and $2$-$N_{\mathbb{C}}$-skew-symmetric. 
 \end{coro}
 \begin{proof}
Follows easily from Lemma \ref{secondlemmacomplex}
\end{proof}
The previous Lemmas give an idea of the shape of the quotient space in the right hand side of (\ref{bijectionquotientspacecomplex}).

Next we answer the question : what Jordan normal forms $\mathcal{J}_{N_{\mathbb{C}}}$ admit solutions of a given rank $R$ for equation (\ref{transposepermutationequation})?   In light of the previous results it will be convenient to separate the Jordan normal form in two parts the blocks with zero real part eigenvalues $\mathcal J_{\mathbb C}^{0}$ and the blocks with nonzero real part $\mathcal J_{\mathbb C}^{\neq 0}$:

\begin{equation}\label{jordannormalforminspecialordercomplex}   
\mathcal J_{N_{\mathbb C}}
=
\mathcal J_{\mathbb C}^{0}
\oplus
\mathcal J_{\mathbb C}^{\neq 0}.
\end{equation}
Here the terms are defined as follows
\[
\mathcal J_{\mathbb C}^{0}
:=
\bigoplus_{\rho=1}^{q}\mathfrak J_{\rho}^{0}
\]
\[
\mathfrak J_{\rho}^{0}
:=
\bigoplus_{r=1}^{k_{\rho}^{0}}
\mathcal C_{n_{\rho,r}^{0}}(0,b_{\rho}),
\qquad
n_{\rho,1}^{0}
\geq
n_{\rho,2}^{0}
\geq
\cdots
\geq
n_{\rho,k_{\rho}^{0}}^{0}.
\]
\[
\mathcal J_{\mathbb C}^{\neq 0}
=
\bigoplus_{\mu=1}^{s}\mathfrak J_{\mu}^{\neq 0}
\]
\[
\mathfrak J_{\mu}^{\neq 0}
:=
\mathfrak J_{\mu}^{+}
\oplus
\mathfrak J_{\mu}^{-}.
\]
\[
\mathfrak J_{\mu}^{+}
:=
\bigoplus_{r=1}^{k_{\mu}^{+}}
\mathcal C_{n_{\mu,r}^{+}}(a_{\mu},b_{\mu}),
\qquad
n_{\mu,1}^{+}
\geq
n_{\mu,2}^{+}
\geq
\cdots
\geq
n_{\mu,k_{\mu}^{+}}^{+}
\]
\[
\mathfrak J_{\mu}^{-}
:=
\bigoplus_{r=1}^{k_{\mu}^{-}}
\mathcal C_{n_{\mu,r}^{-}}(-a_{\mu},b_{\mu}),
\qquad
n_{\mu,1}^{-}
\geq
n_{\mu,2}^{-}
\geq
\cdots
\geq
n_{\mu,k_{\mu}^{-}}^{-}.
\]
Then we define the following partial sizes 

\[
N_{\rho}^{0}
:=2
\sum_{r=1}^{k_{\rho}^{0}}n_{\rho,r}^{0}, \qquad    N_{\mu}^{+}
:=2
\sum_{r=1}^{k_{\mu}^{+}}n_{\mu,r}^{+},
\qquad
N_{\mu}^{-}
:=
2\sum_{r=1}^{k_{\mu}^{-}}n_{\mu,r}^{-}.
\]
Finally, the total sizes 
\[
N_{\mathbb C}^{0}
=
\sum_{\rho=1}^{q}N_{\rho}^{0},
\qquad
N_{\mathbb C}^{\neq 0}
=
\sum_{\mu=1}^{s}(N_{\mu}^{+}+N_{\mu}^{-}).
\]

 Of course we have $ N_{\mathbb C} = N_{\mathbb C}^{0} + N_{\mathbb C}^{\neq 0}$. For the eigenvalue $(a_{\mu},b_{\mu})$, $a_{\mu},b_{\mu}>0$, define the following vectors in  $\mathbb{R}^{\max(k_{\mu}^+,k_{\mu}^-)}$ as

 \begin{align}  
  \vec{L}_{(a_{\mu},b_{\mu})}
&=(n^+_{\mu,1},\ldots,n^+_{\mu,k_{\mu}^+},0,\ldots,0),\\
\vec{L}_{(-a_\mu,b_\mu)}
&=(n^-_{\mu,1},\ldots,n^-_{\mu,k_{\mu}^-},0,\ldots,0).
\end{align}
The $0$ are added as needed to complete the size of the vectors. As before $d_M$ denotes the Manhattan distance.

\begin{equation} \label{maximalrankRcomplexpartial}
      \mathcal{R}_{\mathfrak J_{\mu}^{\neq 0}}:=N_{\mu}^{+}+N_{\mu}^{-}- 2 d_M(\vec{L}_{(a_{\mu},b_{\mu})}, \vec{L}_{(-a_{\mu},b_{\mu})}).
    \end{equation}
Then we have the sum of these terms for the different eigenvalues and we define
 \begin{equation} \label{maximalrankRcomplex}
      \mathcal{R}_{\mathcal J_{\mathbb C}^{\neq 0}}:=N_{\mathbb C}^{\neq 0}- 2 \sum_{\mu=1}^{s} d_M(\vec{L}_{(a_{\mu},b_{\mu})}, \vec{L}_{(-a_{\mu},b_{\mu})}).
    \end{equation}

These two last equations are these analogues of \eqref{maximalranksrealcasepartial} and \eqref{maximalrankreal}.   Notice that there is no term associated with $\mathfrak J_{\rho}^{0}$, that is with the blocks with eigenvalues with real part zero. This is a consequence of the fact that the blocks associated with such eigenvalues can always have a nonzero diagonal (see \eqref{diagonalelemtsincomplexcase}) and this is not true for the real case.  Also notice that the term associated with eigenvalues having nonzero real part is
similar to that in the real case, but it contains a factor of $2$. This comes from the fact that we are considering $2$-upper Toeplitz matrices and not only upper Toeplitz matrices. These two differences will be discussed with more detail in the proofs of the following Lemma and proposition.  
Finally we define the following sets, analogous to the sets \eqref{possibleranksrealsets}, 
\begin{equation}\label{possiblerankscomplexsets}
    \begin{aligned}
    R_{\mathfrak J_{\rho}^{0}}&:=\left\{ r \mid r\in 2\mathbb{Z}, 0\leq r \leq N^0_{\rho}\right\}=\left\{ 0,2,\ldots, N^0_{\rho}\right\}, \\
    R_{\mathcal J_{\mathbb C}^{0}}&:=\left\{ r \mid r\in 2\mathbb{Z}, 0\leq r \leq N_{\mathbb C}^{0}\right\}=\left\{ 0,2,\ldots, N_{\mathbb C}^{0}\right\}, \\    
    R_{\mathfrak J_{\mu}^{\neq 0}}&:=\left\{ r \mid r\in 4\mathbb{Z}, 0\leq r \leq  \mathcal{R}_{\mathfrak J_{\mu}^{\neq 0}}\right\}=\left\{ 0,4,8,\ldots, \mathcal{R}_{\mathfrak J_{\mu}^{\neq 0}}\right\}, \\
    R_{\mathcal J_{\mathbb C}^{\neq 0}}&:=\left\{ r \mid r\in 4\mathbb{Z}, 0\leq r \leq  \mathcal{R}_{\mathcal J_{\mathbb C}^{\neq 0}}\right\}=\left\{ 0,4,8,\ldots, \mathcal{R}_{\mathcal J_{\mathbb C}^{\neq 0}}\right\}, \\
    R_{\mathcal J_{N_{\mathbb C}}}&:=R_{\mathcal J_{\mathbb C}^{\neq 0}}+R_{\mathcal J_{\mathbb C}^{0}}.
    \end{aligned} 
    \end{equation}

 We prove the equivalent of Lemma \ref{simpleformofsolutions}.  
\begin{lemma}\label{simpleformofsolutionscomplexcase}
    If $A=[A_{ij}] \in \dot{\mathcal{H}}_{\mathcal{J}_{N_{\mathbb{C}}}}^K $ , then there is also $B=[B_{ij}] \in \dot{\mathcal{H}}_{\mathcal{J}_{N_{\mathbb{C}}}}^K $ with the following properties
    \begin{enumerate}
        \item For each $i$ there exists at most one $j$ such that $B_{ij}\neq 0$, but for all other $l\neq j$, $B_{il}= 0$.
        \item For each $B_{ij}\neq 0$, $i\leq j$, we have 
        \[
    B_{ij}
    =
    \begin{cases}
    I_{ij,2}^{[k]\pm},&i\neq j,\\
    W_{ii}^{[k]\pm},&i=j.  
    \end{cases}, \quad B_{ji}=-\left(B_{ij}\right)^{\boxstartwo} .
    \]
        
       for some appropriate $k$.
    \end{enumerate}
    \end{lemma}
  
\begin{proof}
The proof is analogous to that of Lemma
\ref{simpleformofsolutions}, using Corollaries
\ref{firstcorollarycomplexcase} and
\ref{secondlemmacomplexcasecorollary}. The only differences are that the scalar rows and columns are
considered in pairs corresponding to the $2\times2$ block
structure, and that every nonzero $2\times2$ coefficient satisfying
\eqref{8} is nonsingular. Hence, zero rows and the corresponding
columns can be deleted and reinserted in pairs.

After applying the same reduction procedure as in the real case,
every nonzero block is of the form
$\varepsilon I_{ij,2}^{[k]\pm}$ or
$\varepsilon W_{ii}^{[k]\pm}$, where
$\varepsilon\in\{-1,1\}$. Since only the existence of a matrix of the same rank is required, we may remove these signs, making the corresponding replacement in the opposite block when $i<j$. The resulting matrix has the required form and the same rank as the original one.
\end{proof}

Next we have the equivalent of Lemma \ref{ranksforzeroeigenvaluesreal}. In the complex case we study a block $\mathfrak J_{\rho}^{0}$. 
\begin{lemma} \label{ranksforzeroeigenvaluescomplex}
We have that 
\[
\dot{\mathcal{H}}_{\mathfrak J_{\rho}^{0}}^{R}\neq \emptyset
\]
if and only if $R \in  R_{\mathfrak J_{\rho}^{0}} $. 
\end{lemma}
\begin{proof}
First we show that
     \[
     \dot{\mathcal{H}}_{\mathfrak J_{\rho}^{0}}^{N^0_{\rho}}\neq \emptyset.
     \]
    For simplicity, throughout this proof we write
\[
n_i:=n_{\rho,i}^{0},
\qquad
1\leq i\leq k_\rho^{0}.
\]
   The proof is similar to that of Lemma \ref{ranksforzeroeigenvaluesreal}. 
    We will construct a block matrix $B=[B_{ij}]_{1\leq i,j\leq k_{\rho}^{0}} $  
   in this set. We use the notation introduced at the beginning of this section . We will construct it with the properties described in Lemma $\ref{simpleformofsolutionscomplexcase}$, so that we only define one nonzero block in each row and such nonzero blocks will only have one nonzero diagonal.  For all $i$ just define $B_{ii}=W_{n_i}^{\pm}$ (see \eqref{thematrixthatcommutescomplexcase}). The resulting matrix is block diagonal and all blocks are nonsingular, so it has rank  $N^0_{\rho}$. Also is easy to see that indeed $B \in \dot{\mathcal{H}}_{\mathfrak J_{\rho}^{0}}^{N^0_{\rho}}$.
    
 Next we construct matrices with all the smaller ranks in  $R_{\mathfrak J_{\rho}^{0}}$. Again, as in  the proof of Lemma  \ref{ranksforzeroeigenvaluesreal} the idea is to \say{slide} up a nonzero block of a matrix of a given rank to obtain a new one with smaller rank. Let $B$ be the diagonal matrix constructed before, for some $i$ make the following substitution.

\[
W_{n_i}^{\pm}=W_{n_i}^{[1]\pm} \to W_{n_i}^{[2]\pm}
\]
Call the resulting matrix $C^{\prime}$.  We can check easily that $ C^{\prime} \in \dot{\mathcal{H}}_{\mathfrak J_{\rho}^{0}}^{N^0_{\rho}-2} $. We can repeat the process in the obvious way to obtain all the even ranks in $R_{\mathfrak J_{\rho}^{0}}$  by doing the substitutions:
\[
W_{n_i}^{\pm}=W_{n_i}^{[k]\pm} \to W_{n_i}^{[k+1]\pm}.
\]
The odd ranks are immediately forbidden from the skew-symmetry of the matrix. As the maximal possible rank stated in the statement of this Lemma is the full rank of the matrix of such size, there is nothing else to prove. (contrary to the real case where we needed to prove that certain bigger ranks were not possible).
 \end{proof}
 Next we have the equivalent of Lemma \ref{ranksfornonzeroeigenvaluesreal}.
\begin{lemma}\label{ranksfornonzeroeigenvaluescomplex}
    We have that 

    \[
   \dot{\mathcal{H}}_{{\mathfrak J_{\mu}^{\neq 0}}}^{R}  \neq \emptyset
    \]
    if and only if $R \in R_{\mathfrak J_{\mu}^{\neq 0}}$.
\end{lemma}
\begin{proof}
    The proof is practically the same as Lemma \ref{ranksfornonzeroeigenvaluesreal} with some small modifications. For simplicity, throughout this proof we enumerate the Jordan blocks in
$\mathfrak J_\mu^{\neq0}$ consecutively and write
\[
n_i
:=
\begin{cases}
n_{\mu,i}^{+},
&1\leq i\leq k_\mu^{+},\\[1mm]
n_{\mu,i-k_\mu^{+}}^{-},
&k_\mu^{+}<i\leq k_\mu^{+}+k_\mu^{-}.
\end{cases}
\]
    First we show that
     \[
\dot{\mathcal H}_{\mathfrak J_{\mu}^{\neq 0}}^{
\mathcal R_{\mathfrak J_{\mu}^{\neq 0}}
}
\neq\emptyset.
\]
    We will construct a block matrix $A=[A_{ij}]_{1\leq i,j\leq k_\mu^{+}+k_\mu^{-}} $ in this set.

    We suppose it has the shape of Lemma $\ref{simpleformofsolutionscomplexcase}$, so that we only define one nonzero block in each row and such nonzero blocks will only have one nonzero $2 \times 2$ block diagonal.  We \say{pair} blocks in order, that is for each $i$ such that  $1\leq i \leq \min(k_{\mu}^+,k_{\mu}^-) $ , make $A_{i\,k_\mu^++i}
:=
I_{i\,k_\mu^++i,2}^{[1]\pm}$ (and define the corresponding $A_{k_\mu^++i\,\,i}$ correspondingly). Notice each of such pairings introduces $2\abs{n_i-n_{k_{\mu}^+ +i}}$  zero rows and corresponding columns in the matrix $A$. All blocks not paired in the previous way make them zero. After summing all the ranks of these block matrices and the zero blocks, the resulting rank of the matrix $A$ is precisely $\mathcal R_{\mathfrak J_{\mu}^{\neq 0}}$. 
    
 Next we construct matrices with all the rest of smaller ranks in  $ R_{\mathfrak J_{\mu}^{\neq 0}}$. Again, as in the previous lemma, we \say{slide} up matrix of a given rank to obtain a new one with smaller rank. That is for a nonzero block we make the substitution. 
 \[
 I_{ij,2}^{[k]\pm} \to I_{ij,2}^{[k+1]\pm}.
 \]
As many times as needed. Notice that each substitution introduces 2 zero rows and 2 zero columns  and therefore  the rank is reduced by $4$ each time,  obtaining all the ranks in $ R_{\mathfrak J_{\mu}^{\neq 0}}$. Now if $B$ is any matrix in $\dot{\mathcal H}_{\mathfrak J_{\mu}^{\neq 0}}^{R}$ for some $R$, we can suppose it is given in the shape stated in $\ref{simpleformofsolutionscomplexcase}$. Notice that the rank of such a matrix is the sum of matrices of the shape 

\[
\begin{pmatrix}
0 & I_{ij,2}^{[k]\pm}\\[1mm]
-\left(I_{ij,2}^{[k]\pm}\right)^{\boxstartwo} & 0
\end{pmatrix}.
\]
Therefore the rank of such a matrix is multiple of $4$. 

At last we need to show that 
\[
     \dot{\mathcal H}_{\mathfrak J_{\mu}^{\neq 0}}^{R}= \emptyset,
     \]
     for all $R>\mathcal R_{\mathfrak J_{\mu}^{\neq 0}}$. The proof is practically the same as in Lemma $\ref{ranksfornonzeroeigenvaluesreal}$ so we omit it. We just need to remember that in the complex case the elements of the matrices considered are $2\times2$ matrices. 
\end{proof}
Next the main result of this section: the equivalent of Proposition \ref{sufficientneccesaryconditionsexistence}.
  
\begin{prop}\label{sufficientneccesaryconditionsexistencecomplexcase}
    We have that 
    \[
   \dot{\mathcal{H}}_{\mathcal J_{N_{\mathbb C}}}^{R}  \neq \emptyset
    \]
     if and only if $R\in R_{\mathcal J_{N_{\mathbb C}}}$.  
  
\end{prop}
\begin{proof}
    Combine the results of the previous two lemmas. 
\end{proof}

We describe the possible $\mathcal{J}_{N_{\mathbb{C}}}$ such that $\dot{\mathcal{H}}_{\mathcal{J}_{N_{\mathbb{C}}}}^K\neq \emptyset$ when $K$ is maximal.
\begin{coro}\label{maximaldimensioncasecomplex}
$\dot{\mathcal{H}}_{\mathcal{J}_{N_{\mathbb{C}}}}^{N_{\mathbb{C}}}\neq \emptyset$ if and only if   $N_{\mathcal{J}_{N_{\mathbb{C}}}}(m, (a,b))=N_{\mathcal{J}_{N_{\mathbb{C}}}}(m, (-a,b))$ for all  $(a,b),\; a,b>0,\; m\in \mathbb{N}^+$. 
    
\end{coro}
\begin{proof}
    From Proposition \ref{sufficientneccesaryconditionsexistencecomplexcase} we get 
    \begin{equation*}
   \sum_{i} d_M(\vec{L}_{(a_i,b_i)}, \vec{L}_{(-a_i,b_i)}) =0.        
    \end{equation*}

All the terms involved are nonnegative, so this gives the desired result. 
\end{proof}
Compare with the result for the real case in Corollary \ref{maximaldimensioncaseoddreal}.
\begin{example}\label{examplecomplex}
   Consider the case $N_{\mathbb{C}}=8$ and 
    \[
    \mathcal{J}_8= \mathcal{C}_4(0, \lambda), \;\; \lambda \neq 0.
    \]
    For this Jordan normal form we have $R_{\mathcal{J}_{N_{\mathbb{C}}}}=\lbrace 0,2,4,6,8 \rbrace$. So according to Proposition $\ref{sufficientneccesaryconditionsexistencecomplexcase}$, all these ranks are attainable. For example we have that 
\[
B_1=\begin{pmatrix}
1&0&0&1&0&0&0&0\\
0&1&-1&0&0&0&0&0\\
0&0&-1&0&0&-1&0&0\\
0&0&0&-1&1&0&0&0\\
0&0&0&0&1&0&0&1\\
0&0&0&0&0&1&-1&0\\
0&0&0&0&0&0&-1&0\\
0&0&0&0&0&0&0&-1
\end{pmatrix} \in \dot{\mathcal{H}}_{\mathcal{J}_{N_{\mathbb{C}}}}^{8}.
\]
If we \say{slide up} this matrix we obtain 
\[
B_2=\begin{pmatrix}
0&0&0&1&1&0&0&0\\
0&0&-1&0&0&1&0&0\\
0&0&0&0&0&-1&-1&0\\
0&0&0&0&1&0&0&-1\\
0&0&0&0&0&0&0&1\\
0&0&0&0&0&0&-1&0\\
0&0&0&0&0&0&0&0\\
0&0&0&0&0&0&0&0
\end{pmatrix} \in \dot{\mathcal{H}}_{\mathcal{J}_{N_{\mathbb{C}}}}^{6}.
\]
\end{example}
Now the second main result of this section: the equivalent of Proposition \ref{canonicalformsreal}. 
\begin{prop}\label{canonicalformscomplex}
Let $B=[B_{ij}] \in \dot{\mathcal{H}}_{\mathcal{J}_{N_{\mathbb{C}}}}^{N_{\mathbb{C}}}$. Then there exists $S=[S_{ij}] \in \dot{\mathcal{T}}_{\mathcal{J}_{N_{\mathbb{C}}}}$ and $P \in \mathrm{Per}(N_{\mathbb{C}},\mathbb{R})$  such that 
\[
S^{\boxstartwo}BS=
\mathcal{P}_{N_{\mathbb{C}},2}P^tJ_{N_{\mathbb{C}}}P.
\]
\end{prop}

\begin{proof}
The proof is analogous to the even-dimensional real case, using
Corollaries \ref{firstcorollarycomplexcase} and
\ref{secondlemmacomplexcasecorollary}. We apply these corollaries successively. If a diagonal block has
nonzero main coefficient, we apply Corollary
\ref{firstcorollarycomplexcase}. Otherwise, all the diagonal main
coefficients are zero, and we choose a block row of maximal size.
Since $B$ is nonsingular, this block row must contain a square
off-diagonal block $B_{ij}$ with nonzero main coefficient.
Otherwise, the shape of the rectangular blocks would give a zero
pair of scalar rows. Since $b_{ii}=b_{jj}=0$, Corollary
\ref{secondlemmacomplexcasecorollary} applies.

We repeat this argument on the block rows and columns that have not
yet been isolated. By construction, each successive congruence
transformation leaves the previously isolated components unchanged.
Taking the product of these transformations, we obtain
$S\in\dot{\mathcal T}_{\mathcal J_{N_{\mathbb C}}}$ such that, for
$B':=S^{\boxstartwo}BS$, every scalar row and column of
$\mathcal P_{N_{\mathbb C},2}B'$ contains exactly one nonzero entry.
Therefore,
\[
\left|\mathcal P_{N_{\mathbb C},2}B'\right|
\]
is a permutation matrix.

Also $\mathcal P_{N_{\mathbb C},2}B'$ is skew-symmetric.
Proposition \ref{skewsymmetricpermutation} therefore gives a
permutation matrix
$P\in\operatorname{Per}(N_{\mathbb C},\mathbb R)$ such that
\[
\mathcal P_{N_{\mathbb C},2}B'
=
P^tJ_{N_{\mathbb C}}P.
\]
Therefore,
\[
S^{\boxstartwo}BS
=
\mathcal P_{N_{\mathbb C},2}
P^tJ_{N_{\mathbb C}}P.
\]
\end{proof}

\section{Main Theorems}

We combine the results of previous sections and restate them in the context of Lie algebras. Recall that we are considering an almost abelian Lie algebra $\mathfrak{g}=\mathbb{R}e_1 \ltimes L$ of dimension $D$, defined by a matrix $\mathcal{J}_{N}$  in real Jordan normal form, where $N=D-1$. Moreover, we suppose that it is given as in Equation $(\ref{jornanrealpluscomplex})$. We also suppose that 
$\mathcal{J}_{N_{\mathbb{R}}}$ is as in Equation (\ref{jordannormalforminspecialorderreal})  and $\mathcal{J}_{N_{\mathbb{C}}}$ is as in Equation (\ref{jordannormalforminspecialordercomplex}).
\begin{theorem}\label{Main1}
    Let $\mathfrak{g}$  be an almost abelian Lie algebra as described at the beginning of this section. Such a Lie algebra admits a  presymplectic form of rank $R$ if and only if  
    \begin{equation}\label{maintheoremequation}
           R \in (R_{\mathcal{J}_{N_{\mathbb{R}}}} + R_{\mathcal{J}_{N_{\mathbb{C}}}}) \cup (R_{\mathcal{J}_{N_{\mathbb{R}}}} + R_{\mathcal{J}_{N_{\mathbb{C}}}} +2),
       \end{equation}    
\end{theorem}
$R_{\mathcal{J}_{N_{\mathbb{R}}}}$ and $R_{\mathcal{J}_{N_{\mathbb{C}}}}$ are as in Equation $(\ref{possibleranksrealsets})$ and Equation $(\ref{possiblerankscomplexsets})$, respectively. 
\begin{proof}
From Proposition \ref{conditionsforexistenceclosedform}, we know that for $R=D$, that is the symplectic case, the existence of a symplectic form is equivalent to the condition that $\mathcal{H}^{N-1}_{\mathcal{J}_{N}}\neq \emptyset$. For a presymplectic form with rank $R<D$ we know that the existence of such form is equivalent to the condition that $\mathcal{H}_{\mathcal{J}_N}^{R} \cup \mathcal{ H}_{\mathcal{J}_N}^{R-2} \neq \emptyset$. 
 
 Suppose first that such sets are non empty. From the discussion at the beginning of Section \ref{matrixproblemsection} there exist $Y, W$ satisfying $0 \leq Y\leq N_{\mathbb{R}}, 0 \leq W \leq N_{\mathbb{C}}$ such that  
\[
\mathcal{ H}_{\mathcal{J}_{N_{\mathbb{R}}}}^{Y} \oplus \mathcal{ H}_{\mathcal{J}_{N_{\mathbb{C}}}}^{W} \neq \emptyset,
\] 
with 
\begin{equation}\label{equationinsidemaintheorem}
Y+W=\begin{cases}
         R-2 & R=D,  \\
        R \; \text{or} \; R-2  & R<D.
     \end{cases}
\end{equation}

From Lemma $\ref{equivalentmatrixproblemreal}$ and Lemma $\ref{equivalentmatrixproblemcomplex}$,  we get that this condition is equivalent to
 \[
\dot{\mathcal{ H}}_{\mathcal{J}_{N_{\mathbb{R}}}}^{Y} \oplus \dot{\mathcal{H}}_{\mathcal{J}_{N_{\mathbb{C}}}}^{W} \neq \emptyset,
\] 
 with the same conditions for $Y$ and $W$.  From Propositions \ref{sufficientneccesaryconditionsexistence} and \ref{sufficientneccesaryconditionsexistencecomplexcase}, we know that
\begin{align*}
Y \in  R_{\mathcal{J}_{N_{\mathbb{R}}}},  \\
W \in  R_{\mathcal{J}_{N_{\mathbb{C}}}}.
\end{align*}
From this we get that 
\[
    Y+W \in R_{\mathcal{J}_{N_{\mathbb{R}}}} + R_{\mathcal{J}_{N_{\mathbb{C}}}}
    \]

To write this equation in terms of $R$ we use \eqref{equationinsidemaintheorem}, there are two cases, but if $R=D$, we get immediately, from the dimension of the sets, that $R\notin R_{\mathcal{J}_{N_{\mathbb{R}}}} + R_{\mathcal{J}_{N_{\mathbb{C}}}} $. Therefore, we can combine both cases and just state that 

\[
R \in (R_{\mathcal{J}_{N_{\mathbb{R}}}} + R_{\mathcal{J}_{N_{\mathbb{C}}}}) \cup (R_{\mathcal{J}_{N_{\mathbb{R}}}} + R_{\mathcal{J}_{N_{\mathbb{C}}}} +2)
\]

Now suppose that Equation $\ref{maintheoremequation}$ is true. If $R=D$ then there exist $Y \in R_{\mathcal{J}_{N_{\mathbb{R}}}} $ and $W \in R_{\mathcal{J}_{N_{\mathbb{C}}}} $ such that  $Y+W=R-2$. If $R<D$  then there exist $Y \in R_{\mathcal{J}_{N_{\mathbb{R}}}} $ and $W \in R_{\mathcal{J}_{N_{\mathbb{C}}}} $  such that either $Y+W=R-2$ or $Y+W=R$. Then by going back in the previous argument (as all the steps depend on if and only if statements) we get that $\mathcal{H}^{R-2}_{\mathcal{J}_{N}}\neq \emptyset$ in the former case and either $\mathcal{H}^{R-2}_{\mathcal{J}_{N}}\neq \emptyset$ or $\mathcal{H}^{R}_{\mathcal{J}_{N}}\neq \emptyset$ in the latter. In either case, Proposition
\ref{conditionsforexistenceclosedform} implies that $\mathfrak g$ admits a
presymplectic form of rank $R$.
\end{proof}

Depending on the dimension $D$ of the almost Abelian Lie algebra and the required rank $R$ of the 2-form, there are many possible different real Jordan normal forms that satisfy the conditions of previous theorem. We list the possible ones in the symplectic case.  This result can also be found in \cite{symplecticlassification}.

\begin{coro}
Let $\mathfrak{g}=\mathbb{R}e_1 \ltimes L$ of dimension $D$, $D$ even . There exists a symplectic form $\omega$ if and only if the following conditions hold
     \begin{enumerate}
     \item $N_{\mathcal{J}_{N_{\mathbb{C}}}}(m, (a,b))=N_{\mathcal{J}_{N_{\mathbb{C}}}}(m, (-a,b))$ for all  $(a,b),\; a,b>0,\; m\in \mathbb{N}^+$

     \item One of the following:
     \begin{enumerate}

         \item $N_{\mathcal{J}_{N_{\mathbb{R}}}}(m, \lambda)=N_{\mathcal{J}_{N_{\mathbb{R}}}}(m, -\lambda)$ for all $\lambda \in \mathbb{R}\setminus\lbrace 0 \rbrace$, $m\in \mathbb{N}^+$  and there exists only one odd $l\in \mathbb{N}^+$ such that $N_{\mathcal{J}_{N_{\mathbb{R}}}}(l,0)$ is odd. 
            \item $N_{\mathcal{J}_{N_{\mathbb{R}}}}(m,0)$ is even for all odd $m\in 
            \mathbb{N}^+$. One of the following:
            \begin{enumerate}
                
                \item There exists a unique $\alpha \in \mathbb R\setminus\{0\}$ such that 
                \[N_{\mathcal J_{N_{\mathbb R}}}(1,\alpha)- N_{\mathcal J_{N_{\mathbb R}}}(1,-\alpha)=1.\]
For every pair $(m,\lambda)\notin\{(1,\alpha),(1,-\alpha)\}$ with $\lambda\in\mathbb R\setminus\{0\}$, we have 
\[
N_{\mathcal J_{N_{\mathbb R}}}(m,\lambda)
=
N_{\mathcal J_{N_{\mathbb R}}}(m,-\lambda).
\]

                 \item  There exist a unique $l\in \mathbb{N}^+$ and a unique $\alpha \in \mathbb{R}\setminus\lbrace 0 \rbrace$  such that $N_{\mathcal{J}_{N_{\mathbb{R}}}}(l,\alpha)-N_{\mathcal{J}_{N_{\mathbb{R}}}}(l,-\alpha)=1$ and $N_{\mathcal{J}_{N_{\mathbb{R}}}}(l+1,-\alpha)-N_{\mathcal{J}_{N_{\mathbb{R}}}}(l+1,\alpha)=1$.  For every pair $(m,\lambda)$ different from $(l,\alpha)$, 
$(l+1,\alpha)$, $(l,-\alpha)$ and 
$(l+1,-\alpha)$ with $\lambda\in\mathbb R\setminus\{0\}$, we have $N_{\mathcal J_{N_{\mathbb R}}}(m,\lambda)
=
N_{\mathcal J_{N_{\mathbb R}}}(m,-\lambda).$   
                \end{enumerate} 
               
\end{enumerate}
\end{enumerate}
    \end{coro}

\begin{proof}
  By Theorem \ref{Main1} such symplectic $\omega$ exists if and only if 
    \[
  D \in R_{\mathcal{J}_{N_{\mathbb{R}}}} + R_{\mathcal{J}_{N_{\mathbb{C}}}} +2
    \]
  
 This is only possible if $\dot{\mathcal{H}}_{\mathcal{J}_{N_{\mathbb{R}}}}^{N_{\mathbb{R}}-1}\neq \emptyset$ and  $\dot{\mathcal{H}}_{\mathcal{J}_{N_{\mathbb{C}}}}^{N_{\mathbb{C}}}\neq \emptyset$. Then we just need to combine the  Corollaries \ref{maximaldimensioncaseoddreal} and \ref{maximaldimensioncasecomplex}.
\end{proof}
Next we have the second main result of this paper, a finite description of the moduli space of Symplectic forms. 
\begin{theorem}\label{main2}
    Let $\mathfrak{g}$ be an almost abelian Lie algebra as described at the beginning of this section. $\mathfrak{P}\Omega^2_{D,\,closed}(\mathfrak{g})$ is finite. Moreover, if $[\omega] \in \mathfrak{P}\Omega^2_{D,\,closed}(\mathfrak{g})$, then there exists a permutation $P\in \mathrm{Per}(D,\mathbb{R})_{e_1}$ such that $[\omega]= [P.\omega_0]$.   
       \end{theorem}
    \begin{proof}
    Proposition \ref{bijjectiontomatricesforsymplecticcase} shows that it is enough to study the quotient space 
    \[
    \mathcal{H}^{N-1}_{\mathcal{J}_N}/_{\text{cong}}\mathcal{T}_{\mathcal{J}_{N}}
    \]
     As discussed in Section \ref{matrixproblemsection} , such a matrix problem can be solved for the Real and Complex parts of the real Jordan normal form independently. Then Lemma \ref{equivalentmatrixproblemreal} and Lemma \ref{equivalentmatrixproblemcomplex} transform those quotient spaces into equivalent ones. Finally, Proposition \ref{canonicalformsreal} and \ref{canonicalformscomplex} give us the solution for the real and complex part, respectively. By joining these results and Corollary \ref{corollaryconditionsforpermutation} we complete the proof. 
    \end{proof}

    \begin{example}
      Let $\mathfrak{g}$ be an almost abelian algebra of dimension $14$, defined by the matrix given in real Jordan normal form $\mathcal{J}_{13}= \mathcal{J}_3(\lambda)\oplus \mathcal{J}_2(-\lambda) \oplus \mathcal{C}_4(0, \beta) , \;\; \lambda,\beta \neq 0$. Consider the matrix  

\begin{equation}
    M_{ij}:=\left( \begin{array}{c|c} 
0 & v^t \\ \hline
-v & \mathcal{P}_{5}A_i\oplus \mathcal{P}_{8,2} B_j
\end{array} \right),
\end{equation}
where $A_i,B_j$, $i,j \in \lbrace 1,2 \rbrace$ are as in Example \ref{examplereal} and Example \ref{examplecomplex}. With the identification in (\ref{identificationwithskewsymmetric}) we will not distinguish in the following table between the matrix and the corresponding $2$-form.  We have different results depending on the values of $i,j$ and depending on $v \in  \mathrm{Im}(\mathcal{P}_{5}A_i\oplus \mathcal{P}_{8,2} B_j)$  or  $v \notin  \mathrm{Im}(\mathcal{P}_{5}A_i\oplus \mathcal{P}_{8,2} B_j)$

\[
\begin{array}{c|c}
v\in\operatorname{Im}
\left(\mathcal P_5A_i\oplus\mathcal P_{8,2}B_j\right)
&
v\notin\operatorname{Im}
\left(\mathcal P_5A_i\oplus\mathcal P_{8,2}B_j\right)
\\ \hline
M_{11}\in\Omega^2_{12,\,\mathrm{closed}}(\mathfrak g)
&
M_{11}\in\Omega^2_{14,\,\mathrm{closed}}(\mathfrak g)
\\
M_{21}\in\Omega^2_{10,\,\mathrm{closed}}(\mathfrak g)
&
M_{21}\in\Omega^2_{12,\,\mathrm{closed}}(\mathfrak g)
\\
M_{12}\in\Omega^2_{10,\,\mathrm{closed}}(\mathfrak g)
&
M_{12}\in\Omega^2_{12,\,\mathrm{closed}}(\mathfrak g)
\\
M_{22}\in\Omega^2_{8,\,\mathrm{closed}}(\mathfrak g)
&
M_{22}\in\Omega^2_{10,\,\mathrm{closed}}(\mathfrak g)
\end{array}
\]
      
    \end{example}

\section*{Acknowledgment}   
The author would like to thank Hiroshima University, in particular Dr. Takayuki Okuda and Dr. Shoichi Fujimori, where part of this research was conducted. The author is also grateful for the support of the Osaka Central Advanced Mathematical Institute at Osaka Metropolitan University, where part of this research was also carried out. The author thanks Dr. Hiroshi Tamaru and the members of his seminar for their helpful comments and discussions, as well as for their encouragement to complete this work, without which this manuscript would certainly have taken longer to finish.

\bibliographystyle{amsplain}
\bibliography{references}
\end{document}